\documentclass[11pt]{article}
\usepackage[utf8]{inputenc}
\usepackage[english]{babel}
\usepackage{graphicx}
\usepackage{amsmath}
\usepackage{amssymb}
\usepackage{amsfonts}
\usepackage{amsthm}
\usepackage[left=2.1cm,top=1.5cm,right=2.1cm, bottom=2.2cm,letterpaper]{geometry}
\usepackage{mathrsfs}
\usepackage{caption}
\usepackage{subcaption}
\usepackage{latexsym}
\usepackage{xfrac}
\usepackage{tcolorbox}
\usepackage{enumitem}
\usepackage{float}
\usepackage{hyperref}
\hypersetup{
	colorlinks=true,
	linkcolor=red,
	filecolor=magenta,
	urlcolor=cyan,
}

\usepackage[all]{hypcap}
\usepackage{autonum}
\usepackage{url}
\usepackage[toc]{appendix}
\newtheorem{theorem}{Theorem}[section]

\newtheorem{lemma}[theorem]{Lemma}
\newtheorem{definition}[theorem]{Definition}
\newtheorem{proposition}[theorem]{Proposition}

\newtheorem{remark}[theorem]{Remark}
\numberwithin{equation}{section}
\numberwithin{figure}{section}

\newcommand{\N}{\mathbb{N}}
\newcommand{\R}{\mathbb{R}}
\newcommand{\eps}{\varepsilon}

\newcommand{\bs}{\boldsymbol}
\makeatletter
\renewcommand*\env@matrix[1][*\c@MaxMatrixCols c]{%
	\hskip -\arraycolsep
	\let\@ifnextchar\new@ifnextchar
	\array{#1}}
\makeatother

\def\neweq#1{\begin{equation}\label{#1}}
	\def\endeq{\end{equation}}

\newcommand{\wto}{\rightharpoonup}

\def\XXint#1#2#3{{\setbox0=\hbox{$#1{#2#3}{\int}$ }
		\vcenter{\hbox{$#2#3$ }}\kern-.6\wd0}}

\definecolor{darkblue}{rgb}{0,0,0.7} 
\definecolor{darkred}{rgb}{0.9,0.1,0.1}
\definecolor{darkgreen}{rgb}{0,0.5,0}

\definecolor{labelkey}{gray}{.8}
\definecolor{refkey}{gray}{.8}

\begin{document}
	
	\title{Effective Brinkman's term for the steady-state Navier-Stokes equations with prescribed flux or pressure drop in a junction of pipes}
	
	\author{Richard H\"ofer -- Amina Mecherbet -- Gianmarco Sperone}
	\date{}
	\maketitle
	
	\begin{abstract}
	\noindent
	The steady motion of a viscous incompressible fluid in a junction of pipes, containing several small particles of diameter $\eps^3$ and mutual distance $\eps$, is modeled through the Navier-Stokes equations with mixed boundary conditions. Apart from inhomogeneous Dirichlet boundary conditions on the particles, these involve the Bernoulli pressure and the tangential velocity on the outlets, while either the transversal flux rate or the pressure drop is prescribed along the domain. Applying the energy method in homogenization theory, we study the asymptotic behavior of the solutions to these systems as $\eps \to 0$, without any restriction on the magnitude of the data, and show that the effective equations display an additional Brinkman term. An important feature of the present work concerns the required uniform bounds, which are achieved (in the case of the prescribed flux problem) by a contradiction argument based on Bernoulli's law for solutions of the stationary Euler equations. 
		\par\noindent
		
	\smallskip
	\noindent	
	{\bf AMS Subject Classification:} 76M50, 76D05, 35B27, 35M32, 35Q31.\par\noindent
	{\bf Keywords:} incompressible fluids, mixed boundary conditions, homogenization, perforated domain.
	\end{abstract}
	
	{
		\hypersetup{linkcolor=black}
		\tableofcontents
	}
	
	\newpage
	\section{Introduction} \label{sec:intro}
	The study of the laminar motion of a viscous incompressible fluid in a junction of pipes $\Omega \subset \mathbb{R}^{3}$, whose bounding walls are rigid and impermeable, constitutes a fundamental theme in Fluid Mechanics \cite{landau} due to its applications in a variety of fields such as aerodynamics \cite{von2004aerodynamics}, hemodynamics \cite{galdi2008hemodynamical} and petroleum engineering \cite{bradley1987petroleum}. In all these disciplines, an essential investigation involves the fluid motion through an array of solid particles or perforations contained in the junction of pipes \cite{bear2013dynamics, hornung1996homogenization}. More precisely, for an arbitrarily small parameter $\eps > 0$, the fluid flow equations are set in a new domain $\Omega_{\eps} \subset \mathbb{R}^{3}$, obtained by placing inside $\Omega$ a large quantity $N(\eps)$ of particles, each one of diameter $\eps^\alpha$ ($\alpha \geq 1$), at a mutual distance $\eps$ from each other. The goal is then to derive an \textit{effective} description of the system dynamics when $\eps \to 0$; mathematically, this can be achieved within the context of \textit{homogenization theory} \cite{tartar2009general}.
	 
	In the seminal works of Marchenko \& Khruslov \cite[Chapter IV]{marchenko1974boundary}, Sánchez-Palencia \cite{ene1973equations,sanchez1980non}, Tartar \cite{tartar2006homogeneisation} and Allaire \cite{allaire1, allaire2}, three different regimes have been identified concerning the homogenization of the stationary Stokes or Navier-Stokes equations:
	
	    \noindent
		$\bullet$ In the \emph{subcritical regime} $\alpha > 3$, the particles have a vanishing effect as $\eps \to 0$, resulting in the original Stokes or Navier-Stokes system in $\Omega$.
		
		\noindent
		$\bullet$ In the \emph{critical} regime $\alpha = 3$, the particles give rise to an additional zero-order  term -- the \emph{Brinkman force} --  in the homogenized equation. This effective system is governed by the \textit{Brinkman equation} \cite{brinkman1949calculation}.
		
		\noindent
		$\bullet$ In the \emph{supercritical regime} $\alpha < 3$, the Brinkman force dominates the viscous forces (and inertial forces in the case of the Navier-Stokes system), resulting in \textit{Darcy's law} \cite{darcy1856fontaines} as the effective limit equation.   

	Since then, homogenization methods in the context of Fluid Mechanics have attracted the interest of several authors. This includes, for instance, the study of convergence rates \cite{mecherbet2020estimates,jing2025unified,shen2022sharp},  the evolutionary incompressible and compressible Navier-Stokes equations \cite{feireisl2016homogenization, hofer2023homogenization, hofer2024quantitative,masmoudi2002homogenization}, the study of non-zero boundary data \cite{CarrapatosoHillairet20,DesvillettesGolseRicci08,Hillairet2018,hillairet2019effect}, relaxation of the separation assumptions between the particles towards random configurations \cite{CarrapatosoHillairet20,Giunti21,GiuntiHoefer21,Hillairet2018,HoeferJansen20}, among various others.
	
	Traditionally, the homogenization result for the stationary Navier-Stokes equations is obtained as a by-product of the corresponding result for the Stokes problem via a compact perturbation argument \cite{allaire1, allaire2, conca1985application}.  This approach is unfeasible in fluid flow models where a-priori bounds cannot be derived by simply testing the momentum equation by the solution itself, or else, without imposing a \textit{smallness} condition on the data of the underlying system. One example is given by the celebrated \textit{Leray flux problem} for the stationary Navier-Stokes equations in multiply-connected domains, under non-homogeneous Dirichlet boundary conditions \cite{korobkov2015solution}. A similar situation arises in the study of pipe flows with mixed boundary conditions, where the pressure drop ought to be calculated from a given velocity net flux \cite{heywood1996artificial, korobkov2020solvability}. In fact, the unrestricted solvability of these problems, at the $\eps$-level, is guaranteed by the Leray-Schauder Principle, reaching the required a-priori bounds through a \textit{reductio ad absurdum} argument that employs Bernoulli's law \cite{korobkov2011bernoulli} for solutions of the stationary Euler equations. Application of homogenization techniques in such models becomes a delicate matter because, if one intends to deploy this contradiction argument for the obtainment of $\eps$-uniform bounds, the complex structure of the perforations must be taken into account. Indeed, this approach was implemented in \cite{patriarca2025homogenization} for the homogenization of Leray's flux problem in two dimensions, as well as in \cite{sperone2023homogenization} for the homogenization of the prescribed net flux/prescribed pressure drop problem in a perforated pipe in the subcritical case $\alpha > 3$.
	
	The present paper investigates the homogenization of the steady-state, incompressible Navier-Stokes equations in obstructed pipes in a polydisperse regime with mixed boundary conditions involving the Bernoulli pressure and the tangential velocity on the outlets, while either the transversal flux rate or the pressure drop is prescribed along the pipes. Our main contribution is the extension of the results contained in \cite{sperone2023homogenization} to the critical case $\alpha = 3$. Moreover, we consider a junction of bounded pipes with cylindrical outlets, and we allow non-zero constant Dirichlet boundary conditions at individual particles contained inside the pipes. This feature is relevant for applications, and a first step towards a more complete description that includes the particle evolution. Unfortunately, the rigorous mathematical treatment of the asymptotic behavior of such fully-coupled, dynamical many-particle-fluid problems is extremely challenging and there are only a few works in related simplified settings \cite{bravin2024collective,HoeferSchubert25, HMS25}.
	
	The most significant results of this manuscript, Theorems~\ref{thm_hom} and \ref{thm_hom_pd}, dictate that the homogenization process furnishes Brinkman's equation in the limit, without any smallness assumption on the data (external force, flux rate or pressure drop). This is achieved by applying the energy method of Tartar \cite[Appendix]{sanchez1980non} which, broadly speaking, consists of three steps:
	\begin{itemize}[leftmargin=7mm]
		\item[(a)] Proof of $\eps$-uniform bounds for the solutions of the $\eps$-dependent Navier-Stokes model.
		\item[(b)] Modification (correction) of the test functions in the homogenized domain (without perforations), turning them into admissible test functions at the $\eps$-level.
		\item[(c)] Passage to the limit, as $\eps \to 0$, in the Navier-Stokes equations at the $\eps$-level.
	\end{itemize}
We emphasize once more that, in most related  homogenization results for the (in)-stationary Navier-Stokes equations in perforated domains, step (a) is just a consequence of a simple energy estimate (and a Poincar\'e inequality in the supercritical regime). In contrast, here we exploit the fact that we already anticipate the behavior of the $\eps$-dependent fluid velocity as $\eps \to 0$, and therefore incorporate this feature in the aforementioned contradiction argument. This is also in distinction to  \cite{sperone2023homogenization}, where the precise way in which the particles influence the fluid flow was irrelevant, since their effect is negligible in the limit as $\eps \to 0$, in the subcritical case $\alpha > 3$. More precisely, a  \emph{scalar capacitary corrector} was sufficient in  \cite{sperone2023homogenization} to complete all three steps (a)-(b)-(c), because the \textit{relative capacity} of the swarm of particles vanishes in the homogenization limit, whereas it just remains uniformly bounded when $\alpha=3$, see Lemma~\ref{refcapacitypro}. Therefore, in the current framework, this fundamental role is played by \textit{Stokes correctors} which, in turn, allow for the construction of a suitable \textit{restriction operator} in this polydisperse setting. 
		
Once step (a) is fulfilled, the same restriction operator allows us to complete phases (b)-(c). Contrary to the classical procedure due to Allaire \cite[Theorem 1.1.8]{allaire1}, we pass to the limit by applying the restriction operator to test functions rather than multiplying the momentum equation by an \textit{oscillating} test function. In this regard, here we proceed similarly as in \cite{GiuntiHoefer21}, but extending the method to a polydisperse setting with non-zero particle velocities, and including the convergence of the pressure term in the homogenization limit.
Moreover, in contrast to \cite[Lemma 2.5]{GiuntiHoefer21}, we build a genuine restriction operator, in the sense that it does not modify test functions that already vanish on the particles. 
		
Since the discoveries of Brinkman \cite{brinkman1949calculation}, it is well understood nowadays that the spatial distribution of the particles in the critical regime causes the appearance of an additional (friction) force in the fluid equations,
originated from the collective reaction of the particles to the drag forces exerted on them by the viscous incompressible flow. Since the critical regime is dilute, the total drag force on the fluid is effectively given by the sum of the individual drag forces of the particles, which are determined by Stokes' law: they are proportional to the resistance of the particle and the difference between the particle and the fluid velocity in the vicinity of the particle.  Mathematically, this phenomenon has been properly identified in the case of zero particle velocities \cite{allaire1, feireisl2016homogenization}, and in \cite{CarrapatosoHillairet20,DesvillettesGolseRicci08,Hillairet2018,hillairet2019effect} for non-zero particle velocities; accordingly, the Brinkman term \eqref{homog.eqpnf}-\eqref{homog.eqppd} produced by our homogenization process is the sum of a friction term $\mathcal G \bs u$ and a  source term $\bs{\mathcal{J}}$ that arises from the limiting particle flux.
 
This paper is organized as follows. In Section \ref{sec:setting}, we first define the precise geometric setting of the junction of pipes and the perforations, as well as the fluid equations under investigation. We then present the main results of our paper, and the necessary assumptions to reach them. We gather, in Section \ref{section_preliminary_results}, several preliminary lemmas involving classical tools such as the capacity potential, the divergence equation and the construction of a flux carrier adapted to a perforated domain with non-homogeneous Dirichlet boundary conditions. The Stokes correctors and the restriction operator are also assembled in this section, see Theorem \ref{pro:corrector}. Section \ref{epslevelsec} is devoted to the definition of the functional spaces, the weak formulation of the problems under study, together with the proofs of Theorems \ref{epslevel}-\ref{epslevelpd} concerning the obtainment of $\eps$-uniform bounds. We furnish the proof of our main result, Theorem \ref{thm_hom}, at the end of Section \ref{epslevelsec}. Finally, the Appendix \ref{appendix_whole} is devoted to a detailed description of three essential topics which, although they are recurrent in the homogenization literature, are included here for the sake of completeness: the Stokes problem in the exterior of a single particle, the Stokes resistance matrix and the limits of the capacity densities treated in this manuscript. 
	
	\section{Presentation of the problem and main results} 
	\label{sec:setting}
	
	\subsection{The junction of pipes}
	
	\begin{definition} \label{addomain3}
	Given an integer $M \geq 2$, an open, bounded and simply-connected set $\Omega \subset \mathbb{R}^{3}$ will be called \textbf{admissible} if $\partial \Omega$ is piecewise of class $\mathcal{C}^{2}$, $\partial \Omega$ consists of one connected component and $\Omega$ is the union of $M+1$ disjoint subsets as follows:
	\begin{itemize}
		\item  For any $k \in \{1, \ldots, M\}$, in some coordinate system, $\Omega_{k} \doteq \Theta_{k} \times (0,\ell_{k})$, for some $\ell_{k}>0$ and some open, bounded and planar domain $\Theta_{k} \subset \mathbb{R}^{2}$ having boundary of class $\mathcal{C}^{2}$. Therefore, $\Omega_{k}$ is a cylinder of length $\ell_{k}$ and fixed cross-section, given by the planar set $\Theta_{k}$;
		\item  $\Omega_{0} \doteq \Omega \setminus \left( \Omega_{1} \cup \Omega_{2} \cup \ldots \cup \Omega_{M} \right)$  (note that $\Omega_{0}$ is not open).
	\end{itemize}
	The boundary of $\Omega$ is decomposed as 
	$$ 
	\partial \Omega = \Gamma \cup \Gamma_{W }\doteq \left( \bigcup_{k=1}^{M} \Gamma_{k} \right) \cup \Gamma_{W} \, ,
	$$
	where
	\begin{equation}\label{boundaryomega3d}
 \Gamma_{k} \doteq \Theta_{k} \times \{\ell_{k}\} \quad \text{in the coordinate system defining $\Omega_{k}$, for $k \in \{1, \ldots, M\}$}  \,, 
	\end{equation}	
	and $\Gamma_{W} \subset \mathbb{R}^{3}$ represents the $\mathcal{C}^{2}$-surface connecting $\Gamma_{i}$ with $\Gamma_{j}$, for every $i,j \in \{1, \ldots, M\}$.
	\end{definition}
	\noindent	
	Let $\Omega \subset \mathbb{R}^{3}$ be an admissible domain in the sense of Definition \ref{addomain3}. The outward unit normal to $\partial \Omega$ is denoted by $\bs{\nu} \in \mathbb{R}^{3}$. Henceforth we will refer to the components of $\Gamma$ in \eqref{boundaryomega3d} as the \textit{outlets} of $\Omega$, while $\Gamma_{W}$ includes all the \textit{physical walls} of $\Omega$.
	
	In few words, an admissible domain $\Omega \subset \mathbb{R}^{3}$, in the sense of Definition \ref{addomain3}, is a truncation (orthogonal to the symmetry axis of the cylindrical outlets) of the set considered in the celebrated \textbf{Leray problem}, as described in \cite[Definition 1.1]{amick1977steady} or \cite[Chapter III]{pileckas2007navier}. Figure \ref{dom3} illustrates an example of such admissible junction of pipes.
	\vspace*{-2mm}
	\begin{figure}[H]
		\begin{center}
			\includegraphics[scale=0.97]{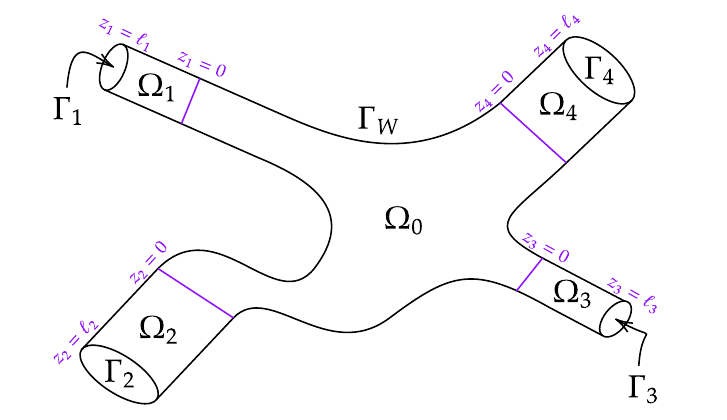}
		\end{center}
		\vspace*{-6mm}
		\caption{Representation of an admissible domain $\Omega \subset \mathbb{R}^{3}$.}\label{dom3}
	\end{figure}
	
	The reason behind the choice of domains with cylindrical ends is motivated by the existence of \textit{fully developed flows} in these outlets, such as the \textit{Hagen-Poiseuille flow} \cite[Chapter II]{landau}, characterized by the fact that the only nonzero component of the associated velocity field is directed along the axis of the tubes, see formula \eqref{poi23d} below. This allows us to construct, in Lemma \ref{fluxcarrier}, sufficiently smooth \emph{flux carriers}, i.e., vector fields satisfying the boundary conditions of the prescribed net flux problem.

	\subsection{The perforated domain} \label{perfdomassumptions}
	
	For $\bs{\xi} \in \mathbb{R}^{3}$ and $r > 0$ we denote by $B(\bs{\xi}, r) \subset \mathbb{R}^{3}$ the open ball of radius $r$ with center at $\bs{\xi}$. Let $( K_{n} )_{n \in \mathbb{N}}$ be a sequence of open, bounded and simply connected $\mathcal{C}^{2}$-sets, with connected boundaries, such that 
\begin{equation} \label{perforationinitial}
	\bs{0} \in K_{n} \quad \text{and} \quad \overline{K_{n}} \subset B(\bs{0},\delta_0) \qquad \forall n \in \mathbb{N} \, ,
\end{equation}
for some $\delta_0 > 0$ independent of $ n \in \mathbb{N}$. Moreover, we assume that the domains
\begin{equation} \label{perforationinitialjohn2}
B(0,\delta_0) \setminus \overline{K_{n}} \qquad \forall n \in \mathbb{N} \, ,
\end{equation}
have a uniform John constant \cite[Definition 2.2.1]{acosta2017divergence}. Following \cite{diening2017inverse,feireisl2021homogenization}, given $\varepsilon_{*} > 0$, we assume that for all $ \varepsilon \in (0,\varepsilon_{*})$ there exist an integer $N(\varepsilon) \geq 1$ and a collection of points $\bs{\xi}^{\varepsilon}_{1},...,\bs{\xi}^{\varepsilon}_{N(\varepsilon)} \in \mathbb{R}^{3}$ such that
	\begin{equation} \label{perforation}
		\begin{aligned}
			& \bs{\xi}^{\varepsilon}_{n} + \varepsilon^{3} \, \overline{K_{n}} \subset B(\bs{\xi}^{\varepsilon}_{n}, \delta_{0} \varepsilon^{3}) \subset B(\bs{\xi}^{\varepsilon}_{n}, \delta_{1} \varepsilon) \subset B \! \left(\bs{\xi}^{\varepsilon}_{n}, \delta_{2} \varepsilon \right) \subset \Omega \qquad \forall n \in \{1,...,N(\varepsilon)\} \,, \\[6pt]
			& \partial B \! \left(\bs{\xi}^{\varepsilon}_{n}, \delta_{2} \varepsilon \right) \cap \partial B \! \left(\bs{\xi}^{\varepsilon}_{m}, \delta_{2} \varepsilon \right) = \emptyset \qquad \forall n,m \in \{1,...,N(\varepsilon)\} \,, \ n \neq m \, , \\[6pt]
			& \partial B \! \left(\bs{\xi}^{\varepsilon}_{n}, \delta_{2} \varepsilon \right) \cap \partial \Omega = \emptyset \qquad \forall n \in \{1,...,N(\varepsilon)\} \, ,
		\end{aligned}
	\end{equation}
	for some constants $\delta_{2}> \delta_{1} > 0$ that are independent of $\varepsilon \in (0,\varepsilon_{*})$. Setting $K^{\varepsilon}_{n} \doteq \bs{\xi}^{\varepsilon}_{n} + \varepsilon^{3} K_{n}$ for every $n \in \{1,...,N(\varepsilon)\}$, we refer to the family $\{ K^{\varepsilon}_{n} \}^{N(\varepsilon)}_{n=1}$ satisfying \eqref{perforation} as the \textit{particles}, while
	\begin{equation} \label{perfordomain}
		\Omega_{\varepsilon} \doteq \Omega \setminus \overline{K_{\varepsilon}} \doteq \Omega \setminus \bigcup^{N(\varepsilon)}_{n=1} \overline{K^{\varepsilon}_{n}} \, ,
	\end{equation}
	represents the \textit{perforated fluid domain} at the $\varepsilon$-level. We emphasize that, given $\varepsilon \in (0,\varepsilon_{*})$, the family of particles $\{ K^{\varepsilon}_{n} \}^{N(\varepsilon)}_{n=1}$ is built in such a way that the \textit{size} of each solid is proportional to $\varepsilon^{3}$, while the mutual distance between any two consecutive particles is proportional to $\varepsilon$. Moreover, since we only consider those particles that are \textit{strictly} contained in $\Omega$ (in the sense of \eqref{perforation}$_3$), the following bound on the number $N(\varepsilon)$ holds:
	\begin{equation} \label{count}
		N(\varepsilon) \leq \dfrac{3| \Omega |}{4 \pi \delta_{2}^{3} \varepsilon^{3}} \, .
	\end{equation}
	Notice, however, that the solids $\{ K^{\varepsilon}_{n} \}^{N(\varepsilon)}_{n=1}$ may have different shapes and that they are not necessarily periodically distributed in $\Omega$. 

	\begin{figure}[H]
		\begin{center}
			\includegraphics[scale=0.97]{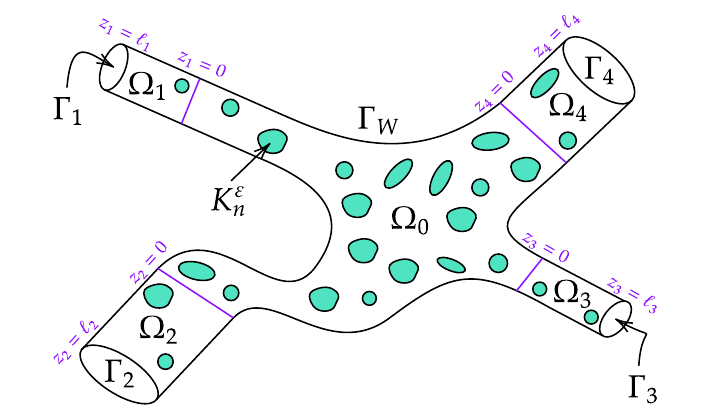}
		\end{center}
		\vspace*{-6mm}
		\caption{Representation of the perforated domain $\Omega_{\varepsilon}$.}\label{dom1}
	\end{figure}
	
	\subsection{The fluid flow equations}
	We decompose the boundary of $\Omega_{\varepsilon}$ as $ \partial \Omega_{\varepsilon} = \Gamma \cup \Gamma_{W} \cup \partial K_{\varepsilon}$, where
	\begin{equation}\label{boundaryomega1}
		\partial K_{\varepsilon} \doteq \bigcup^{N(\varepsilon)}_{n=1} \partial K^{\varepsilon}_{n} \, .
	\end{equation}
	The outward unit normal to $\partial \Omega_{\varepsilon}$ is still denoted by $\bs{\nu}$ (with some abuse of notation, as such vector also depends on $\varepsilon$). Given $\varepsilon \in (0,\varepsilon_{*})$ and a collection of constant vectors $\bs{\mu}^{\varepsilon}_{1},\ldots,\bs{\mu}^{\varepsilon}_{N(\varepsilon)} \in \mathbb{R}^{3}$,
	we analyze the steady motion of a viscous incompressible fluid (having unit kinematic viscosity) along $\Omega_{\varepsilon}$, which is characterized by its velocity vector field $\bs{u}_{\varepsilon} : \Omega_{\varepsilon} \longrightarrow \mathbb{R}^3$ and its scalar pressure $p_{\varepsilon} : \Omega_{\varepsilon} \longrightarrow \mathbb{R}$, under the action of an external force $\bs{f} : \Omega \longrightarrow \mathbb{R}^3$. Such stationary motion will be modeled through two different boundary-value problems (with mixed boundary conditions) associated to the steady-state Navier-Stokes equations in $\Omega_{\varepsilon}$. We firstly consider the \textbf{prescribed net flux problem}, which reads
	\begin{equation}\label{nsstokes0}
		\left\{
		\begin{aligned}
			& -\Delta \bs{u}_{\varepsilon} + (\bs{u}_{\varepsilon}\cdot\nabla)\bs{u}_{\varepsilon}+\nabla p_{\varepsilon}=\bs{f} \, , \quad  \nabla\cdot \bs{u}_{\varepsilon}=0 \ \ \mbox{in} \ \ \Omega_{\varepsilon} \, ; \\[4pt]
			& \bs{u}_{\varepsilon}=\bs{0} \ \ \mbox{on} \ \ \Gamma_{W} \, ; \\[4pt]
			& \bs{u}_{\varepsilon}=\bs{\mu}^{\varepsilon}_{n} \ \ \mbox{on} \ \ \partial K^{\varepsilon}_{n} \qquad \forall n \in \{1,...,N(\varepsilon)\} \, ; \\[4pt]
			& \bs{u}_{\varepsilon} \times \bs{\nu} = \bs{0} \, , \quad - \dfrac{\partial}{\partial \bs{\nu}}(\bs{u}_{\varepsilon} \cdot \bs{\nu}) + p_{\varepsilon} +  \dfrac{1}{2} |\bs{u}_{\varepsilon}|^{2} = p_{k}^{\eps} \ \ \mbox{ on } \ \ \Gamma_{k} \ \ \ \forall k \in \{1,\ldots,M\} \, ;\\[4pt]
			& \int_{\Gamma_{k}} \bs{u}_{\varepsilon} \cdot \bs{\nu} = \mathcal{F}_{k} \ \ \ \forall k \in \{1,\ldots,M\} \, .
		\end{aligned}
		\right.
	\end{equation}
	While \eqref{nsstokes0}$_{2}$ corresponds to the usual no-slip boundary condition on the physical walls $\Gamma_{W}$, the identity in \eqref{nsstokes0}$_{3}$ determines the motion of the particle $K^{\varepsilon}_{n}$ with constant velocity $\bs{\mu}^{\varepsilon}_{n}$, for every $n \in \{1,\ldots,N(\varepsilon)\}$. On the other hand, the first equality in \eqref{nsstokes0}$_{4}$ dictates that the fluid flow must enter or leave the domain $\Omega$ orthogonal to the outlets $\Gamma_{k}$, $k \in \{1,\ldots,M\}$. Introducing the \textit{Bernoulli pressure} as 
	$$
	\Phi_{\varepsilon} \doteq p_{\varepsilon} + \dfrac{1}{2} |\bs{u}_{\varepsilon}|^{2} \ \ \text{in} \ \ \Omega_{\varepsilon} \, ,
	$$ 
	the second identity in \eqref{nsstokes0}$_{4}$ imposes that, on the outlet $\Gamma_{k}$, the sum between the normal derivative of $\bs{u}_{\varepsilon} \cdot \bs{\nu}$ and the Bernoulli pressure must equal some constant $p_{k}^{\eps} \in \mathbb{R}$ representing the \textit{unknown pressure drop} along the cylindrical end $\Omega_{k}$, for $k \in \{1,\ldots,M\}$ (therefore, the constants $p_{k}^{\eps}$ are not prescribed, and they \textit{depend} on the solution). Finally, given a collection of numbers $(\mathcal{F}_{k})^{M}_{k=1} \subset \mathbb{R}$ satisfying the compatibility condition 
		\begin{equation} \label{compcon}
		\sum_{k=1}^{M} \mathcal{F}_{k} = 0 \, ,
	\end{equation}
(which is necessary for existence results, see \cite{korobkov2015solution})	the requirement \eqref{nsstokes0}$_{5}$ dictates that the flow rate of the velocity field, across the outlet $\Gamma_{k}$, is given by $\mathcal{F}_{k} \in \mathbb{R}$, for every $k \in \{1,\ldots,M\}$.
	
	Secondly, we analyze the \textbf{prescribed pressure drop problem}, which reads
	\begin{equation}\label{nsstokespd}
		\left\{
		\begin{aligned}
			& -\Delta \bs{u}_{\varepsilon} + (\bs{u}_{\varepsilon}\cdot\nabla)\bs{u}_{\varepsilon}+\nabla p_{\varepsilon}=\bs{f} \, , \quad  \nabla\cdot \bs{u}_{\varepsilon}=0 \ \ \mbox{in} \ \ \Omega_{\varepsilon} \, ; \\[4pt]
			& \bs{u}_{\varepsilon}=\bs{0} \ \ \mbox{on} \ \ \Gamma_{W} \, ; \\[4pt]
			& \bs{u}_{\varepsilon}=\bs{\mu}^{\varepsilon}_{n} \ \ \mbox{on} \ \ \partial K^{\varepsilon}_{n} \qquad \forall n \in \{1,\ldots,N(\varepsilon)\} \, ; \\[4pt]
			& \bs{u}_{\varepsilon} \times \bs{\nu} = \bs{0} \, , \quad - \dfrac{\partial}{\partial \bs{\nu}}(\bs{u}_{\varepsilon} \cdot \bs{\nu}) + p_{\varepsilon} +  \dfrac{1}{2} |\bs{u}_{\varepsilon}|^{2} = p_{k} \ \ \mbox{ on } \ \ \Gamma_{k} \ \ \ \forall k \in \{1,\ldots,M\} \, ,
		\end{aligned}
		\right.
	\end{equation}
	where now the constants $p_{k} \in \mathbb{R}$ are given and prescribe the pressure drop along the cylindrical end $\Omega_{k}$, for $k \in \{1,\ldots,M\}$.
	
	Since the scalar pressure can be determined up to an additive constant, without loss of generality we may take $p_{M}^{\eps}=0$ in \eqref{nsstokes0}$_{4}$ (as a consequence, we can no longer assume that the scalar pressure has zero mean value over $\Omega_{\varepsilon}$). Moreover, in view of the identity
	\begin{equation} \label{vectorcalculus}
		\nabla \left( \dfrac{1}{2} |\bs{u}_{\varepsilon}|^2 \right) = (\bs{u}_{\varepsilon}\cdot\nabla)\bs{u}_{\varepsilon} - (\nabla \times \bs{u}_{\varepsilon}) \times \bs{u}_{\varepsilon} \quad \text{in} \quad \Omega_{\varepsilon} \, ,
	\end{equation}
	it is customary (see, for example, \cite{heywood1996artificial,korobkov2020solvability}) to formulate systems \eqref{nsstokes0}-\eqref{nsstokespd} in terms of the Bernoulli pressure, thereby resulting in the following equivalent problems:
	\begin{equation}\label{nsstokes1}
		\left\{
		\begin{aligned}
			& -\Delta \bs{u}_{\varepsilon}+ (\nabla \times \bs{u}_{\varepsilon}) \times \bs{u}_{\varepsilon} + \nabla \Phi_{\varepsilon}=\bs{f} \, ,\ \quad  \nabla\cdot \bs{u}_{\varepsilon}=0 \ \ \mbox{in} \ \ \Omega_{\varepsilon} \, ; \\[4pt]
            & \bs{u}_{\varepsilon}=\bs{0} \ \ \mbox{on} \ \ \Gamma_{W} \, ; \qquad \bs{u}_{\varepsilon}=\bs{\mu}^{\varepsilon}_{n} \ \ \mbox{on} \ \ \partial K^{\varepsilon}_{n} \qquad \forall n \in \{1,\ldots,N(\varepsilon)\} \, ; \\[4pt]
            & \bs{u}_{\varepsilon} \times \bs{\nu} = \bs{0} \, , \quad - \dfrac{\partial}{\partial \bs{\nu}}(\bs{u}_{\varepsilon} \cdot \bs{\nu}) + \Phi_{\varepsilon} = p_{k}^{\eps} \ \ \mbox{ on } \ \ \Gamma_{k} \ \ \ \forall k \in \{1,\ldots,M-1\} \, ;\\[4pt]
            & \bs{u}_{\varepsilon} \times \bs{\nu} = \bs{0} \, , \quad - \dfrac{\partial}{\partial \bs{\nu}}(\bs{u}_{\varepsilon} \cdot \bs{\nu}) + \Phi_{\varepsilon} = 0 \ \ \mbox{ on } \ \ \Gamma_{M} \, ;\\[4pt]
            & \int_{\Gamma_{k}} \bs{u}_{\varepsilon} \cdot \bs{\nu} = \mathcal{F}_{k} \ \ \ \forall k \in \{1,\ldots,M\} \, .
		\end{aligned}
		\right.
	\end{equation}
	and
	\begin{equation}\label{nsstokespd1}
		\left\{
		\begin{aligned}
			& -\Delta \bs{u}_{\varepsilon}+ (\nabla \times \bs{u}_{\varepsilon}) \times \bs{u}_{\varepsilon} + \nabla \Phi_{\varepsilon}=\bs{f} \, ,\ \quad  \nabla\cdot \bs{u}_{\varepsilon}=0 \ \ \mbox{in} \ \ \Omega_{\varepsilon} \, ; \\[4pt]
			& \bs{u}_{\varepsilon}=\bs{0} \ \ \mbox{on} \ \ \Gamma_{W} \, ; \qquad \bs{u}_{\varepsilon}=\bs{\mu}^{\varepsilon}_{n} \ \ \mbox{on} \ \ \partial K^{\varepsilon}_{n} \qquad \forall n \in \{1,\ldots,N(\varepsilon)\} \, ; \\[4pt]
			& \bs{u}_{\varepsilon} \times \bs{\nu} = \bs{0} \, , \quad - \dfrac{\partial}{\partial \bs{\nu}}(\bs{u}_{\varepsilon} \cdot \bs{\nu}) + \Phi_{\varepsilon} = p_{k} \ \ \mbox{ on } \ \ \Gamma_{k} \ \ \ \forall k \in \{1,\ldots,M\} \, .
		\end{aligned}
		\right.
	\end{equation}
	
	\begin{remark} \label{firstidentity}
		Let $\Omega \subset \mathbb{R}^{3}$ be an admissible domain in the sense of Definition \ref{addomain3}. Given two vector fields $\bs{u}, \bs{\varphi} \in  W^{2,2}(\Omega_{\varepsilon}; \mathbb{R}^3)$ such that
		$$
		\bs{u} \times \bs{\nu} = \bs{\varphi} \times \bs{\nu} = \bs{0} \ \ \mbox{ on } \ \ \Gamma_{k} \ \ \ \forall k \in \{1,\ldots,M\} \, ,
		$$
		a straightforward computation shows 
		\begin{equation} \label{firstidentity0}
			\bs{\varphi} \cdot \dfrac{\partial \bs{u}}{\partial \bs{\nu}}  = (\bs{\varphi} \cdot \bs{\nu}) \dfrac{\partial}{\partial \bs{\nu}}(\bs{u} \cdot \bs{\nu})  \ \ \mbox{ on } \ \ \Gamma_{k} \ \ \ \forall k \in \{1,\ldots,M\} \, .
		\end{equation}
	\end{remark}

	\subsection{Assumptions and main results}
	
	The weak formulation of the prescribed net flux \eqref{nsstokes1} and pressure drop \eqref{nsstokespd1} problems can be found in Definitions \ref{weaksolutionpnf}-\ref{weaksolutionppd}, respectively.
	In Theorem~\ref{epslevel} we provide uniform bounds with respect to $\eps \in I_{*}$, with $I_{*} \doteq (0,\varepsilon_{*})$, for weak solutions of the prescribed net flux problem \eqref{nsstokes1} under two additional assumptions: 
	\begin{itemize}[leftmargin=4mm]
	\item[-] the constant vectors $\bs{\mu}^{\varepsilon}_{1},...,\bs{\mu}^{\varepsilon}_{N(\varepsilon)} \in \mathbb{R}^{3}$ satisfy the following uniform bound on the kinetic energy:
	\begin{equation} \label{kineticuni}
		\mathcal{S}_{*} \doteq \sqrt{ \sup_{\varepsilon \in (0,\varepsilon_{*})} \dfrac{1}{N(\varepsilon)} \sum_{n=1}^{N(\varepsilon)} \left| \bs{\mu}^{\varepsilon}_{n} \right|^{2} } < +\infty \, ,
	\end{equation}
	see also \cite[Assumption (A1)]{Hillairet2018} or \cite[Hypothesis (H3)]{mecherbet2020estimates}.
	\item[-] there are no particles in the vicinity of the outlets $\Gamma_{k}$, for $k \in \{1,\ldots,M\}$, see condition \eqref{extrahyp}. We refer to Remark \ref{rem1} for more details.
	\end{itemize}
	Accordingly, $\eps$-uniform bounds for the weak solutions of the prescribed pressure drop problem \eqref{nsstokespd1} are furnished, under the sole assumption \eqref{kineticuni}, in Theorem~\ref{epslevelpd}. 
	\eject
	These uniform bounds justify the main goal of the present article: the investigation of the asymptotic behavior of the solutions of problems \eqref{nsstokes1}-\eqref{nsstokespd1} as $\varepsilon \to 0^{+}$. The derivation of such effective equations is ensured provided that the \emph{Stokes capacity densities} admit a limit, which, in turn, enables to define the Brinkman force. For that, given $\varepsilon \in I_{*}$ and $n \in \{1,\ldots,N(\eps)\}$, denote by $\mathcal{G}_{n}\in \R^{3\times 3}$ the \textit{Stokes resistance matrix} associated to the Stokes problem in the exterior of $K_{n}$.  For $\bs{V} \in \R^3$, $ \mathcal{G}_{n} \bs{V}$ is the drag force exerted on the fluid by the particle $K_n$ when moving with velocity $\bs{V}$ relative to the fluid speed at infinity, in an otherwise unperturbed Stokes flow; we refer to \eqref{def_resistance_matrix} in Appendix \ref{appendix_sec} for the precise definition. Then, we are naturally led to the study of the limiting behavior of
	\begin{equation} \label{empmeasuresdef}
	\left\{
	\begin{aligned}
	& \mathcal{G}_{\eps}(M) \doteq \eps^3 \sum_{n=1}^{N(\eps)} \mathcal{G}_{n} \,  M(\bs{\xi}^{\eps}_{n}) \qquad \forall M \in W^{1,\infty}(\Omega; \R^{3\times 3}) \, ; \\[4pt]
	& \mathcal{J}_{\eps}(\bs{\varphi}) \doteq \eps^3 \sum_{n=1}^{N(\eps)} \mathcal{G}_{n} \bs{\mu}^{\varepsilon}_{n} \cdot \bs{\varphi}(\bs{\xi}^{\eps}_{n}) \qquad \forall \bs{\varphi} \in W^{1,\infty}(\Omega; \R^3) \, ,
	\end{aligned}
	\right.
	\end{equation}
	where the scaling factor $\eps^3$ arises from the fact that the Stokes capacity is linear in the diameter of the particles. We consequently assume the existence of $\mathcal{G} \in L^2(\Omega;\R^{3\times 3})$ and $\bs{\mathcal{J}} \in L^2(\Omega;\R^{ 3})$ such that 
	\begin{equation}\label{hyp_cv_measures}
	\left\{
\begin{aligned}
& \mathcal{G}_{\eps} \rightarrow \mathcal{G} \ \ \ \text{in the strong operator topology of} \ \mathcal{L}(W^{1,\infty}(\Omega; \R^{3\times 3}),\R^{3\times 3}) \ \ \text{as} \ \ \eps \to 0^{+} \, ; \\[4pt]
& \mathcal{J}_{\eps} \rightarrow \bs{\mathcal{J}} \ \ \ *-\text{weakly in} \ (W^{1,\infty}(\Omega; \R^3))^{*} \ \ \text{as} \ \ \eps \to 0^{+} \, ,
	\end{aligned}
	\right.
	\end{equation}
or equivalently, that
	\begin{equation}\label{hyp_cv_measures2}
	\left\{
	\begin{aligned}
		& \lim_{\varepsilon \to 0^{+}} \mathcal{G}_{\eps}(M) = \int_{\Omega} \mathcal{G} M  \qquad \forall M \in W^{1,\infty}(\Omega; \R^{3\times 3}) \, ; \\[6pt]
		& \lim_{\varepsilon \to 0^{+}} \mathcal{J}_{\eps}(\bs{\varphi} ) = \int_{\Omega} \bs{\mathcal{J}} \cdot \bs{\varphi}  \qquad \forall \bs{\varphi} \in W^{1,\infty}(\Omega; \R^{3}) \, .
	\end{aligned}
	\right.
\end{equation}
As we will see in Section \ref{empmeasures}, assumption \eqref{hyp_cv_measures} is always satisfied along a subsequence due to our hypothesis \eqref{perforation} on the separation of the particles.
	
For the prescribed net flux problem, our homogenization result reads:
	\begin{theorem}\label{thm_hom}
		Let $\Omega \subset \mathbb{R}^{3}$ be an admissible domain, $\bs{f} \in L^{2}(\Omega; \mathbb{R}^{3})$,  $(\bs{\mu}^{\varepsilon}_{n})^{N(\eps)}_{n=1} \subset \mathbb{R}^{3}$ and $(\mathcal{F}_{k})^{M}_{k=1} \subset \mathbb{R}$ satisfying \eqref{compcon}. Let $(\bs{u}_{\varepsilon}, \Phi_{\varepsilon}) \in W^{1,2}(\Omega_{\varepsilon}; \mathbb{R}^{3}) \times L^{2}(\Omega_{\varepsilon}; \mathbb{R})$ be a weak solution of the prescribed flux problem \eqref{nsstokes1}, stated in Theorem \ref{epslevel}. Define the extensions $\{(\widetilde{\bs{u}}_{\eps}, \widetilde{\Phi}_{\eps})\}_{\eps \in I_{*}} \subset W^{1,2}(\Omega; \mathbb{R}^{3}) \times L^{2}(\Omega; \mathbb{R})$ by
		\begin{align}
			\widetilde{\bs{u}}_{\eps} \doteq
			\begin{cases} 
				\bs u_\eps & \quad \text{in } \ \Omega_\eps\\[4pt]
				\bs{\mu}^{\varepsilon}_{n} & \quad \text{in } \ \overline{K_n^\eps} \quad \forall n \in \{1,\ldots,N(\eps)\} \, ,
			\end{cases} 
		\qquad \text{and} \qquad 
			\widetilde{\Phi}_{\eps} \doteq
		\begin{cases} 
			\Phi_{\eps} & \quad \text{in } \ \Omega_\eps\\[4pt]
			0 & \quad \text{in } \ \overline{K_\eps} \, .
		\end{cases} 
		\end{align}
		Suppose that $(\bs{u}, \Phi) \in W^{1,2}(\Omega; \mathbb{R}^{3}) \times L^{2}(\Omega; \mathbb{R})$ is a weak accumulation point of the sequence $\{(\widetilde{\bs{u}}_{\eps}, \widetilde{\Phi}_{\eps})\}_{\eps \in I_{*}}$ in $W^{1,2}(\Omega; \mathbb{R}^{3}) \times L^{2}(\Omega; \mathbb{R})$. Then, under assumptions \eqref{kineticuni}-\eqref{hyp_cv_measures}, $(\bs{u}, \Phi)$ is a weak solution to the following prescribed net flux problem in $\Omega$:
		\begin{equation}\label{homog.eqpnf}
			\left\{
			\begin{aligned}
				& -\Delta \bs{u}+ (\nabla \times \bs{u}) \times \bs{u} + \nabla \Phi + \mathcal{G} \bs{u} - \bs{\mathcal{J}} =\bs{f} \, ,\ \quad  \nabla\cdot \bs{u}=0 \ \ \mbox{in} \ \ \Omega \, ; \\[4pt]
				& \bs{u}=\bs{0} \ \ \mbox{on} \ \ \Gamma_{W} \, ; \\[4pt]
				& \bs{u} \times \bs{\nu} = \bs{0} \, , \quad - \dfrac{\partial}{\partial \bs{\nu}}(\bs{u} \cdot \bs{\nu}) + \Phi = p_{k} \ \ \mbox{ on } \ \ \Gamma_{k} \ \ \ \forall k \in \{1,\ldots,M-1\} \, ; \\[4pt]
				& \bs{u} \times \bs{\nu} = \bs{0} \, , \quad - \dfrac{\partial}{\partial \bs{\nu}}(\bs{u} \cdot \bs{\nu}) + \Phi = 0 \ \ \mbox{ on } \ \ \Gamma_{M} \, ; \\[4pt]
				& \int_{\Gamma_{k}} \bs{u}\cdot \bs{\nu} = \mathcal{F}_{k} \ \ \ \forall k \in \{1,\ldots,M\} \, ,
			\end{aligned}
			\right.
		\end{equation}
	for some constants $(p_{k})^{M-1}_{k=1} \subset \mathbb{R}$.
	\end{theorem}

Similarly, for the prescribed pressure drop problem, our homogenization result reads:

	\begin{theorem}\label{thm_hom_pd}
		Let $\Omega \subset \mathbb{R}^{3}$ be an admissible domain, $\bs{f} \in L^{2}(\Omega; \mathbb{R}^{3})$, $(\bs{\mu}^{\varepsilon}_{n})^{N(\eps)}_{n=1} \subset \mathbb{R}^{3}$ and $(p_{k})^{M}_{k=1} \subset \mathbb{R}$. Let $(\bs{u}_{\varepsilon}, \Phi_{\varepsilon}) \in W^{1,2}(\Omega_{\varepsilon}; \mathbb{R}^{3}) \times L^{2}(\Omega_{\varepsilon}; \mathbb{R})$ be a weak solution of the prescribed pressure drop problem \eqref{nsstokespd1}, stated in Theorem \ref{epslevelpd}.  Define the extensions $\{(\widetilde{\bs{u}}_{\eps}, \widetilde{\Phi}_{\eps})\}_{\eps \in I_{*}} \subset W^{1,2}(\Omega; \mathbb{R}^{3}) \times L^{2}(\Omega; \mathbb{R})$ by
		\begin{align}
			\widetilde{\bs{u}}_{\eps} \doteq
			\begin{cases} 
				\bs u_\eps & \quad \text{in } \ \Omega_\eps\\[4pt]
				\bs{\mu}^{\varepsilon}_{n} & \quad \text{in } \ \overline{K_n^\eps} \quad \forall n \in \{1,\ldots,N(\eps)\} \, ,
			\end{cases} 
			\qquad \text{and} \qquad 
			\widetilde{\Phi}_{\eps} \doteq
			\begin{cases} 
				\Phi_{\eps} & \quad \text{in } \ \Omega_\eps\\[4pt]
				0 & \quad \text{in } \ \overline{K_\eps} \, .
			\end{cases} 
		\end{align}
		Suppose that $(\bs{u}, \Phi) \in W^{1,2}(\Omega; \mathbb{R}^{3}) \times L^{2}(\Omega; \mathbb{R})$ is a weak accumulation point of the sequence $\{(\widetilde{\bs{u}}_{\eps}, \widetilde{\Phi}_{\eps})\}_{\eps \in I_{*}}$ in $W^{1,2}(\Omega; \mathbb{R}^{3}) \times L^{2}(\Omega; \mathbb{R})$. Then, under assumptions \eqref{kineticuni}-\eqref{hyp_cv_measures}, $(\bs{u}, \Phi)$ is a weak solution to the following prescribed pressure drop problem in $\Omega$:
		\begin{equation}\label{homog.eqppd}
			\left\{
			\begin{aligned}
				& -\Delta \bs{u}+ (\nabla \times \bs{u}) \times \bs{u} + \nabla \Phi + \mathcal{G} \bs{u} - \bs{\mathcal{J}} =\bs{f} \, ,\ \quad  \nabla\cdot \bs{u}=0 \ \ \mbox{in} \ \ \Omega \, ; \\[4pt]
				& \bs{u}=\bs{0} \ \ \mbox{on} \ \ \Gamma_{W} \, ; \\[4pt]
			& \bs{u} \times \bs{\nu} = \bs{0} \, , \quad - \dfrac{\partial}{\partial \bs{\nu}}(\bs{u} \cdot \bs{\nu}) + \Phi = p_{k} \ \ \mbox{ on } \ \ \Gamma_{k} \ \ \ \forall k \in \{1,\ldots,M\} \, .
			\end{aligned}
			\right.
		\end{equation}
	\end{theorem}
	
	\begin{remark}
 Theorem~\ref{thm_hom} dictates that any weak accumulation point of a sequence of solutions to the prescribed net flux problem \eqref{nsstokes1} obeys the Brinkman law \eqref{homog.eqpnf} under assumptions \eqref{kineticuni}-\eqref{hyp_cv_measures}. The existence of (at least) one such accumulation point follows directly from Theorem~\ref{epslevel}, supposing the further condition \eqref{extrahyp}. Similarly, Theorem~\ref{epslevelpd} ensures that the sequence of solutions to the prescribed pressure drop problem \eqref{nsstokespd1} admits a weak accumulation point. 
	\end{remark}

	\section{Preliminary results} \label{section_preliminary_results}
	\subsection{Capacity potential, divergence operator and construction of a flux carrier}\label{epslevelsecpre}
	Let $\Omega \subset \mathbb{R}^{3}$ be an admissible domain in the sense of Definition \ref{addomain3}, and let $\varepsilon \in I_{*}$ be a fixed parameter. Many of the results contained in the present article exploit the concept of \textit{relative capacity} of $K_{\varepsilon}$ with respect to $\Omega$, defined as 
	\begin{equation} \label{relcap}
		\mbox{Cap}_{\Omega}(K_{\varepsilon}) \doteq \min_{v\in W^{1,2}_{0}(\Omega;\mathbb{R})} \left\{  \int_{\Omega} |\nabla v|^2  \ \Big| \ v=1 \  \text{ in } \ \overline{K_{\varepsilon}} \, \right\} \, .
	\end{equation}
	The \textit{relative capacity potential} of $K_{\varepsilon}$ with respect to $\Omega$, that is, the scalar function $\phi_{\varepsilon} \in W_{0}^{1,2}(\Omega;\mathbb{R})$ achieving the minimum in \eqref{relcap}, satisfies
	\begin{equation} \label{cap1}
		\Delta \phi_{\varepsilon}=0 \ \ \text{in} \ \ \Omega_{\varepsilon} \, , \qquad \phi_{\varepsilon} = 0 \ \ \text{on} \ \ \partial \Omega\, , \qquad
		\phi_{\varepsilon} = 1 \ \ \text{in} \ \ \overline{K_{\varepsilon}} \, , \qquad\mbox{Cap}_{\Omega}(K_{\varepsilon})=\|\nabla \phi_{\varepsilon} \|^2_{L^{2}(\Omega)} \, ,
	\end{equation}
	see \cite[Chapter 2]{maz2013sobolev} for more details. Further essential properties of the relative capacity potential are collected in the following result, in the spirit of  \cite[Proposition~4.3]{allaire3} and the examples of \cite[Section 2]{cioranescu2018strange}:
	\begin{lemma} \label{refcapacitypro}
		Let $\Omega \subset \mathbb{R}^{3}$ be an admissible domain and $\phi_{\varepsilon} \in W_{0}^{1,2}(\Omega;\mathbb{R})$ be the function satisfying \eqref{cap1}. Then, $\phi_{\varepsilon} \in W^{2,2}(\Omega_{\varepsilon};\mathbb{R})$ and the following estimate holds:
		\begin{equation} \label{cap2}
			\|\phi_{\varepsilon} \|_{L^{\infty}(\Omega_{\varepsilon})} + \|\nabla \phi_{\varepsilon} \|_{L^{2}(\Omega_{\varepsilon})} \leq C_{*} \, ,
		\end{equation}
		for some constant $C_{*} > 0$ that depends on $\Omega$ and $\{ \delta_{i} \}^{2}_{i=0}$, but is independent of $\varepsilon \in I_{*}$.
	\end{lemma}
	\noindent
	\begin{proof}
		In what follows, $C > 0$ will always denote a generic constant that depends on $\Omega$ and $\{ \delta_{i} \}^{2}_{i=0}$ (independently of $\varepsilon \in I_{*}$), but that may change from line to line.
		
		Since $K^{\varepsilon}_{n}$ has a boundary of class $\mathcal{C}^{2}$ for every $n \in \{1,\ldots,N(\varepsilon)\}$, the domains $\Omega_{1},\ldots,\Omega_{M}$ are cylinders, and the lateral boundary $\Gamma_{W}$ is smooth, standard elliptic regularity arguments show that $\phi_{\varepsilon} \in W^{2,2}(\Omega_{\varepsilon}; \mathbb{R})$. The first estimate in \eqref{cap2} follows directly from the Maximum Principle. Concerning the second estimate in \eqref{cap2}, choose $\lambda > \delta_{0}$ (independent of $\varepsilon \in I_{*}$) such that
		$$
		\overline{K^{\varepsilon}_{n}} \subset B(\bs{\xi}^{\varepsilon}_{n}, \delta_{0} \varepsilon^{3}) \subset B(\bs{\xi}^{\varepsilon}_{n}, \lambda \varepsilon^{3}) \subset B(\bs{\xi}^{\varepsilon}_{n}, \delta_{1} \varepsilon)  \qquad \forall n \in \{1,\ldots,N(\varepsilon)\} \, .
		$$
		Given $n \in \{1,\ldots,N(\varepsilon)\}$, take a cutoff function $\varphi^{\varepsilon}_{n} \in \mathcal{C}_{0}^{\infty}(B(\bs{\xi}^{\varepsilon}_{n}, \lambda \varepsilon^{3});[0,1])$ such that
		\begin{equation} \label{forlater}
			\varphi^{\varepsilon}_{n} \equiv 1 \ \ \text{in} \ \ \overline{B(\bs{\xi}^{\varepsilon}_{n}, \delta_{0} \varepsilon^{3})} \, ; \qquad | \nabla \varphi^{\varepsilon}_{n}(\bs{\xi}) | \leq \dfrac{C}{\varepsilon^{3}} \quad \text{and} \quad | \nabla^{2} \varphi^{\varepsilon}_{n}(\bs{\xi}) | \leq \dfrac{C}{\varepsilon^{6}} \qquad \forall \bs{\xi} \in B(\bs{\xi}^{\varepsilon}_{n}, \lambda \varepsilon^{3}) \, ,
		\end{equation}
		see \cite[Chapter I, Section 1.4]{hormander1998analysis}. Since the relative capacity is an outer measure and is increasing with respect to domain inclusion (see \cite[Section 2.2]{maz2013sobolev}), from \eqref{perforation}-\eqref{count}-\eqref{forlater} we get
		$$
		\| \nabla \phi_{\varepsilon} \|^{2}_{L^{2}(\Omega)} = \mbox{Cap}_{\Omega}(K_{\varepsilon}) \leq \sum_{n=1}^{N(\varepsilon)} \mbox{Cap}_{\Omega}(K^{\varepsilon}_{n}) \leq \sum_{n=1}^{N(\varepsilon)} \mbox{Cap}_{\Omega}(B(\bs{\xi}^{\varepsilon}_{n}, \delta_{0} \varepsilon^{3})) \leq \sum_{n=1}^{N(\varepsilon)} \left( \int_{\Omega} |\nabla \varphi^{\varepsilon}_{n}|^2 \right) \leq C \, .
		$$
	This concludes the proof.
	\end{proof}
	
	In the sequel, an important role will be played by a \textit{suitably small} divergence-free vector field having a localized support around each of the particles $K^{\varepsilon}_{n}$, for $n \in \{1,\ldots,N(\varepsilon)\}$. More precisely, following the ideas contained in \cite[Lemma~6]{hillairet2019effect} , we prove:
	
	\begin{lemma} \label{smallfield}
		Let $\Omega \subset \mathbb{R}^{3}$ be an admissible domain and $\bs{\mu}^{\varepsilon}_{1},...,\bs{\mu}^{\varepsilon}_{N(\varepsilon)} \in \mathbb{R}^{3}$. There exists a vector field $\bs{A}_{\varepsilon} \in \mathcal{C}^{\infty}(\overline{\Omega}; \mathbb{R}^{3})$ such that
		\begin{equation} \label{smallfieldprop}
			\nabla \cdot \bs{A}_{\varepsilon} =0 \ \ \mbox{in} \ \ \Omega \, ; \qquad \bs{A}_{\varepsilon} = \bs{0} \ \ \mbox{on} \ \ \partial \Omega \, ; \qquad \bs{A}_{\varepsilon} = \bs{\mu}^{\varepsilon}_{n} \ \ \mbox{in} \ \ \overline{K^{\varepsilon}_{n}} \qquad \forall n \in \{1,\ldots,N(\varepsilon)\} \, .
		\end{equation}
		Moreover, under assumption \eqref{kineticuni}, there hold the bounds
		\begin{equation} \label{smallfieldpropbounds}
			\|\nabla \bs{A}_{\varepsilon}  \|_{L^{2}(\Omega)} \leq C_{*} \mathcal{S}_{*} \qquad \text{and} \qquad \|\bs{A}_{\varepsilon}  \|_{L^{4}(\Omega)} \leq C_{*} \mathcal{S}_{*} \, \varepsilon^{3/4} \, ,
		\end{equation}
		for some constant $C_{*} > 0$ that depends on $\Omega$ and $\{ \delta_{i} \}^{2}_{i=0}$, but is independent of $\varepsilon \in I_{*}$.
	\end{lemma}
	\noindent
	\begin{proof}
		In what follows, $C > 0$ will always denote a generic constant that depends on $\Omega$ and $\{ \delta_{i} \}^{2}_{i=0}$ (independently of $\varepsilon \in I_{*}$), but that may change from line to line.
		
		Let $\varphi^{\varepsilon}_{n} \in \mathcal{C}_{0}^{\infty}(B(\bs{\xi}^{\varepsilon}_{n}, \lambda \varepsilon^{3});[0,1])$ be the cutoff function observing \eqref{forlater}, for any $n \in \{1,\ldots,N(\varepsilon)\}$. Define
		$$
		\bs{A}_{\varepsilon} (\bs{\xi}) \doteq \dfrac{1}{2} \sum_{n=1}^{N(\varepsilon)} \nabla \times ( \varphi^{\varepsilon}_{n}(\bs{\xi}) (\bs{\mu}^{\varepsilon}_{n} \times (\bs{\xi} - \bs{\xi}^{\varepsilon}_{n}))) \qquad \forall \bs{\xi} \in \overline{\Omega} \, .
		$$
		Clearly, $\bs{A}_{\varepsilon} \in \mathcal{C}^{\infty}(\overline{\Omega}; \mathbb{R}^{3})$ is a divergence-free vector field such that, due to \eqref{perforation}-\eqref{forlater}, vanishes on $\partial \Omega$. Given any $j \in \{1,\ldots,N(\varepsilon)\}$ and $\bs{\xi} \in B(\bs{\xi}^{\varepsilon}_{j}, \delta_{0} \varepsilon^{3})$, from \eqref{forlater}$_1$ we infer
		$$
		\bs{A}_{\varepsilon} (\bs{\xi}) = \dfrac{1}{2} \nabla \times ( \bs{\mu}^{\varepsilon}_{j} \times (\bs{\xi} - \bs{\xi}^{\varepsilon}_{j})) = \bs{\mu}^{\varepsilon}_{j} \qquad \forall \bs{\xi} \in B(\bs{\xi}^{\varepsilon}_{j}, \delta_{0} \varepsilon^{3}) \, , 
		$$
		and therefore, $\bs{A}_{\varepsilon} = \bs{\mu}^{\varepsilon}_{j}$ in $\overline{K^{\varepsilon}_{j}}$, for every $j \in \{1,\ldots,N(\varepsilon)\}$. Concerning the estimates in \eqref{smallfieldpropbounds}, recall that $\varphi^{\varepsilon}_{j} \in \mathcal{C}_{0}^{\infty}(B(\bs{\xi}^{\varepsilon}_{j}, \lambda \varepsilon^{3});[0,1])$ for every $j \in \{1,\ldots,N(\varepsilon)\}$, and thus 
		\begin{equation} \label{smallfieldprop1}
			\int_{\Omega} | \nabla \bs{A}_{\varepsilon} |^{2} = \sum_{j=1}^{N(\varepsilon)} \left( \int_{B(\bs{\xi}^{\varepsilon}_{j}, \lambda \varepsilon^{3})} | \nabla \bs{A}_{\varepsilon} |^{2} \right) \, , \qquad \int_{\Omega} | \bs{A}_{\varepsilon} |^{4} = \sum_{j=1}^{N(\varepsilon)} \left( \int_{B(\bs{\xi}^{\varepsilon}_{j}, \lambda \varepsilon^{3})} | \bs{A}_{\varepsilon} |^{4} \right) \, .
		\end{equation}
		Given any $j \in \{1,\ldots,N(\varepsilon)\}$, we have
		\begin{equation} \label{smallfieldprop2}
			\bs{A}_{\varepsilon} (\bs{\xi}) = \dfrac{1}{2} \nabla \times ( \varphi^{\varepsilon}_{j}(\bs{\xi}) (\bs{\mu}^{\varepsilon}_{j} \times (\bs{\xi} - \bs{\xi}^{\varepsilon}_{j}))) = \varphi^{\varepsilon}_{j}(\bs{\xi}) \bs{\mu}^{\varepsilon}_{j} + \dfrac{1}{2} \nabla \varphi^{\varepsilon}_{j}(\bs{\xi}) \times ( \bs{\mu}^{\varepsilon}_{j} \times (\bs{\xi} - \bs{\xi}^{\varepsilon}_{j})) \qquad \forall \bs{\xi} \in B(\bs{\xi}^{\varepsilon}_{j}, \lambda \varepsilon^{3}) \, ,
		\end{equation}
		thereby implying, in virtue of \eqref{forlater}$_2$-\eqref{forlater}$_3$,
		\begin{equation} \label{smallfieldprop3}
			| \nabla \bs{A}_{\varepsilon}(\bs{\xi}) | \leq \dfrac{C}{\varepsilon^{3}} | \bs{\mu}^{\varepsilon}_{j} | \quad \text{and} \quad | \bs{A}_{\varepsilon}(\bs{\xi}) | \leq C | \bs{\mu}^{\varepsilon}_{j} | \qquad \forall \bs{\xi} \in B(\bs{\xi}^{\varepsilon}_{j}, \lambda \varepsilon^{3}) \, , \quad \forall j \in \{1,\ldots,N(\varepsilon)\} \, .
		\end{equation}
		Inserting \eqref{smallfieldprop3} into the right-hand side of both identities \eqref{smallfieldprop1}, and then enforcing \eqref{count}-\eqref{kineticuni}, entails
		$$
		\int_{\Omega} | \nabla \bs{A}_{\varepsilon} |^{2} \leq C \varepsilon^{3} \sum_{n=1}^{N(\varepsilon)} \left| \bs{\mu}^{\varepsilon}_{n} \right|^{2} \leq C \mathcal{S}_{*}^{2} \, ,
		$$
		as well as
		$$
		\int_{\Omega} | \bs{A}_{\varepsilon} |^{4} \leq C \varepsilon^{9} \sum_{n=1}^{N(\varepsilon)} \left| \bs{\mu}^{\varepsilon}_{n} \right|^{4} \leq C \varepsilon^{3} \left( \varepsilon^{3} \sum_{n=1}^{N(\varepsilon)} \left| \bs{\mu}^{\varepsilon}_{n} \right|^{2} \right)^{2} \leq C \mathcal{S}_{*}^{4} \varepsilon^{3} \, ,
		$$
		from where we directly obtain \eqref{smallfieldpropbounds}. This concludes the proof.
	\end{proof}
	
Let $\Omega \subset \mathbb{R}^{3}$ be an admissible domain. Given $k \in \{1,\ldots,M\}$, we define the horizontal cylinder $\mathcal{P}_{k} \doteq \Theta_{k} \times (0,\ell_{k})$, and denote by $\bs{\nu} \in \mathbb{R}^{3}$ its outward unit normal. Let $u_{k} \in \mathcal{C}^{2}(\overline{\Theta_{k}}; \mathbb{R})$ be the unique classical solution to the following torsion problem in $\Theta_{k}$:
	$$
	-\Delta u_{k} = 1 \ \ \mbox{in} \ \ \Theta_{k} \, , \qquad u_{k}=0 \ \ \mbox{on} \ \ \partial \Theta_{k} \, ,
	$$
	and put
	$$
	\varrho_{k} \doteq \int_{\Theta_{k}} u_{k} = \int_{\Theta_{k}} | \nabla u_{k} |^{2} > 0 \, .
	$$ 
	The steady-state Hagen-Poiseuille flow in $\mathcal{P}_{k}$, having unit kinematic viscosity and flow rate $F \in \mathbb{R}$, is characterized by the velocity field $\bs{U}^{(k)}_{\! \! F} : \overline{\mathcal{P}_{k}} \longrightarrow \mathbb{R}^3$ and scalar pressure $Q^{(k)}_{F} : \overline{\mathcal{P}_{k}} \longrightarrow \mathbb{R}$ given by
	\begin{equation}\label{poi23d}
		\bs{U}^{(k)}_{\! \! F}(x,y,z) \doteq \dfrac{F}{\varrho_{k}} (0,0, u_{k}(x,y) ) \qquad \text{and} \qquad Q^{(k)}_{F}(x,y,z) \doteq -\dfrac{F}{\varrho_{k}} \left( z - \ell_{k} \right) \qquad \forall (x,y,z) \in \overline{\mathcal{P}_{k}} \, .
	\end{equation}
	Now, given the geometry of $\Omega_{k}$, there exists a point $(x_{k}, y_{k},z_{k}) \in \mathbb{R}^3$ and a rotation matrix $\mathcal{Q}_{k} \in \text{SO}(3)$ such that the set $\overline{\Omega_{k}}$ can be mapped onto $\overline{\mathcal{P}_{k}}$ through the following rigid transformation: 
	$$
	\mathbf{T}_{k}(x,y,z) = (x_{k}, y_{k},z_{k}) + \mathcal{Q}_{k} \, (x,y,z) \qquad \forall (x,y,z) \in \overline{\Omega_{k}} \, .  
	$$
	Then, the steady-state Hagen-Poiseuille flow associated to $\Omega_{k}$, having unit kinematic viscosity and flow rate $F$, can be defined as
	\begin{equation}\label{poi23dd}
		\bs{V}^{(k)}_{\! \! F}(x,y,z) \doteq \mathcal{Q}_{k}^{\top} \bs{U}^{(k)}_{\! \! F}(\mathbf{T}_{k}(x,y,z)) \quad \text{and} \quad
		P_{F}^{(k)}(x,y,z) \doteq Q^{(k)}_{\! F}(\mathbf{T}_{k}(x,y,z)) \qquad \forall (x,y,z) \in \overline{\Omega_{k}} \, ,
	\end{equation}
	so that $(\bs{V}^{(k)}_{\! \! F}, Q^{(k)}_{\! F}) \in \mathcal{C}^{2}(\overline{\Omega_{k}}; \mathbb{R}^{3}) \times \mathcal{C}^{1}(\overline{\Omega_{k}}; \mathbb{R})$ solve the following Stokes system in $\Omega_{k}$:
	\begin{equation}\label{flux03d}
		\left\{
		\begin{aligned}
			\begin{aligned}
				& -\Delta \bs{V}^{(k)}_{\! \! F} + \nabla P^{(k)}_{F}  = \bs{0} \, , \quad  \nabla\cdot \bs{V}^{(k)}_{\! F} = 0 \ \ \mbox{in} \ \ \Omega_{k} \, , \\[5pt]
				& \bs{V}^{(k)}_{\! \! F}=\bs{0} \ \ \mbox{on} \ \ \Gamma_{W} \cap \partial \Omega_{k} \, , \\[5pt]
				& \int_{\Sigma} \bs{V}^{(k)}_{\! \! F} \cdot \bs{\nu} = F \quad \text{for any inner cross-section $\Sigma \subset \overline{\Omega_{k}}$ of $\Omega_{k}$} \, .
			\end{aligned}
		\end{aligned}
		\right.
	\end{equation}
	
Let $D \subset \mathbb{R}^{3}$ be any bounded Lipschitz domain, and consider the space of square-integrable functions in $D$ having zero mean value:
\begin{equation} \label{l02}
	L^2_0(D; \mathbb{R}) \doteq \left\{ g \in L^2(D; \mathbb{R}) \ \Big | \ \int_D g = 0 \right\} \, .
\end{equation}
It is well-known (see \cite[Section III.3]{galdi2011introduction} and Lemma~\ref{bogtype} below) that this space plays a fundamental role in the inversion of the divergence operator on $L^2_0(D;\mathbb{R})$, and therefore, also for the obtainment of pressure estimates in the equations of hydrodynamics in general. On the other hand, the boundary conditions of the pressure drop problem \eqref{nsstokespd1}$_4$ forbid us from modifying the Bernoulli pressure by an additive constant. As a consequence, in contrast with the standard case of Dirichlet boundary conditions, we can no longer assume that the Bernoulli pressure is an element of  $L_{0}^{2}(\Omega_{\varepsilon};\mathbb{R})$, see \eqref{l02}.
Another essential preliminary result concerns, therefore, the construction of an $\varepsilon$-uniform inverse of the divergence operator on the space $L^{2}(\Omega_{\varepsilon}; \mathbb{R})$, which exploits the corresponding uniform operator on $L_{0}^{2}(\Omega_{\varepsilon}; \mathbb{R})$, built in \cite{diening2017inverse}, and the particular geometry of our setting. Inspired by \cite[Lemma~4.2]{korobkov2020solvability}, we prove:
\begin{lemma} \label{bogtype}
	Let $\Omega \subset \mathbb{R}^{3}$ be an admissible domain and $q \in L^{2}(\Omega_{\varepsilon}; \mathbb{R})$. There exists a vector field $\bs{J}_{\! \varepsilon} \in W^{1,2}(\Omega_{\varepsilon}; \mathbb{R}^3)$ such that
	\begin{equation} \label{vecje}
		\left\{
	\begin{aligned}
		& \nabla \cdot \bs{J}_{\! \varepsilon} = q \ \ \mbox{in} \ \ \Omega_{\varepsilon} \, ; \qquad \bs{J}_{\! \varepsilon} = \bs{0} \ \ \mbox{on} \ \ \Gamma_{W} \cup \partial K_{\varepsilon} \cup \Gamma_{k} \quad \forall k \in \{1, \ldots, M-1\} \, ; \\[5pt]
		& \bs{J}_{\! \varepsilon} \times \bs{\nu} = \bs{0} \ \ \mbox{on} \ \ \Gamma_{M} \, ; \qquad \| \nabla \bs{J}_{\! \varepsilon} \|_{L^{2}(\Omega_{\varepsilon})} \leq C_{*} \| q \|_{L^{2}(\Omega_{\varepsilon})} \, , 
	\end{aligned}
	\right.
	\end{equation}
	for some constant $C_{*} > 0$ that depends on $\Omega$ and $\{ \delta_{i} \}^{2}_{i=0}$, but is independent of $\varepsilon \in I_{*}$.
\end{lemma}
\noindent
\begin{proof}
	In what follows, $C > 0$ will always denote a generic constant that depends on $\Omega$ and $\{ \delta_{i} \}^{2}_{i=0}$ (independently of $\varepsilon \in I_{*}$), but that may change from line to line.
	
	Let $\bs{V}_{\! \! M} \in W^{2,2}(\Omega_{M};\mathbb{R}^3)$ be as in \eqref{poi23dd}, with $F=1$. Invoking \cite[Theorem IV.5.2]{boyer2012mathematical}, consider the unique weak solution $(\bs{Z}, \Pi) \in W^{1,2}(\Omega;\mathbb{R}^3) \times L_{0}^{2}(\Omega;\mathbb{R})$ to the following Stokes system in $\Omega$:
\begin{equation} \label{hagen00}
	\left\{
	\begin{aligned}
		& -\Delta \bs{Z} + \nabla \Pi = \bs{0} \, , \quad  \nabla\cdot \bs{Z} = \dfrac{1}{| \Omega |}  \ \ \mbox{in} \ \ \Omega \, , \\[4pt]
		& \bs{Z}=\bs{0} \ \ \mbox{on} \ \ \Gamma_{W} \cup \Gamma_{k} \quad \forall k \in \{1, \ldots, M-1\} \, ; \quad \bs{Z} = \bs{V}_{\! \! M} \ \ \mbox{on} \ \ \Gamma_{M} \, .
	\end{aligned}
	\right.
\end{equation}
Observe that $\bs{Z} \times \bs{\nu} = \bs{0}$ on $\Gamma_{M}$, and
\begin{equation} \label{hagen1}
	\int_{\Omega} \nabla \cdot \bs{Z}  = 1 \, . 
\end{equation}
Since the domains $\Omega_{1}, \ldots, \Omega_{M}$ are cylinders, and the lateral boundary $\Gamma_{W}$ is of class $\mathcal{C}^{2}$, by merging the well-known regularity results for the solutions of the steady-state Stokes equations under non-homogeneous Dirichlet boundary conditions (see \cite[Teorema, page 311]{cattabriga1961problema}) with a localization argument (as in \cite[Chapter VII, Section 3]{dautray1999mathematicalvol2}), we can establish that $(\bs{Z}, \Pi) \in W^{2,2}(\Omega;\mathbb{R}^3) \times W^{1,2}(\Omega;\mathbb{R})$ and
\begin{equation} \label{cattabriga0}
	\| \bs{Z} \|_{W^{2,2}(\Omega)} + \| \Pi \|_{W^{1,2}(\Omega)} \leq C \| \bs{V}_{\! \! M} \|_{W^{3/2,2}( \Gamma_{M})} \leq C \, .
\end{equation}
Subsequently, the Sobolev embedding $W^{2,2}(\Omega;\mathbb{R}^3) \subset \mathcal{C}(\overline{\Omega};\mathbb{R}^3)$ ensures that
\begin{equation} \label{cattabriga00}
	\| \bs{Z} \|_{L^{\infty}(\Omega)} \leq C \, .
\end{equation}

Now, as in the proof of Lemma \ref{refcapacitypro}, let $\phi_{\varepsilon} \in W^{2,2}(\Omega_{\varepsilon};  \mathbb{R}) \cap W_{0}^{1,2}(\Omega; \mathbb{R})$ be the relative capacity potential of $K_{\varepsilon}$ with respect to $\Omega$.
Subsequently, introduce the vector field
\begin{equation} \label{bogtypevec}
	\bs{Q}_{\varepsilon} \doteq \left( \int_{\Omega_{\varepsilon}} q \right) (1 - \phi_{\varepsilon}) \bs{Z} \quad \text{in} \ \ \Omega_{\varepsilon} \, ,
\end{equation}
which, owing to \eqref{cap1}-\eqref{hagen00}, is an element of $W^{1,2}(\Omega_{\varepsilon};\mathbb{R}^3)$ such that 
\begin{equation} \label{pointwise2pre}
	\bs{Q}_{\! \varepsilon}=\bs{0} \ \ \mbox{on} \ \ \Gamma_{W} \cup \partial K_{\varepsilon} \cup \Gamma_{k} \quad \forall k \in \{1, \ldots, M-1\} \, ; \qquad \bs{Q}_{\! \varepsilon} \times \bs{\nu} = \bs{0} \ \ \mbox{on} \ \ \Gamma_{M} \, .
\end{equation}
Moreover, from \eqref{cap2}-\eqref{cattabriga0}-\eqref{cattabriga00} we readily infer the estimate
\begin{equation} \label{pointwise22q}
	\| \nabla \bs{Q}_{\varepsilon} \|_{L^{2}(\Omega_{\varepsilon})} \leq \left| \int_{\Omega_{\varepsilon}} q \right| \left( \| \nabla \phi_{\varepsilon} \|_{L^{2}(\Omega_{\varepsilon})} \| \bs{Z} \|_{L^{\infty}(\Omega)} + \| 1 - \phi_{\varepsilon} \|_{L^{\infty}(\Omega_{\varepsilon})} \| \nabla \bs{Z} \|_{L^{2}(\Omega)} \right) \leq C \| q \|_{L^{2}(\Omega_{\varepsilon})} \, .
\end{equation}
On the other hand, from the Divergence Theorem and \eqref{hagen1}-\eqref{pointwise2pre} we obtain
$$
\begin{aligned}
	\int_{\Omega_{\varepsilon}} \nabla \cdot \bs{Q}_{\varepsilon} = \int_{\partial \Omega} \bs{Q}_{\varepsilon} \cdot \bs{\nu} = \left( \int_{\Omega_{\varepsilon}} q \right) \left( \int_{\partial \Omega} \bs{Z} \cdot \bs{\nu} \right) = \int_{\Omega_{\varepsilon}} q \, ,
\end{aligned}
$$
so that $q - \nabla \cdot \bs{Q}_{\varepsilon} \in L^{2}_{0}(\Omega_{\varepsilon};\mathbb{R})$, see \eqref{l02}.

Then, standard results such as \cite[Lemma~1]{pileckas1983spaces} (consult also \cite[Lemma~1.8.13]{korobkov2024steady}) ensure the existence of another vector field $\bs{Y}_{\! \! \varepsilon} \in W^{1,2}_{0}(\Omega_{\varepsilon};\mathbb{R}^3)$ such that
	\begin{equation} \label{bogtypevec2}
		\nabla \cdot \bs{Y}_{\! \! \varepsilon} = q - \nabla \cdot \bs{Q}_{\varepsilon} \ \ \text{in} \ \ \Omega_{\varepsilon} \qquad \text{and} \qquad  \| \nabla \bs{Y}_{\! \! \varepsilon} \|_{L^{2}(\Omega_{\varepsilon})} \leq C(\Omega_{\varepsilon}) \| q - \nabla \cdot \bs{Q}_{\varepsilon} \|_{L^{2}(\Omega_{\varepsilon})}  \, .
	\end{equation}
	From \cite[Theorem~2.3]{diening2017inverse} we know that $C(\Omega_{\varepsilon})$ (the continuity constant of the inverse of the divergence operator in $\Omega_{\varepsilon}$) admits the uniform bound 
	\begin{equation} \label{boguni0}
		C(\Omega_{\varepsilon}) \leq C \, .
	\end{equation}
	We emphasize that our assumption \eqref{perforationinitialjohn2} on the uniform John constant of the sets $(B(0,\delta_0) \setminus \overline{K_{n}})_{n \in \mathbb{N}}$ is essential for \eqref{boguni0} to hold, as it ensures the validity of \cite[Lemma 3.1]{diening2017inverse}, in light of \cite[Theorem 2.6]{acosta2017divergence}.
	
	We set $\bs{J}_{\! \varepsilon} \doteq \bs{Q}_{\varepsilon} + \bs{Y}_{\! \! \varepsilon}$ which, in view of \eqref{pointwise2pre}-\eqref{pointwise22q}-\eqref{bogtypevec2}-\eqref{boguni0}, is an element of $W^{1,2}(\Omega_{\varepsilon}; \mathbb{R}^{3})$ observing the properties stated in \eqref{vecje}.
\end{proof}
	
With respect to the boundary-value problem \eqref{nsstokes1}, as an extension of \cite[Theorem~3.1]{sperone2022}, in the next lemma we build a sufficiently smooth \emph{flux carrier}.
	
\begin{lemma} \label{fluxcarrier}
		Let $\Omega \subset \mathbb{R}^{3}$ be an admissible domain, $(\bs{\mu}^{\varepsilon})_{n=1}^{N(\varepsilon)} \subset \mathbb{R}^{3}$ and $(\mathcal{F}_{k})_{k=1}^{M} \subset \mathbb{R}$ satisfying the compatibility condition \eqref{compcon}. Then, there exists a vector field $\bs{\Psi}_{\! \varepsilon} \in W^{2,2}(\Omega_{\varepsilon}; \mathbb{R}^{3})$ such that
		\begin{equation} \label{vecpsi}
			\left\{
			\begin{aligned}
				& \nabla \cdot \bs{\Psi}_{\! \varepsilon} =0 \ \ \mbox{in} \ \ \Omega_{\varepsilon} \, ; \qquad \bs{\Psi}_{\! \varepsilon} = \bs{0} \ \ \mbox{on} \ \ \Gamma_{W} \, ; \qquad \bs{\Psi}_{\! \varepsilon} = \bs{\mu}^{\varepsilon}_{n} \ \ \mbox{on} \ \ \partial K^{\varepsilon}_{n} \quad \forall n \in \{1,\ldots,N(\varepsilon)\} \, ; \\[3pt]
				& \bs{\Psi}_{\! \varepsilon} \times \bs{\nu} = \bs{0} \ \ \mbox{on} \ \ \Gamma_{k} \quad \text{and} \quad \int_{\Gamma_{k}} \bs{\Psi}_{\! \varepsilon} \cdot \bs{\nu} = \mathcal{F}_{k} \qquad \forall k \in \{1,\ldots, M\} \, . 
			\end{aligned}
			\right.
		\end{equation}
		Moreover, under assumption \eqref{kineticuni}, there holds the bound
		\begin{equation} \label{vecpsibound}
			\| \nabla \bs{\Psi}_{\! \varepsilon} \|_{L^{2}(\Omega_{\varepsilon})} \leq C_{*} \left( \mathcal{S}_{*} + \max_{k \in \{1,\ldots,M\}} |\mathcal{F}_{k}| \right) \, ,
		\end{equation}
		for some constant $C_{*} > 0$ that depends on $\Omega$ and $\{ \delta_{i} \}^{2}_{i=0}$, but is independent of $\varepsilon \in I_{*}$.
	\end{lemma}
	\noindent
	\begin{proof}
		In what follows, $C > 0$ will always denote a generic constant that depends on $\Omega$ and $\{ \delta_{i} \}^{2}_{i=0}$ (independently of $\varepsilon \in I_{*}$), but that may change from line to line.
		
	For $k \in \{1, \ldots, M\}$, let $\bs{V}_{\! \! k} \in W^{2,2}(\Omega_{k};\mathbb{R}^3)$ be as in \eqref{poi23dd} with $F=\mathcal{F}_{k}$, so that 
	\begin{equation} \label{fluxcarrierj1}
		\bs{V}_{\! \! k}  \times \bs{\nu} = \bs{0} \ \ \mbox{ on } \ \ \Gamma_{k} \, ; \qquad \int_{\Gamma_{k}} \bs{V}_{\! \! k} \cdot \bs{\nu} = \mathcal{F}_{k} \, .
	\end{equation}
	In virtue of \cite[Theorem 2.3.2]{korobkov2024steady} and the compatibility condition \eqref{compcon}, consider the unique weak solution $(\bs{\Psi}, \Pi) \in W^{1,2}(\Omega;\mathbb{R}^3) \times L_{0}^{2}(\Omega;\mathbb{R})$ to the following Stokes system in $\Omega$:
	\begin{equation} \label{hagen00F}
		\left\{
		\begin{aligned}
			& -\Delta \bs{\Psi} + \nabla \Pi  = \bs{0} \, , \quad  \nabla\cdot \bs{\Psi} = 0 \ \ \mbox{in} \ \ \Omega \, , \\[6pt]
			& \bs{\Psi}=\bs{0} \ \ \mbox{on} \ \ \Gamma_{W} \, ; \quad \bs{\Psi} = \bs{V}_{\! \! k} \ \ \mbox{on} \ \ \Gamma_{k} \quad \forall k \in \{1, \ldots, M\} \, .
		\end{aligned}
		\right.
	\end{equation}
	Therefore, $\bs{\Psi}$ is divergence-free in $\Omega$, it vanishes on $\Gamma_{W}$ and, due to \eqref{fluxcarrierj1}-\eqref{hagen00F}$_2$, 
	\begin{equation} \label{hagen1F}
		\bs{\Psi} \times \bs{\nu} = \bs{0} \ \ \mbox{on} \ \ \Gamma_{k} \quad \text{and} \quad \int_{\Gamma_{k}} \bs{\Psi} \cdot \bs{\nu} = \mathcal{F}_{k} \quad \forall k \in \{1, \ldots, M\} \, .
	\end{equation}
	As in the proof of Lemma \ref{bogtype}, we deduce that $(\bs{\Psi}, \Pi) \in W^{2,2}(\Omega;\mathbb{R}^3) \times W^{1,2}(\Omega;\mathbb{R})$ by elliptic regularity, together with the estimate
	\begin{equation} \label{hagen1FF}
		\| \bs{\Psi} \|_{W^{2,2}(\Omega)} + \| \Pi \|_{W^{1,2}(\Omega)} \leq C \sum_{k=1}^{M} \| \bs{V}_{\! \! k} \|_{W^{3/2,2}( \Gamma_{k})}\leq C \max_{k \in \{1,\ldots,M\}} |\mathcal{F}_{k}| \, ,
	\end{equation}
	where the last inequality in \eqref{hagen1FF} follows directly by the explicit formula \eqref{poi23dd}.
		
		On the other hand, as in the proof of Lemma~\ref{refcapacitypro}, let $\phi_{\varepsilon} \in W^{2,2}(\Omega_{\varepsilon};  \mathbb{R}) \cap W_{0}^{1,2}(\Omega; \mathbb{R})$ be the relative capacity potential of $K_{\varepsilon}$ with respect to $\Omega$, and $\varphi^{\varepsilon}_{n} \in \mathcal{C}_{0}^{\infty}(B(\bs{\xi}^{\varepsilon}_{n}, \lambda \varepsilon^{3});[0,1])$ the cutoff function observing \eqref{forlater}, for every $n \in \{1,...,N(\varepsilon)\}$. Then, introduce the vector field
		\begin{equation} \label{bogtypevecF}
			\bs{X}_{\! \varepsilon} \doteq (1 - \phi_{\varepsilon}) \bs{\Psi} +  \sum_{n=1}^{N(\varepsilon)} \varphi^{\varepsilon}_{n} \, \bs{\mu}^{\varepsilon}_{n}  \quad \text{in} \ \ \Omega_{\varepsilon} \, ,
		\end{equation}
		which, due to \eqref{cap1}-\eqref{forlater}-\eqref{hagen00F}-\eqref{hagen1F}, is an element of $W^{2,2}(\Omega_{\varepsilon};\mathbb{R}^3)$ such that 
		\begin{equation} \label{pointwise2preF}
			\left\{
			\begin{aligned}
				& \bs{X}_{\! \varepsilon}=\bs{0} \ \ \mbox{on} \ \ \Gamma_{W} \, ; \qquad \bs{X}_{\! \varepsilon}=\bs{\mu}^{\varepsilon}_{n} \ \ \mbox{on} \ \ \partial K^{\varepsilon}_{n} \qquad \forall n \in \{1,...,N(\varepsilon)\} \, ; \\[6pt]
				& \bs{X}_{\! \varepsilon} \times \bs{\nu} = \bs{0} \ \ \mbox{on} \ \ \Gamma_{k} \quad \text{and} \quad \int_{\Gamma_{k}} \bs{X}_{\! \varepsilon} \cdot \bs{\nu} = \mathcal{F}_{k} \qquad \forall k \in \{1, \ldots, M\} \, .
			\end{aligned}
			\right.
		\end{equation}
		Owing to \eqref{compcon}-\eqref{pointwise2preF} and the Divergence Theorem, the restriction of $\bs{X}_{\! \varepsilon}$ to $\partial \Omega_{\varepsilon}$ is an element of $W^{3/2,2}(\partial \Omega_{\varepsilon};\mathbb{R}^3)$ with zero total flux across $\partial \Omega_{\varepsilon}$, that is, 
		\begin{equation} \label{compstokesF}
			\int_{\partial \Omega_{\varepsilon}} \bs{X}_{\! \varepsilon} \cdot \bs{\nu} = \sum_{n=1}^{N(\varepsilon)} \left( \int_{\partial K^{\varepsilon}_{n}} \bs{X}_{\! \varepsilon} \cdot \bs{\nu} \right) + \sum_{k=1}^{M} \left( \int_{\Gamma_{k}} \bs{X}_{\! \varepsilon} \cdot \bs{\nu} \right) = \sum_{k=1}^{M} \mathcal{F}_{k} = 0 \, .
		\end{equation}
		In view of the zero flux condition \eqref{compstokesF}, \cite[Theorem~IV.1.1]{galdi2011introduction} ensures the existence of a unique weak solution $(\bs{\Psi}_{\! \varepsilon}, \Pi_{\varepsilon}) \in W^{1,2}(\Omega_{\varepsilon};\mathbb{R}^3) \times L_{0}^{2}( \Omega_{\varepsilon};\mathbb{R})$ to the following Stokes problem in $\Omega_{\varepsilon}$:
		\begin{equation} \label{stokespsi}
			\left\{
			\begin{aligned}
				& -\Delta \bs{\Psi}_{\! \varepsilon} + \nabla \Pi_{\varepsilon} = \bs{0} \, , \quad  \nabla\cdot \bs{\Psi}_{\! \varepsilon} = 0 \ \ \mbox{in} \ \ \Omega_{\varepsilon} \, , \\[4pt]
				& \bs{\Psi}_{\! \varepsilon} = \bs{X}_{\! \varepsilon} \ \ \mbox{on} \ \ \partial \Omega_{\varepsilon} \, ,
			\end{aligned}
			\right.
		\end{equation}
		that is,
		\begin{equation} \label{stokespsi22}
			\int_{\Omega_{\varepsilon}} \nabla \bs{\Psi}_{\! \varepsilon} : \nabla \bs{\varphi} - \int_{\Omega_{\varepsilon}} \Pi_{\varepsilon} (\nabla \cdot \bs{\varphi}) = 0 \qquad \forall \bs{\varphi} \in W_{0}^{1,2}(\Omega_{\varepsilon};\mathbb{R}^3) \, .
		\end{equation}
		Since $K^{\varepsilon}_{n}$ has a boundary of class $\mathcal{C}^{2}$ for every $n \in \{1,...,N(\varepsilon)\}$ and $\bs{X}_{\! \varepsilon} \in W^{3/2,2}(\partial \Omega_{\varepsilon};\mathbb{R}^3)$, we can invoke the previous elliptic regularity argument to infer that $(\bs{\Psi}_{\! \varepsilon}, \Pi_{\varepsilon}) \in W^{2,2}(\Omega_{\varepsilon};\mathbb{R}^3) \times W^{1,2}( \Omega_{\varepsilon};\mathbb{R})$. Therefore, in virtue of \eqref{pointwise2preF}-\eqref{stokespsi} the vector field $\bs{\Psi}_{\! \varepsilon} \in W^{2,2}(\Omega_{\varepsilon};\mathbb{R}^3)$ satisfies the properties stated in \eqref{vecpsi}. 
		
		In order to obtain the bound \eqref{vecpsibound}, define the vector fields $\bs{X}^{(1)}_{\! \varepsilon}, \bs{X}^{(2)}_{\! \varepsilon} \in W^{2,2}(\Omega_{\varepsilon};\mathbb{R}^3)$ by
		$$
		\bs{X}^{(1)}_{\! \varepsilon} \doteq (1 - \phi_{\varepsilon}) \bs{\Psi} \qquad \text{and} \qquad \bs{X}^{(2)}_{\! \varepsilon} \doteq \sum_{n=1}^{N(\varepsilon)} \varphi^{\varepsilon}_{n} \, \bs{\mu}^{\varepsilon}_{n} \quad \text{in} \ \ \Omega_{\varepsilon} \, .
		$$
		From \eqref{cap2}-\eqref{hagen1FF} and the Sobolev embedding $W^{2,2}(\Omega;\mathbb{R}^3) \subset \mathcal{C}(\overline{\Omega};\mathbb{R}^3)$ we infer the estimate
		\begin{equation} \label{pointwise22}
			\| \nabla \bs{X}^{(1)}_{\! \varepsilon} \|_{L^{2}(\Omega_{\varepsilon})} \leq C \max_{k \in \{1,\ldots,M\}} |\mathcal{F}_{k}| \, .
		\end{equation}
		Instead, from \eqref{count}-\eqref{kineticuni}-\eqref{forlater} we deduce the bound
		\begin{equation} \label{pointwise33}
				\| \nabla \bs{X}^{(2)}_{\! \varepsilon} \|^{2}_{L^{2}(\Omega_{\varepsilon})} = \sum_{n=1}^{N(\varepsilon)} \left( \int_{\Omega} \left| \nabla \varphi^{\varepsilon}_{n} \otimes \bs{\mu}^{\varepsilon}_{n} \right|^{2}  \right)
				\leq \sum_{n=1}^{N(\varepsilon)} \left( \int_{\Omega} \left| \nabla \varphi^{\varepsilon}_{n} \right|^{2}  \right) \left| \bs{\mu}^{\varepsilon}_{n} \right|^{2} \leq C \varepsilon^{3} \sum_{n=1}^{N(\varepsilon)} \left| \bs{\mu}^{\varepsilon}_{n} \right|^{2} \leq C \mathcal{S}^{2}_{*} \, .
		\end{equation}
		Combining \eqref{pointwise22}-\eqref{pointwise33} gives us
		\begin{equation} \label{vecpsibound00}
			\| \nabla \bs{X}_{\! \varepsilon} \|_{L^{2}(\Omega_{\varepsilon})} \leq C \left( \mathcal{S}_{*} + \max_{k \in \{1,\ldots,M\}} |\mathcal{F}_{k}| \right) \, .
		\end{equation}
		Recall that $\Pi_{\varepsilon} \in L_{0}^{2}( \Omega_{\varepsilon};\mathbb{R})$, thereby implying the existence of a vector field $\bs{Y}_{\!\!  \varepsilon} \in W^{1,2}_{0}(\Omega_{\varepsilon};\mathbb{R}^3)$ such that
		\begin{equation} \label{bogtypevec2F}
			\nabla \cdot \bs{Y}_{\! \! \varepsilon} = \Pi_{\varepsilon} \ \ \text{in} \ \ \Omega_{\varepsilon} \qquad \text{and} \qquad  \| \nabla \bs{Y}_{\! \! \varepsilon} \|_{L^{2}(\Omega_{\varepsilon})} \leq C \| \Pi_{\varepsilon} \|_{L^{2}(\Omega_{\varepsilon})}  \, ,
		\end{equation} 
		see again \eqref{bogtypevec2}-\eqref{boguni0}. Taking $\bs{\varphi} = \bs{Y}_{\! \! \varepsilon}$ as a test function in \eqref{stokespsi22} entails, due to \eqref{bogtypevec2F} and to the H\"older inequality,
		\begin{equation} \label{preseps0}
			\| \Pi_{\varepsilon} \|_{L^{2}(\Omega_{\varepsilon})} \leq C \| \nabla \bs{\Psi}_{\! \varepsilon} \|_{L^{2}(\Omega_{\varepsilon})}  \, .
		\end{equation}
		In the same fashion, taking $\bs{\varphi} = \bs{\Psi}_{\! \varepsilon} - \bs{X}_{\! \varepsilon}$ as a test function in \eqref{stokespsi22} gives us, as a consequence of \eqref{preseps0} and the H\"older inequality,
		\begin{equation} \label{preseps1}
			\| \nabla \bs{\Psi}_{\! \varepsilon} \|_{L^{2}(\Omega_{\varepsilon})} \leq C \| \nabla \bs{X}_{\! \varepsilon} \|_{L^{2}(\Omega_{\varepsilon})}  \, .
		\end{equation}
		The estimate \eqref{vecpsibound} follows directly from \eqref{vecpsibound00}-\eqref{preseps1}. This concludes the proof.
	\end{proof}
	
\begin{remark}
In the proof of Lemma \ref{fluxcarrier} one might wonder why, instead of solving the Stokes system \eqref{stokespsi}, we do not employ an inverse divergence operator to annihilate the divergence of $\bs{X}_{\! \varepsilon}$ in $\Omega_{\eps}$. The reason is that, due to the corners of the pipes, such operator does not necessarily provide the needed $W^{2,2}(\Omega_{\eps})$-regularity. As far as we know, this has been proved only for $\mathcal{C}^{1,1}$-domains in \cite{farwig1994generalized}.
\end{remark}

\subsection{Stokes correctors and the restriction operator}
Let $\Omega \subset \mathbb{R}^{3}$ be an admissible domain. Given $\varepsilon \in I_{*}$ and $n \in \{1,\ldots,N(\varepsilon)\}$, we introduce the sets
	$$  
	C_{n}^{\eps} \doteq B(\bs \xi_n^\eps, \delta_1 \eps) \setminus \overline{K_{n}^{\eps}} \, , \qquad D_{n}^{\eps} \doteq B(\bs \xi_n^\eps, \delta_2 \eps) \setminus \overline{B(\bs \xi_n^\eps, \delta_1 \eps)} \, , \qquad \widetilde{\Omega}_{\varepsilon} \doteq \Omega_{\varepsilon} \setminus \bigcup^{N(\varepsilon)}_{n=1} \overline{B(\bs \xi_n^\eps, \delta_2 \eps)} \, ,
	$$
so that
$$
\Omega_{\varepsilon} = \widetilde{\Omega}_{\varepsilon} \cup \bigcup^{N(\varepsilon)}_{n=1} \left( C_{n}^{\eps} \cup \overline{D_{n}^{\eps}} \right) \, .
$$
For any vector field $\bs{\varphi} \in W^{1,2}(\Omega; \mathbb{R}^{3})$, let $(\bs{\omega}^{\varepsilon}_{n,1}, p^{\varepsilon}_{n,1}) \in D^{1,2}(\R^3 \setminus \overline{K_{n}^{\eps}}; \R^{3}) \times L^{2}(\R^3 \setminus \overline{K_{n}^{\eps}}; \R)$ be the unique generalized solution to the following Stokes problem in the exterior of $K_{n}^{\eps}$:
\begin{equation}\label{extstokes00sc}
	\left\{
	\begin{aligned}
		& -\Delta \bs{\omega}^{\varepsilon}_{n,1} + \nabla p^{\varepsilon}_{n,1}  = \bs{0} \, , \quad  \nabla\cdot \bs{\omega}^{\varepsilon}_{n,1}  = 0 \ \ \mbox{in} \ \ \R^3 \setminus \overline{K_{n}^{\eps}} \, , \\[5pt]
		& \bs{\omega}^{\varepsilon}_{n,1} = \bs{\varphi} \ \ \mbox{on} \ \ \partial K_{n}^{\eps}  \, , \\[5pt]
		& \bs{\omega}^{\varepsilon}_{n,1} \to \bs{0} \ \ \mbox{as} \ \ | \bs{\xi} | \to + \infty \, ,
	\end{aligned}
	\right.
\end{equation}
see the Appendix \ref{appendix_sec}. Here, for any open Lipschitz set $U \subset \R^3$,  $D^{1,2}(U;\R^3)$ represents the homogeneous Sobolev space
\begin{equation} \label{D^1,2}
D^{1,2}(U;\R^3) \doteq \{\bs{\varphi} \in L^{6}(U;\R^3) \ | \ \nabla \bs{\varphi} \in L^2(U;\R^{3\times 3}) \, \} \, .
\end{equation}
The equality
$$
\int\limits_{\partial B(\bs \xi_\eps^n,\delta_1 \eps)} \bs{\omega}^{\varepsilon}_{n,1} \cdot \bs{\nu} = \int\limits_{K^{\eps}_{n}} \nabla \cdot \bs{\varphi} \, ,
$$
enables us to define the unique weak solution $(\bs{\omega}^{\varepsilon}_{n,2}, p^{\varepsilon}_{n,2}) \in W^{1,2}(D_{n}^{\eps}; \R^{3}) \times L_{0}^{2}(D_{n}^{\eps}; \R)$ to the following Stokes problem in the annulus $D_{n}^{\eps}$:
\begin{equation}\label{extstokes00scann}
	\left\{
	\begin{aligned}
		& -\Delta \bs{\omega}^{\varepsilon}_{n,2} + \nabla p^{\varepsilon}_{n,2}  = \bs{0} \, , \quad  \nabla\cdot \bs{\omega}^{\varepsilon}_{n,2}  = -\dfrac{1}{| D_{n}^{\eps} |} \int_{K^{\eps}_{n}} \nabla \cdot \bs{\varphi} \ \ \mbox{in} \ \ D_{n}^{\eps} \, , \\[5pt]
		& \bs{\omega}^{\varepsilon}_{n,2} = \bs{\omega}^{\varepsilon}_{n,1} \ \ \mbox{on} \ \ \partial B(\bs \xi_n^\eps, \delta_1 \eps)  \, , \\[5pt]
		& \bs{\omega}^{\varepsilon}_{n,2} = \bs{0} \ \ \mbox{on} \ \ \partial B(\bs \xi_n^\eps, \delta_2 \eps)  \, .
	\end{aligned}
	\right.
\end{equation}
Henceforth, we will refer to the pair $(\bs{\omega}^{\varepsilon}_{n,1}, \bs{\omega}^{\varepsilon}_{n,2}) \in D^{1,2}(\R^3 \setminus \overline{K_{n}^{\eps}}; \R^{3}) \times W^{1,2}(D_{n}^{\eps}; \R^{3})$ as the \textit{Stokes correctors}. We then introduce the \textit{restriction operator} $\mathcal{R}_{\eps}(\bs{\varphi}) : W^{1,2}(\Omega; \R^{3}) \longrightarrow W^{1,2}(\Omega; \R^{3})$ as the following linear application:
	\begin{equation}\label{def_w_eps^k}
		\mathcal{R}_{\eps}(\bs{\varphi}) \doteq \left\{
		\begin{array}{lcl}
			\bs{0} & \text{ in } & \overline{K_{n}^{\eps}} \quad \forall n \in \{1,\ldots,N(\varepsilon)\} \, , \\[6pt]
			\bs \varphi- \bs \omega^{\eps}_{n,1}& \text{ in }& C^{\eps}_{n} \quad \forall n \in \{1,\ldots,N(\varepsilon)\} \, , \\[6pt]
			\bs \varphi-\bs \omega^{\eps}_{n,2}&\text{ in }& D^{\eps}_{n} \quad \forall n \in \{1,\ldots,N(\varepsilon)\} \, ,\\[6pt]
			\bs \varphi & \text{ in } & \widetilde{\Omega}_{\eps} \, .
		\end{array}
		\right.
	\end{equation}
The most relevant properties of the restriction operator \eqref{def_w_eps^k} are collected in the next result:
	\begin{theorem} \label{pro:corrector}
	Let $\Omega \subset \mathbb{R}^{3}$ be an admissible domain. There exists $\varepsilon_{\star} \in I_{*}$ such that, for any $\varepsilon \in (0,\varepsilon_{\star})$ and any $\bs \varphi \in W^{1,2}(\Omega; \R^{3}) \cap L^{\infty}(\Omega; \R^{3})$, the vector field $\mathcal{R}_{\eps}(\bs{\varphi}) \in W^{1,2}(\Omega; \R^{3})$, defined in \eqref{def_w_eps^k}, satisfies the following properties:
	\begin{equation} \label{propreps0}
		\left\{
		\begin{aligned}
		& \mathcal{R}_{\eps}(\bs{\varphi}) = \bs{0} \ \ \mbox{in} \ \ \overline{K_{\eps}} \, ; \quad \mathcal{R}_{\eps}(\bs{\varphi}) = \bs{\varphi} \ \ \text{if} \ \bs{\varphi} = \bs{0} \ \ \mbox{in} \ \ \overline{K_{\eps}}\, ; \\[6pt]
		& \nabla \cdot \mathcal{R}_{\eps}(\bs{\varphi}) = 0 \ \ \mbox{in} \ \ \Omega \, , \ \ \text{if} \ \bs{\varphi} \ \text{is divergence-free in } \Omega \, ; \\[6pt]
		& \|\nabla \mathcal{R}_{\eps}(\bs{\varphi})\|_{L^2(\Omega)} \leq C_{*} \left( \| \bs\varphi\|_{L^{\infty}(\Omega)} + \|\nabla \bs \varphi\|_{L^2(\Omega)} \right) \, ; \\[6pt]
		& \|\mathcal{R}_{\eps}(\bs{\varphi}) - \bs \varphi\|_{L^{4}(\Omega)} \leq C_{*} \eps^{1/4} \left( \| \bs\varphi\|_{L^{\infty}(\Omega)} + \|\nabla \bs \varphi\|_{L^2(\Omega)} \right) \, ;  
		\end{aligned}
		\right.
		\end{equation}
		for some constant $C_{*} > 0$ that depends on $\Omega$ and $\{ \delta_{i} \}^{2}_{i=0}$, but is independent of $\varepsilon \in (0,\varepsilon_{\star})$.
		\newline
		\noindent
		Moreover, consider $\bs{\mu}^{\varepsilon}_{1},...,\bs{\mu}^{\varepsilon}_{N(\varepsilon)} \in \mathbb{R}^{3}$ and a sequence $\{(\bs{v}_\eps, \Lambda_{\eps})\}_{\varepsilon \in I_{*}} \subset W^{1,2}(\Omega; \mathbb{R}^{3}) \times L^{2}(\Omega; \mathbb{R})$ such that 
			\begin{equation} \label{sequenceps}
			\left\{
			\begin{aligned}
				& \nabla \cdot \bs{v}_\eps = 0 \ \ \mbox{in} \ \ \Omega \, ; \quad  \bs{v}_\eps = \bs{\mu}^{\varepsilon}_{n} \ \ \text{and} \ \ \Lambda_{\eps} = 0 \ \ \mbox{in} \ \ \overline{K_{n}^{\eps}} \qquad \forall n \in \{1,\ldots,N(\varepsilon)\} \, ;\\[6pt] 
				& (\exists \bs{v} \in W^{1,2}(\Omega; \mathbb{R}^{3})) \ \bs{v}_\eps  \wto \bs{v} \ \  \text{weakly in} \ \ W^{1,2}(\Omega; \mathbb{R}^{3}) \ \ \text{as} \ \ \varepsilon \to 0^{+}  \, ;\\[6pt] 
				& (\exists \Lambda \in L^{2}(\Omega; \mathbb{R})) \ \Lambda_\eps  \wto \Lambda \ \  \text{weakly in} \ \ L^{2}(\Omega; \mathbb{R}) \ \ \text{as} \ \ \varepsilon \to 0^{+}  \, .
			\end{aligned}
			\right.
		\end{equation}
		Then, under assumptions \eqref{kineticuni}-\eqref{hyp_cv_measures}, for every vector field $\bs \varphi \in W^{1,\infty}(\Omega; \mathbb{R}^{3})$ there holds
		\begin{align}\label{W_eps_resistance}
			\lim_{\varepsilon \to 0^{+}} \left( \int_\Omega  \nabla \bs{v}_\eps : \nabla \mathcal{R}_{\eps}(\bs{\varphi}) - \int_\Omega  \Lambda_\eps (\nabla \cdot \mathcal{R}_{\eps}(\bs{\varphi})) \right)  = \int_\Omega \nabla \bs v : \nabla \bs \varphi + \int_{\Omega} (\mathcal{G} \bs{v} - \bs{\mathcal{J}}) \cdot \bs{\varphi} - \int_\Omega  \Lambda (\nabla \cdot \bs{\varphi}) \, .
		\end{align}
	\end{theorem}
\noindent
	\begin{proof}
	In what follows, $C > 0$ will always denote a generic constant that depends on $\Omega$ and $\{ \delta_{i} \}^{2}_{i=0}$ (independently of $\varepsilon \in I_{*}$), but that may change from line to line.
	
	Let $\varepsilon \in I_{*}$ and $\bs \varphi \in W^{1,2}(\Omega; \R^{3}) \cap L^{\infty}(\Omega; \R^{3})$.
	\newline
	\underline{Proof of \eqref{propreps0}$_1$-\eqref{propreps0}$_2$}.
	Properties \eqref{propreps0}$_1$-\eqref{propreps0}$_2$ readily follow from the definition \eqref{def_w_eps^k}. In fact, if $\bs{\varphi}=\bs{0}$ in $\overline{K_{\eps}}$, it clearly follows from \eqref{extstokes00sc}-\eqref{extstokes00scann} that $\bs{\omega}^{\eps}_{n,1} = \bs{\omega}^{\eps}_{n,2} \equiv \bs{0}$.
	\newline
	\underline{Proof of \eqref{propreps0}$_3$}.
		According to \eqref{def_w_eps^k}, we have
\begin{equation} \label{rext0}
	\bs{\varphi} - \mathcal{R}_{\eps}(\bs{\varphi}) = \left\{
	\begin{array}{lcl}
		\bs{\varphi} & \text{ in } & \overline{K_{n}^{\eps}} \quad \forall n \in \{1,\ldots,N(\varepsilon)\} \, , \\[6pt]
		\bs \omega^{\eps}_{n,1}& \text{ in }& C^{\eps}_{n} \quad \forall n \in \{1,\ldots,N(\varepsilon)\} \, , \\[6pt]
		\bs{\omega}^{\eps}_{n,2}&\text{ in }& D^{\eps}_{n} \quad \forall n \in \{1,\ldots,N(\varepsilon)\} \, ,\\[6pt]
		\bs{0} & \text{ in } & \widetilde{\Omega}_{\eps} \, .
	\end{array}
	\right.
\end{equation}
It follows, from \eqref{rext0}, 
	\begin{equation} \label{rext1}
		\|\nabla (\mathcal{R}_{\eps}(\bs{\varphi}) - \bs{\varphi})\|_{L^2(\Omega)}^{2} = \sum_{n=1}^{N(\eps)} \left( \|\nabla \bs{\omega}^{\eps}_{n,1} \|_{L^2(C^{\eps}_{n})}^2+\|\nabla \bs{\omega}^{\eps}_{n,2}\|_{L^2(D^{\eps}_{n})}^2+\|\nabla \bs \varphi\|_{L^2(K_{n}^{\eps})}^2 \right) \, .
		\end{equation}
We estimate, separately, each of the three addends appearing on the right-hand side of \eqref{rext1}.

Firstly, from the energy bound \eqref{energy_stokes_1} we readily derive
		\begin{equation}\label{rext2}
\|\nabla \bs{\omega}^{\eps}_{n,1} \|_{L^2(C^{\eps}_{n})}^2 \leq C \left( \eps^{3} \|\bs  \varphi\|^{2}_{L^{\infty}(\Omega)} + \|\nabla \bs \varphi\|^{2}_{L^2(B(\bs{\xi}_{n}^{\eps},\delta_0 \eps^3))} \right) \qquad \forall n \in \{1,\ldots,N(\varepsilon)\} \, ,
		\end{equation}
thereby yielding, recalling \eqref{count}, 
			\begin{equation}\label{rext4}
	\sum_{n=1}^{N(\eps)} \|\nabla \bs{\omega}^{\eps}_{n,1} \|_{L^2(C^{\eps}_{n})}^2 \leq C \left(  \|\bs  \varphi\|^{2}_{L^{\infty}(\Omega)} + \sum_{n=1}^{N(\eps)} \|\nabla \bs \varphi\|^{2}_{L^2(B(\bs{\xi}_{n}^{\eps},\delta_0 \eps^3))} \right) \leq C \left( \|\bs  \varphi\|^{2}_{L^{\infty}(\Omega)} + \|\nabla \bs \varphi\|^{2}_{L^2(\Omega)} \right) \, .
	\end{equation}	
	
Secondly, let us put $\lambda \doteq \frac{\delta_{2}}{\delta_{1}} > 1$,
$$
\bs{\omega}_{n,2} (\bs{\xi}) \doteq \bs{\omega}^{\eps}_{n,2} (\bs{\xi}^{\varepsilon}_{n} + \delta_{1} \varepsilon \bs{\xi}) \qquad \text{and} \qquad p_{n,2} (\bs{\xi}) \doteq \delta_{1} \varepsilon \, p^{\varepsilon}_{n,2}(\bs{\xi}^{\varepsilon}_{n} + \delta_{1} \varepsilon \bs{\xi}) \qquad \forall \bs{\xi} \in B( \bs{0}, \lambda) \setminus \overline{B(\bs{0}, 1)} \, ,
$$
so that, in view of \eqref{extstokes00scann}, $(\bs{\omega}_{n,2}, p_{n,2}) \in W^{1,2}(B( \bs{0}, \lambda) \setminus \overline{B(\bs{0}, 1)}; \R^{3}) \times L_{0}^{2}(B( \bs{0}, \lambda) \setminus \overline{B(\bs{0}, 1)}; \R)$ is the unique weak solution to the following Stokes problem in $B( \bs{0}, \lambda) \setminus \overline{B(\bs{0}, 1)}$:
\begin{equation}\label{annstokes}
	\left\{
	\begin{aligned}
		& -\Delta \bs{\omega}_{n,2} + \nabla p_{n,2}  = \bs{0} \, , \quad  \nabla\cdot \bs{\omega}_{n,2}  = -\dfrac{ \delta_{1} \varepsilon}{| D_{n}^{\eps} |} \int_{K^{\eps}_{n}} \nabla \cdot \bs{\varphi} \ \ \mbox{in} \ \ B( \bs{0}, \lambda) \setminus \overline{B(\bs{0}, 1)} \, , \\[5pt]
		& \bs{\omega}_{n,2} = \bs{\omega}_{n,1} \ \ \mbox{on} \ \ \partial B(\bs{0}, 1)  \, , \\[5pt]
		& \bs{\omega}_{n,2} = \bs{0} \ \ \mbox{on} \ \ \partial B( \bs{0}, \lambda)\, ,
	\end{aligned}
	\right.
\end{equation}
where $\bs{\omega}_{n,1} \in D^{1,2}( \mathbb{R}^{3}; \mathbb{R}^{3})$ is given by
$$
\bs{\omega}_{n,1} (\bs{\xi}) \doteq \bs{\omega}^{\eps}_{n,1} (\bs{\xi}^{\varepsilon}_{n} + \delta_{1} \varepsilon \bs{\xi}) \qquad \forall \bs{\xi} \in  \mathbb{R}^{3} \, .
$$
We appeal then to \cite[Theorem~IV.1.1]{galdi2011introduction} (see also \cite[Lemma~B.1]{GiuntiHoefer21}), to the Poincaré and trace inequalities (both applied in $B( \bs{0}, \lambda) \setminus \overline{B(\bs{0}, 1)}$) in order to derive the bound
\begin{equation}\label{rext5}
\|\nabla \bs{\omega}_{n,2} \|_{L^2(B( \bs{0}, \lambda) \setminus \overline{B(\bs{0}, 1)})}^2 \leq C \left( \|\bs{\omega}_{n,1} \|_{W^{1,2}(B( \bs{0}, \lambda) \setminus \overline{B(\bs{0}, 1)})}^2 + \eps^2 \|\nabla \bs \varphi\|^{2}_{L^2(K^{\eps}_{n})} \right) \, .
\end{equation}
Applying the change of variables
$$
\bs{\xi} \in B( \bs{0}, \lambda) \setminus \overline{B(\bs{0}, 1)} \ \ \longleftrightarrow \ \ \bs{\xi}^{\varepsilon}_{n} + \delta_{1} \varepsilon \bs{\xi} \in D^{\eps}_{n} \, ,
$$
we deduce, as in the proof of Lemma~\ref{lemma_Abis} (see \eqref{divext4}), 
\begin{equation} \label{rext6}
	\begin{aligned}
		& \|\nabla \bs{\omega}_{n,2} \|_{L^2(B( \bs{0}, \lambda) \setminus \overline{B(\bs{0}, 1)})}^2 = \dfrac{1}{\delta_{1} \eps}	\|\nabla  \bs{\omega}^{\eps}_{n,2} \|^{2}_{L^2(D^{\eps}_{n})} \, , \\[6pt]
		& \| \bs{\omega}_{n,1} \|_{L^{2}(B( \bs{0}, \lambda) \setminus \overline{B(\bs{0}, 1)}))}^{2} = \dfrac{1}{\delta_{1}^{3} \eps^{3}} \|\bs{\omega}^{\eps}_{n,1} \|_{L^{2}(D^{\eps}_{n})}^{2} \, , \qquad \|\nabla \bs{\omega}_{n,1} \|_{L^2(B( \bs{0}, \lambda) \setminus \overline{B(\bs{0}, 1)})}^2 = \dfrac{1}{\delta_{1} \eps}	\|\nabla  \bs{\omega}^{\eps}_{n,1} \|^{2}_{L^2(D^{\eps}_{n})}  \, .
	\end{aligned}
\end{equation}
Upon insertion of \eqref{rext6} into \eqref{rext5} we get
\begin{equation}\label{rext7}
\|\nabla  \bs{\omega}^{\eps}_{n,2} \|^{2}_{L^2(D^{\eps}_{n})} \leq C \left( \dfrac{1}{\eps^{2}} \|\bs{\omega}^{\eps}_{n,1} \|_{L^{2}(D^{\eps}_{n})}^{2} + \|\nabla  \bs{\omega}^{\eps}_{n,1} \|^{2}_{L^2(D^{\eps}_{n})} + \eps^2 \|\nabla \bs \varphi\|^{2}_{L^2(K^{\eps}_{n})}  \right) \, .
\end{equation}
Then, select a sufficiently small $\varepsilon_{\star} \in I_{*}$ such that $D^{\eps}_{n} \subset \mathbb{R}^{3} \setminus \overline{B(\bs{\xi}_{n}^{\eps},2\delta_0 \eps^3)}$ for every $\eps \in (0, \varepsilon_{\star})$. The decay estimate \eqref{decay_stokes_1} can be invoked to infer that
\begin{equation}
\left\{
\begin{aligned}
& \|\bs{\omega}^{\eps}_{n,1} \|_{L^{2}(D_{n}^{\eps})}^{2} \leq C \eps^{4} \left( \eps^{3} \|\bs  \varphi\|^{2}_{L^{\infty}(\Omega)} + \|\nabla \bs \varphi\|^{2}_{L^2(B(\bs{\xi}_{n}^{\eps},\delta_0 \eps^3))} \right) \, , \\[6pt]
& \|\nabla \bs{\omega}^{\eps}_{n,1} \|_{L^{2}(D_{n}^{\eps})}^{2} \leq C \eps^{2} \left( \eps^{3} \|\bs  \varphi\|^{2}_{L^{\infty}(\Omega)} + \|\nabla \bs \varphi\|^{2}_{L^2(B(\bs{\xi}_{n}^{\eps},\delta_0 \eps^3))} \right) \, ,
\end{aligned}
\right.
\end{equation}
which, once inserted into the right-hand side of \eqref{rext7}, furnishes
\begin{equation}\label{rext8}
	\begin{aligned}
		\|\nabla  \bs{\omega}^{\eps}_{n,2} \|^{2}_{L^2(D^{\eps}_{n})} \leq C \eps^{2} \left( \eps^{3} \|\bs  \varphi\|^{2}_{L^{\infty}(\Omega)} + \|\nabla \bs \varphi\|^{2}_{L^2(B(\bs{\xi}_{n}^{\eps},\delta_0 \eps^3))} \right) \qquad \forall n \in \{1,\ldots,N(\varepsilon)\} \, .
	\end{aligned}
\end{equation}
Exactly as in \eqref{rext4}, from \eqref{rext8} we then derive
\begin{equation}\label{rext9}
	\sum_{n=1}^{N(\eps)} \|\nabla \bs{\omega}^{\eps}_{n,2} \|_{L^2(D^{\eps}_{n})}^2 \leq C \eps^{2} \left( \|\bs  \varphi\|^{2}_{L^{\infty}(\Omega)} + \|\nabla \bs \varphi\|^{2}_{L^2(\Omega)} \right) \, .
\end{equation}	

Thirdly, since
\begin{equation}\label{rext10}
	\sum_{n=1}^{N(\eps)} \|\nabla \bs{\varphi} \|_{L^2(K^{\eps}_{n})}^2 \leq \|\nabla \bs \varphi\|^{2}_{L^2(\Omega)} \, ,
\end{equation}	
we can insert \eqref{rext4}-\eqref{rext9}-\eqref{rext10} into the right-hand side of \eqref{rext1}, establishing the bound
	\begin{equation} \label{rext11}
	\|\nabla (\mathcal{R}_{\eps}(\bs{\varphi}) - \bs{\varphi})\|_{L^2(\Omega)} \leq C \left( \|\bs  \varphi\|_{L^{\infty}(\Omega)} + \|\nabla \bs \varphi\|_{L^2(\Omega)} \right)  \, ,
\end{equation}
from where \eqref{propreps0}$_3$ is easily obtained.
\newline
\underline{Proof of \eqref{propreps0}$_4$}. As a consequence of \eqref{rext0}, there holds $\bs{\varphi} - \mathcal{R}_{\eps}(\bs{\varphi}) \in W^{1,2}_{0}(B(\bs{\xi}_{n}^{\eps},\delta_2 \eps); \mathbb{R}^{3})$ for every $n \in \{1,\ldots,N(\varepsilon)\}$, so that the following (explicit) Poincaré inequality holds:
\begin{equation} \label{rext12}
\|\mathcal{R}_{\eps}(\bs{\varphi}) - \bs{\varphi}\|_{L^2(B(\bs{\xi}_{n}^{\eps},\delta_2 \eps))} \leq \dfrac{\delta_2 \eps}{\pi} \|\nabla (\mathcal{R}_{\eps}(\bs{\varphi}) - \bs{\varphi})\|_{L^2(B(\bs{\xi}_{n}^{\eps},\delta_2 \eps))}  \, .
\end{equation}
Combining \eqref{rext12} with the Ladyzhenskaya inequality \cite[Chapter 1, Lemma~2]{ladyzhenskaya1969mathematical}, which reads
$$
\| g \|^{4}_{L^4(\mathbb{R}^{3})} \leq 4 \, \| g \|_{L^2(\mathbb{R}^{3})} \| \nabla g \|^{3}_{L^2(\mathbb{R}^{3})} \qquad \forall g \in \mathcal{C}_{0}^{\infty}(\mathbb{R}^{3}; \mathbb{R}) \, ,
$$
we then deduce
\begin{equation} \label{rext13}
	\|\mathcal{R}_{\eps}(\bs{\varphi}) - \bs{\varphi}\|^{4}_{L^4(B(\bs{\xi}_{n}^{\eps},\delta_2 \eps))} \leq \dfrac{4 \delta_2 \eps}{\pi}  \|\nabla (\mathcal{R}_{\eps}(\bs{\varphi}) - \bs{\varphi})\|^{4}_{L^2(B(\bs{\xi}_{n}^{\eps},\delta_2 \eps))}  \, .
\end{equation}
The bounds in \eqref{rext11}-\eqref{rext13} allow us to write
\begin{equation} \label{rext14}
\begin{aligned}
\|\mathcal{R}_{\eps}(\bs{\varphi}) - \bs{\varphi}\|^{4}_{L^4(\Omega)} & = \sum_{n=1}^{N(\eps)}	\|\mathcal{R}_{\eps}(\bs{\varphi}) - \bs{\varphi}\|^{4}_{L^4(B(\bs{\xi}_{n}^{\eps},\delta_2 \eps))} \leq C \eps \sum_{n=1}^{N(\eps)} \|\nabla (\mathcal{R}_{\eps}(\bs{\varphi}) - \bs{\varphi})\|^{4}_{L^2(B(\bs{\xi}_{n}^{\eps},\delta_2 \eps))} \\[6pt]
& \leq  C \eps \left( \sum_{n=1}^{N(\eps)} \|\nabla (\mathcal{R}_{\eps}(\bs{\varphi}) - \bs{\varphi})\|^{2}_{L^2(B(\bs{\xi}_{n}^{\eps},\delta_2 \eps))} \right)^{2} \leq C \eps \left( \|\bs  \varphi\|_{L^{\infty}(\Omega)} + \|\nabla \bs \varphi\|_{L^2(\Omega)} \right)^{4}  \, ,
\end{aligned}
\end{equation}
thereby yielding \eqref{propreps0}$_4$.
\newline		
\underline{Proof of \eqref{W_eps_resistance}.} Consider $\bs{\mu}^{\varepsilon}_{1},...,\bs{\mu}^{\varepsilon}_{N(\varepsilon)} \in \mathbb{R}^{3}$ satisfying \eqref{kineticuni}, and a sequence $(\bs{v}_\eps)_{\varepsilon \in I_{*}} \subset W^{1,2}(\Omega; \mathbb{R}^{3})$ satisfying the properties in \eqref{sequenceps}. Choose a sufficiently small $\varepsilon_{\triangledown} \in I_{*}$ such that $A^{\eps}_{n} \subset \mathbb{R}^{3} \setminus \overline{B(\bs{\xi}_{n}^{\eps},2\delta_0 \eps^3)}$ for every $\eps \in (0, \varepsilon_{\triangledown})$, with
\begin{equation} \label{anothereps}
A^{\eps}_{n} \doteq B(\bs{\xi}_{n}^{\eps},\delta_1 \eps) \setminus  \overline{B\left( \bs{\xi}_{n}^{\eps}, \dfrac{\delta_1 \eps}{2} \right)} \, . 
\end{equation}
Given a vector field $\bs \varphi \in W^{1,\infty}(\Omega; \mathbb{R}^{3})$ and $\eps \in (0, \min \{ \eps_{\star},\varepsilon_{\triangledown} \})$, from \eqref{rext0} we deduce
\begin{equation} \label{brink0}
\begin{aligned}
& \int_{\Omega} \nabla \mathcal{R}_{\eps}(\bs{\varphi}) : \nabla \bs{v}_\eps - \int_{\Omega} \Lambda_{\eps} (\nabla \cdot \mathcal{R}_{\eps}(\bs{\varphi})) \\[6pt]
& \hspace{-4mm} = \int_{\Omega} \nabla \bs \varphi : \nabla \bs{v}_\eps - \int_{\Omega} \nabla (\bs \varphi - \mathcal{R}_{\eps}(\bs{\varphi})) : \nabla \bs{v}_\eps - \int_{\Omega} \Lambda_{\eps} (\nabla \cdot \bs{\varphi}) + \int_{\Omega} \Lambda_{\eps} (\nabla \cdot (\bs{\varphi} - \mathcal{R}_{\eps}(\bs{\varphi})))   \\[6pt]
& \hspace{-4mm} = \int_{\Omega} \nabla \bs \varphi : \nabla \bs{v}_\eps - \int_{\Omega} \Lambda_{\eps} (\nabla \cdot \bs{\varphi}) \\[6pt]
& \hspace{1mm} - \sum_{n=1}^{N(\eps)} \left[ \int_{D^{\eps}_{n}} \nabla \bs{\omega}^{\eps}_{n,2} : \nabla \bs{v}_\eps + \int_{C_{n}^{\eps}} \nabla \bs{\omega}^{\eps}_{n,1} : \nabla \bs{v}_\eps - \dfrac{1}{|D^{\eps}_{n}|} \left( \int_{K^{\eps}_{n}} \nabla \cdot \bs{\varphi} \right) \! \! \left( \int_{D^{\eps}_{n}} \Lambda_{\eps} \right) \right] \, ,
\end{aligned}
\end{equation}
where we have used that $\Lambda_{\eps} = 0$ in $K^{\eps}_{n}$ for all $1\leq n \leq N(\eps)$.
We estimate, separately, each of the five addends appearing on the right-hand side of \eqref{brink0}.

Firstly, owing to \eqref{sequenceps}$_2$-\eqref{sequenceps}$_3$, there holds 
\begin{equation} \label{brink1}
\lim_{\varepsilon \to 0^{+}} \left( \int_{\Omega} \nabla \bs \varphi : \nabla \bs{v}_\eps - \int_{\Omega} \Lambda_{\eps} (\nabla \cdot \bs{\varphi}) \right) = \int_{\Omega} \nabla \bs \varphi : \nabla \bs{v} - \int_{\Omega} \Lambda (\nabla \cdot \bs{\varphi}) \, .
\end{equation}

Secondly, recalling that the sequence $(\bs{v}_\eps)_{\varepsilon \in I_{*}}$ is uniformly bounded in $W^{1,2}(\Omega; \mathbb{R}^{3})$, an application of the H\"older and Cauchy-Schwarz inequalities, together with \eqref{rext9}, entails
\begin{equation} \label{brink2}
\begin{aligned}
\sum_{n=1}^{N(\eps)} \left| \int_{D^{\eps}_{n}} \nabla \bs{\omega}^{\eps}_{n,2} : \nabla \bs{v}_\eps \right| & 
\leq \sum_{n=1}^{N(\eps)} \|\nabla \bs{\omega}^{\eps}_{n,2} \|_{L^2(D^{\eps}_{n})} \|\nabla \bs{v}_{\eps} \|_{L^2(D^{\eps}_{n})} \\[6pt]
& \leq \sqrt{ \sum_{n=1}^{N(\eps)} \|\nabla \bs{\omega}^{\eps}_{n,2} \|^{2}_{L^2(D^{\eps}_{n})}} \, \sqrt{ \sum_{n=1}^{N(\eps)} \|\nabla \bs{v}_{\eps} \|^{2}_{L^2(B(\bs{\xi}_{n}^{\eps},\delta_2 \eps))}} \\[6pt]
& \leq C \eps \sqrt{\|\bs  \varphi\|^{2}_{L^{\infty}(\Omega)} + \|\nabla \bs \varphi\|^{2}_{L^2(\Omega)}} \, ,
\end{aligned}
\end{equation}
thus,
\begin{equation} \label{brink3}
	\lim_{\varepsilon \to 0^{+}} \sum_{n=1}^{N(\eps)} \left( \int_{D^{\eps}_{n}} \nabla \bs{\omega}^{\eps}_{n,2} : \nabla \bs{v}_\eps \right) = 0 \, .
\end{equation}

Thirdly, let $(\bs{z}^{\varepsilon}_{n,1}, \pi^{\varepsilon}_{n,1}) \in D^{1,2}(\R^3 \setminus \overline{K_{n}^{\eps}}; \R^{3}) \times L^{2}(\R^3 \setminus \overline{K_{n}^{\eps}}; \R)$ be the unique generalized solution to the following Stokes problem in the exterior of $K_{n}^{\eps}$:
\begin{equation}\label{brink4}
	\left\{
	\begin{aligned}
		& -\Delta \bs{z}^{\varepsilon}_{n,1} + \nabla \pi^{\varepsilon}_{n,1}  = \bs{0} \, , \quad  \nabla\cdot \bs{z}^{\varepsilon}_{n,1}  = 0 \ \ \mbox{in} \ \ \R^3 \setminus \overline{K_{n}^{\eps}} \, , \\[5pt]
		& \bs{z}^{\varepsilon}_{n,1} = \bs{\varphi}(\bs{\xi}_{n}^{\eps}) \ \ \mbox{on} \ \ \partial K_{n}^{\eps}  \, , \\[5pt]
		& \bs{z}^{\varepsilon}_{n,1} \to \bs{0} \ \ \mbox{as} \ \ | \bs{\xi} | \to + \infty \, ,
	\end{aligned}
	\right.
\end{equation}		
so that $(\bs{z}^{\varepsilon}_{n,1}, \pi^{\varepsilon}_{n,1}) \in \mathcal{C}^{2}(\R^3 \setminus K_{n}^{\eps}; \R^{3}) \times \mathcal{C}^{1}(\R^3 \setminus K_{n}^{\eps}; \R)$ (see \eqref{stokeslet0} in the Appendix \ref{appendix_sec}) and
\begin{equation}\label{brink5}
\int_{C_{n}^{\eps}} \nabla \bs{\omega}^{\eps}_{n,1} : \nabla \bs{v}_\eps = \int_{C_{n}^{\eps}} \nabla \left (\bs{\omega}^{\eps}_{n,1} - \bs{z}^{\eps}_{n,1} \right) : \nabla \bs{v}_\eps + \int_{C_{n}^{\eps}} \nabla \bs{z}^{\eps}_{n,1} : \nabla \bs{v}_\eps \qquad \forall n \in \{1,\ldots,N(\varepsilon)\} \, .
\end{equation}
Since $\bs{\varphi} \in W^{1,\infty}(\Omega; \mathbb{R}^{3})$ is Lipschitz-continuous, it can be easily seen that
\begin{equation}\label{brink6}
\| \bs{\varphi} - \bs{\varphi}(\bs{\xi}_{n}^{\eps}) \|^{2}_{L^2(B(\bs{\xi}_{n}^{\eps},\delta_{0} \eps^{3}))} \leq C \eps^{15} \| \nabla \bs{\varphi} \|^{2}_{L^{\infty}(\Omega)} \qquad \forall n \in \{1,\ldots,N(\varepsilon)\} \, .
\end{equation}
Therefore, we can invoke the energy estimate \eqref{energy_stokes_1}, as well as \eqref{brink6}, in order to infer
\begin{equation}\label{brink7}
\| \nabla \left (\bs{\omega}^{\eps}_{n,1} - \bs{z}^{\eps}_{n,1} \right) \|^{2}_{L^2(C^{\eps}_{n})} \leq C \eps^{9} \| \nabla \bs{\varphi} \|^{2}_{L^{\infty}(\Omega)} \qquad \forall n \in \{1,\ldots,N(\varepsilon)\} \, ,
\end{equation}
and subsequently, by arguing as in \eqref{brink2}-\eqref{brink3}, that
\begin{equation} \label{brink8}
	\lim_{\varepsilon \to 0^{+}} \sum_{n=1}^{N(\eps)} \left( \int_{C_{n}^{\eps}} \nabla \left(\bs{\omega}^{\eps}_{n,1} - \bs{z}^{\eps}_{n,1} \right) : \nabla \bs{v}_\eps \right) = 0 \, .
\end{equation}
Now, multiplying the Stokes equation \eqref{brink4}$_1$ by a constant vector $\bs{V} \in \mathbb{R}^{3}$, and then integrating by parts in $C^{\eps}_{n}$, we readily notice that, for every $n \in \{1,\ldots,N(\varepsilon)\}$,
\begin{equation} \label{brink88}
\int\limits_{\partial B(\bs{\xi}_{n}^{\eps},\delta_{1} \eps)} \left( \dfrac{\partial \bs{z}^{\eps}_{n,1}}{\partial \bs{\nu}} - \pi^{\varepsilon}_{n,1} \bs{\nu} \right) \cdot \bs{V} = - \int\limits_{\partial K_{n}^{\eps}} \left( \dfrac{\partial \bs{z}^{\eps}_{n,1}}{\partial \bs{\nu}} - \pi^{\varepsilon}_{n,1} \bs{\nu} \right) \cdot \bs{V} = - \eps^3 \mathcal{G}_{n} \bs{\varphi}(\bs{\xi}_{n}^{\eps}) \cdot \bs{V} \qquad \forall \bs{V} \in \mathbb{R}^{3} \, ,
\end{equation}
where $\mathcal{G}_{n} \in \mathbb{R}^{3 \times 3}$ is the Stokes resistance matrix of $K_n$ (see \eqref{def_resistance_matrix}), and $\bs{\nu} \in \mathbb{R}^{3}$ is the outward unit normal to $C^{\eps}_{n}$, which is directed towards the interior of $K_{n}^{\eps}$. Similarly, multiplying the Stokes equation \eqref{brink4}$_1$ by $\bs{v}_{\eps}$ and integrating by parts in $C^{\eps}_{n}$ (recalling \eqref{sequenceps}$_1$-\eqref{brink88}), we deduce, for every $n \in \{1,\ldots,N(\varepsilon)\}$,
\begin{equation} \label{brink9}
\begin{aligned}
\int_{C_{n}^{\eps}} \nabla \bs{z}^{\eps}_{n,1} : \nabla \bs{v}_\eps & = \int\limits_{\partial K_{n}^{\eps}} \left( \dfrac{\partial \bs{z}^{\eps}_{n,1}}{\partial \bs{\nu}} - \pi^{\varepsilon}_{n,1} \bs{\nu} \right) \cdot \bs{\mu}_n^\eps + \int\limits_{\partial B(\bs{\xi}_{n}^{\eps},\delta_{1} \eps)} \left( \dfrac{\partial \bs{z}^{\eps}_{n,1}}{\partial \bs{\nu}} - \pi^{\varepsilon}_{n,1} \bs{\nu} \right) \cdot \bs{v}_\eps \\[6pt]
&=\eps^3 \mathcal{G}_{n} \bs{\varphi}(\bs{\xi}_{n}^{\eps}) \cdot \bs \mu_n^\eps - \eps^3 \mathcal{G}_{n} \bs{\varphi}(\bs{\xi}_{n}^{\eps}) \cdot \bar{\bs{v}}^{\eps}_{n} + \int\limits_{\partial B(\bs{\xi}_{n}^{\eps},\delta_{1} \eps)} \left( \dfrac{\partial \bs{z}^{\eps}_{n,1}}{\partial \bs{\nu}} - \pi^{\varepsilon}_{n,1} \bs{\nu} \right) \cdot \left( \bs{v}_\eps - \bar{\bs{v}}^{\eps}_{n} \right) \, ,
\end{aligned}
\end{equation}
where 
$$
\bar{\bs{v}}^{\eps}_{n} \doteq \dfrac{1}{| A^{\eps}_{n}|} \int_{A^{\eps}_{n}} \bs{v}_\eps \qquad \forall n \in \{1,\ldots,N(\varepsilon)\} \, ,
$$
see \eqref{anothereps}. In virtue of \cite[Exercise III.3.5]{galdi2011introduction}, by a translation argument we can build a divergence-free vector field $\bs{X}_{\! n}^{\eps} \in W^{1,2}(A^{\eps}_{n};\mathbb{R}^{3})$ such that
\begin{equation} \label{restop0}
	\left\{
	\begin{aligned}
		& \bs{X}_{\! n}^{\eps} = \bs{0} \ \ \mbox{on} \ \ \partial B \! \left( \bs{\xi}_{n}^{\eps}, \dfrac{\delta_1 \eps}{2} \right)  \, ; \qquad \bs{X}_{\! n}^{\eps} = \bs{v}_\eps - \bar{\bs{v}}^{\eps}_{n}  \ \ \mbox{on} \ \ \partial B(\bs{\xi}_{n}^{\eps},\delta_1 \eps)   \, ; \\[6pt]
		& \|\nabla \bs{X}_{\! n}^{\eps} \|_{L^2(A^{\eps}_{n})} \leq C \| \bs{v}_\eps - \bar{\bs{v}}^{\eps}_{n} \|_{W^{1/2,2}(\partial  A^{\eps}_{n})} \leq C \| \nabla \bs{v}_\eps \|_{L^{2}(A^{\eps}_{n})}  \, ,
	\end{aligned} 
	\right.
\end{equation}
where the trace and Poincaré inequalities (in $A^{\eps}_{n}$) have been employed in \eqref{restop0}$_3$. As before, multiplying \eqref{brink4}$_1$ by $\bs{X}_{\! n}^{\eps}$ and integrating by parts in $A^{\eps}_{n}$ (recalling \eqref{restop0}), entails
\begin{equation} \label{restop1}
\int\limits_{\partial B(\bs{\xi}_{n}^{\eps},\delta_{1} \eps)} \left( \dfrac{\partial \bs{z}^{\eps}_{n,1}}{\partial \bs{\nu}} - \pi^{\varepsilon}_{n,1} \bs{\nu} \right) \cdot \left( \bs{v}_\eps - \bar{\bs{v}}^{\eps}_{n} \right) = \int_{A_{n}^{\eps}} \nabla \bs{z}^{\eps}_{n,1} : \nabla \bs{X}_{\! n}^{\eps} \qquad \forall n \in \{1,\ldots,N(\varepsilon)\} \, .
\end{equation}
The decay estimate \eqref{decay_stokes_1} can be again applied to deduce
\begin{equation} \label{restop2}
\|\nabla \bs{z}^{\eps}_{n,1} \|_{L^{2}(A_{n}^{\eps})}^{2} \leq C \eps^{5}  \|\bs  \varphi\|^{2}_{L^{\infty}(\Omega)} \qquad \forall n \in \{1,\ldots,N(\varepsilon)\} \, ,
\end{equation}
Arguing as in \eqref{brink2}, we then invoke \eqref{restop0}$_3$- \eqref{restop1}-\eqref{restop2} to obtain the bound
\begin{equation} \label{restop3}
	\begin{aligned}
		\sum_{n=1}^{N(\eps)} \left| \, \int\limits_{\partial B(\bs{\xi}_{n}^{\eps},\delta_{1} \eps)} \left( \dfrac{\partial \bs{z}^{\eps}_{n,1}}{\partial \bs{\nu}} - \pi^{\varepsilon}_{n,1} \bs{\nu} \right) \cdot \left( \bs{v}_\eps - \bar{\bs{v}}^{\eps}_{n} \right) \right| \leq C \eps \|\bs  \varphi\|_{L^{\infty}(\Omega)} \, .
	\end{aligned}
\end{equation}
We then sum, over $n \in \{1,\ldots,N(\varepsilon)\}$, on both sides of identity \eqref{brink9}; owing to \eqref{hyp_cv_measures}$_2$-\eqref{brink5}-\eqref{brink8}-\eqref{restop3}-\eqref{W_eps_resistanceapp}, this leads to
\begin{equation} \label{restop4}
	\lim_{\varepsilon \to 0^{+}} \sum_{n=1}^{N(\eps)} \left( \int_{C_{n}^{\eps}} \nabla \bs{\omega}^{\eps}_{n,1} : \nabla \bs{v}_\eps \right) = \int_{\Omega} (\bs{\mathcal{J}} - \mathcal{G} \bs{v}) \cdot \bs{\varphi} \, .
\end{equation}

Fourthly, since the sequence $(\Lambda_\eps)_{\varepsilon \in I_{*}}$ is uniformly bounded in $L^{1}(\Omega; \mathbb{R})$, an application of the H\"older inequality furnishes
\begin{equation} \label{brink100}
	\begin{aligned}
		\sum_{n=1}^{N(\eps)} \left| \dfrac{1}{|D^{\eps}_{n}|} \left( \int_{K^{\eps}_{n}} \nabla \cdot \bs{\varphi} \right) \! \! \left( \int_{D^{\eps}_{n}} \Lambda_{\eps} \right) \right| 
		\leq 3 \|\nabla \bs{\varphi} \|_{L^{\infty}(\Omega)} \sum_{n=1}^{N(\eps)} \dfrac{|K^{\eps}_{n}|}{|D^{\eps}_{n}|}  \| \Lambda_{\eps} \|_{L^1(D^{\eps}_{n})} \leq C \eps^{6} \| \nabla \bs{\varphi} \|_{L^{\infty}(\Omega)} \|\Lambda_{\eps} \|_{L^1(\Omega)} \, ,
	\end{aligned}
\end{equation}
thus,
\begin{equation} \label{brink101}
	\lim_{\varepsilon \to 0^{+}} \sum_{n=1}^{N(\eps)} \dfrac{1}{|D^{\eps}_{n}|} \left( \int_{K^{\eps}_{n}} \nabla \cdot \bs{\varphi} \right) \left( \int_{D^{\eps}_{n}} \Lambda_{\eps} \right) = 0 \, .
\end{equation}

The limit in \eqref{W_eps_resistance} then follows from \eqref{brink0}-\eqref{brink1}-\eqref{brink3}-\eqref{restop4}-\eqref{brink101}.
\end{proof}
	
\section{Uniform bounds and proof of the homogenization result} \label{epslevelsec}
	
	\subsection{Boundary-value problems at the \texorpdfstring{$\eps$}{\eps}-level: uniform bounds}
Let $\Omega \subset \mathbb{R}^{3}$ be an admissible domain. As in \cite[Section 2]{sperone2023homogenization}, we introduce the following functional spaces:
	$$
	\mathcal{U}(\Omega_{\varepsilon}) \doteq \left\lbrace \bs{v} \in W^{1,2}(\Omega_{\varepsilon}; \mathbb{R}^{3}) \ \Bigg\rvert \ 
	\begin{aligned}
		& \nabla \cdot \bs{v}=0 \ \ \mbox{in} \ \ \Omega_{\varepsilon} \, ; \qquad \bs{v} = \bs{0} \ \ \mbox{on} \ \ \Gamma_{W} \cup \partial K_{\varepsilon} \, ; \, \\[3pt]
		&  \bs{v} \times \bs{\nu} = \bs{0} \ \ \mbox{on} \ \ \Gamma_{k} \quad \text{and} \quad \int_{\Gamma_{k}} \bs{v} \cdot \bs{\nu} = 0 \qquad \forall k \in \{1,\ldots, M\} \
	\end{aligned}
	\right\rbrace 
	$$
	and
	$$
	\mathcal{W}(\Omega_{\varepsilon}) \doteq \left\lbrace \bs{v} \in W^{1,2}(\Omega_{\varepsilon}; \mathbb{R}^{3}) \ \Bigg\rvert \ 
	\begin{aligned}
		& \nabla \cdot \bs{v}=0 \ \ \mbox{in} \ \ \Omega_{\varepsilon} \, ; \qquad \bs{v} = \bs{0} \ \ \mbox{on} \ \ \Gamma_{W} \cup \partial K_{\varepsilon} \, ; \, \\[3pt]
		&  \bs{v} \times \bs{\nu} = \bs{0} \ \ \mbox{on} \ \ \Gamma_{k} \qquad \forall k \in \{1,\ldots, M\} \
	\end{aligned}
	\right\rbrace  \, ,
	$$
\bigskip
$$
\mathcal{V}(\Omega_{\varepsilon}) \doteq \left\lbrace \bs{v} \in W^{1,2}(\Omega_{\varepsilon}; \mathbb{R}^{3}) \ \Bigg\rvert \ 
\begin{aligned}
	& \nabla \cdot \bs{v}=0 \ \ \mbox{in} \ \ \Omega_{\varepsilon} \, ; \qquad \bs{v} = \bs{0} \ \ \mbox{on} \ \ \Gamma_{W} \, ; \, \\[3pt]
	&  \bs{v} \times \bs{\nu} = \bs{0} \ \ \mbox{on} \ \ \Gamma_{k} \qquad \forall k \in \{1,\ldots, M\} \
\end{aligned}
\right\rbrace  \, .
$$
Since the Poincaré inequality holds in $\mathcal{U}(\Omega_{\varepsilon})$, $\mathcal{W}(\Omega_{\varepsilon})$  and $\mathcal{V}(\Omega_{\varepsilon})$, they all constitute Hilbert spaces under the Dirichlet scalar product of the gradients. We are now in position to give the following definition for the weak solutions of problems \eqref{nsstokes1}-\eqref{nsstokespd1} (equivalently, of problems \eqref{nsstokes0}-\eqref{nsstokespd}):
	\begin{definition}\label{weaksolutionpnf}
		Let $\Omega \subset \mathbb{R}^{3}$ be an admissible domain, $\bs{f} \in L^{2}(\Omega; \mathbb{R}^{3})$,  $(\bs{\mu}^{\varepsilon}_{n})^{N(\eps)}_{n=1} \subset \mathbb{R}^{3}$, $(\mathcal{F}_{k})^{M}_{k=1} \subset \mathbb{R}$ satisfying \eqref{compcon} and $\bs{\Psi}_{\! \varepsilon} \in W^{2,2}(\Omega_{\varepsilon}; \mathbb{R}^{3})$ the vector field arising from Lemma \ref{fluxcarrier}. A vector field $\bs{u} \in \mathcal{V}(\Omega_{\varepsilon})$ is called a \textbf{weak solution} of the prescribed net flux problem \eqref{nsstokes1} if $\bs{u} - \bs{\Psi}_{\! \varepsilon} \in \mathcal{U}(\Omega_{\varepsilon})$ and
		\begin{equation} \label{nstokesdebil}
			\int_{\Omega_{\varepsilon}} \nabla \bs{u} : \nabla \bs{\varphi} + \int_{\Omega_{\varepsilon}} ((\nabla \times \bs{u}) \times \bs{u}) \cdot \bs{\varphi}  = \int_{\Omega_{\varepsilon}} \bs{f} \cdot \bs{\varphi} \qquad \forall \bs{\varphi} \in \mathcal{U}(\Omega_{\varepsilon}) \, .
		\end{equation}
	\end{definition}
	
	\begin{definition}\label{weaksolutionppd}
		Let $\Omega \subset \mathbb{R}^{3}$ be an admissible domain, $\bs{f} \in L^{2}(\Omega; \mathbb{R}^{3})$, $(\bs{\mu}^{\varepsilon}_{n})^{N(\eps)}_{n=1} \subset \mathbb{R}^{3}$, $(p_{k})^{M}_{k=1} \subset \mathbb{R}$ and $\bs{A}_{\varepsilon} \in \mathcal{C}^{\infty}(\overline{\Omega}; \mathbb{R}^{3})$ the vector field arising from Lemma~\ref{smallfield}. A vector field $\bs{u} \in \mathcal{V}(\Omega_{\varepsilon})$ is called a \textbf{weak solution} of the prescribed pressure drop problem \eqref{nsstokespd1} if $\bs{u} - \bs{A}_{\varepsilon} \in \mathcal{W}(\Omega_{\varepsilon})$ and
		\begin{equation} \label{nstokesdebilpd}
			\int_{\Omega_{\varepsilon}} \nabla \bs{u} : \nabla \bs{\varphi} + \int_{\Omega_{\varepsilon}} ((\nabla \times \bs{u}) \times \bs{u}) \cdot \bs{\varphi} +  \sum_{k=1}^{M} p_{k} \left( \int_{\Gamma_{k}} \bs{\varphi} \cdot \bs{\nu} \right) = \int_{\Omega_{\varepsilon}} \bs{f} \cdot \bs{\varphi} \qquad \forall \bs{\varphi} \in \mathcal{W}(\Omega_{\varepsilon}) \, .
		\end{equation}
	\end{definition}
	
	We refer to \cite[Section 4]{heywood1996artificial} and \cite[Section 2]{korobkov2020solvability} for an explanation of the fact that the boundary conditions involving the Bernoulli pressure in \eqref{nsstokes1}-\eqref{nsstokespd1} are implicitly contained in the variational formulations \eqref{nstokesdebil}-\eqref{nstokesdebilpd} of Definitions \ref{weaksolutionpnf}-\ref{weaksolutionppd}.
	
	The first main result of this section provides uniform bounds (with respect to $\varepsilon \in I_{*}$) for the solutions of the prescribed flux problem \eqref{nsstokes1}.
	
	\begin{theorem} \label{epslevel}
			Let $\Omega \subset \mathbb{R}^{3}$ be an admissible domain, $\bs{f} \in L^{2}(\Omega; \mathbb{R}^{3})$,  $(\bs{\mu}^{\varepsilon}_{n})^{N(\eps)}_{n=1} \subset \mathbb{R}^{3}$ and $(\mathcal{F}_{k})^{M}_{k=1} \subset \mathbb{R}$ satisfying \eqref{compcon}. There exists (at least) one weak solution $\bs{u}_{\varepsilon} \in W^{2,2}(\Omega_{\varepsilon}; \mathbb{R}^{3}) \cap \mathcal{V}(\Omega_{\varepsilon})$ of the prescribed net flux problem \eqref{nsstokes1} and an associated Bernoulli pressure $\Phi_{\varepsilon} \in W^{1,2}(\Omega_{\varepsilon}; \mathbb{R})$ such that the pair $(\bs{u}_{\varepsilon},\Phi_{\varepsilon})$ satisfies \eqref{nsstokes1} in strong form, for some constants $p_{1}^{\varepsilon},\ldots,p_{M-1}^{\varepsilon} \in \mathbb{R}$.
		
		Furthermore, suppose there exists $\delta_{3} \in (0,\min \{\ell_{1},\ldots,\ell_{M}\} )$ such that, for every $\varepsilon \in I_{*} $, 
		\begin{equation} \label{extrahyp}
	\partial B \! \left(\bs{\xi}^{\varepsilon}_{n}, \delta_{2} \, \varepsilon \right) \cap \left\lbrace \bs{\xi} \in \overline{\Omega_{k}} \ \vert \ \textup{dist}( \bs{\xi} , \Gamma_{k}) \leq \delta_{3} \right\rbrace = \emptyset \qquad \forall n \in \{1,\ldots,N(\varepsilon)\} \, , \ \ \forall k \in \{1,\ldots,M\} \, .
		\end{equation}
		Then, under assumption \eqref{kineticuni}, the uniform bound
		\begin{equation} \label{uboundpf}
			\sup_{\varepsilon \in I_{*}} \left( \| \nabla \bs{u}_{\varepsilon} \|_{L^{2}(\Omega_{\varepsilon})} + \| \Phi_{\varepsilon} \|_{L^{2}(\Omega_{\varepsilon})} + \max_{k \in \{1,\ldots,M-1\}} |p_{k}^{\varepsilon}| \right) \leq C_{*} \, ,
		\end{equation}
		holds for some constant $C_{*} > 0$ that depends on $\Omega$, $\bs{f}$, $(\mathcal{F}_{k})^{M}_{k=1}$, $\mathcal{S}_{*}$ and $\{ \delta_{i} \}^{3}_{i=0}$.
	\end{theorem}
	\noindent
	\begin{proof}
		In what follows, $C > 0$ will always denote a generic constant that depends on $\Omega$, $(\mathcal{F}_{k})^{M}_{k=1}$, $\mathcal{S}_{*}$ and $\{ \delta_{i} \}^{2}_{i=0}$ (independently of $\varepsilon \in I_{*}$), but that may change from line to line. 
		
		Given $\bs{f} \in L^{2}(\Omega; \mathbb{R}^{3})$,  $(\bs{\mu}^{\varepsilon}_{n})^{N(\eps)}_{n=1} \subset \mathbb{R}^{3}$,  $(\mathcal{F}_{k})^{M}_{k=1} \subset \mathbb{R}$ satisfying \eqref{compcon} and $\varepsilon \in I_{*}$, a direct extension of \cite[Theorem~3.3]{sperone2022} ensures the existence of at least one weak solution $\bs{u}_{\varepsilon} \in \mathcal{V}(\Omega_{\varepsilon})$ of problem \eqref{nsstokes1}. Afterwards, \cite[Theorem~3.2]{sperone2022} guarantees that $\bs{u}_{\varepsilon} \in W^{2,2}(\Omega_{\varepsilon}; \mathbb{R}^{3})$ and the existence of an associated Bernoulli pressure $\Phi_{\varepsilon} \in W^{1,2}(\Omega_{\varepsilon}; \mathbb{R})$ satisfying
		\begin{equation}\label{nsstokes1aprox}
			\left\{
		\begin{aligned}
			& -\Delta \bs{u}_{\varepsilon}+ (\nabla \times \bs{u}_{\varepsilon}) \times \bs{u}_{\varepsilon} + \nabla \Phi_{\varepsilon}=\bs{f} \, ,\ \quad  \nabla\cdot \bs{u}_{\varepsilon}=0 \ \ \mbox{in} \ \ \Omega_{\varepsilon} \, ; \\[4pt]
			& \bs{u}_{\varepsilon}=\bs{0} \ \ \mbox{on} \ \ \Gamma_{W} \, ; \qquad \bs{u}_{\varepsilon}=\bs{\mu}^{\varepsilon}_{n} \ \ \mbox{on} \ \ \partial K^{\varepsilon}_{n} \qquad \forall n \in \{1,\ldots,N(\varepsilon)\} \, ; \\[4pt]
			& \bs{u}_{\varepsilon} \times \bs{\nu} = \bs{0} \, , \quad - \dfrac{\partial}{\partial \bs{\nu}}(\bs{u}_{\varepsilon} \cdot \bs{\nu}) + \Phi_{\varepsilon} = p_{k}^{\eps} \ \ \mbox{ on } \ \ \Gamma_{k} \ \ \ \forall k \in \{1,\ldots,M-1\} \, ;\\[4pt]
			& \bs{u}_{\varepsilon} \times \bs{\nu} = \bs{0} \, , \quad - \dfrac{\partial}{\partial \bs{\nu}}(\bs{u}_{\varepsilon} \cdot \bs{\nu}) + \Phi_{\varepsilon} = 0 \ \ \mbox{ on } \ \ \Gamma_{M} \, ;\\[4pt]
			& \int_{\Gamma_{k}} \bs{u}\cdot \bs{\nu} = \mathcal{F}_{k} \ \ \ \forall k \in \{1,\ldots,M\} \, ,
		\end{aligned}
		\right.
		\end{equation}
		in strong form, for some (unknown) constants $p_{1}^{\varepsilon},\ldots,p_{M-1}^{\varepsilon} \in \mathbb{R}$. We now extend these functions to the interior of the particles according to the expressions
		\begin{equation} \label{extension}
			\widetilde{\bs{u}}_{\varepsilon} \doteq 
			\begin{cases}
				\bs{u}_{\varepsilon} & \quad \text{in} \ \ \Omega_{\varepsilon} \, ,\\[3pt]
				\bs{\mu}^{\varepsilon}_{n} & \quad  \text{in} \ \ \overline{K^{\varepsilon}_{n}} \qquad \forall n \in \{1,...,N(\varepsilon)\} \, ,
			\end{cases}
			\qquad
			\text{and}
			\qquad 
			\widetilde{\Phi}_{\varepsilon} \doteq 
			\begin{cases}
				\Phi_{\varepsilon} & \quad \text{in} \ \ \Omega_{\varepsilon} \, ,\\[3pt]
				0 & \quad  \text{in} \ \ \overline{K_{\varepsilon}}  \, .
			\end{cases}
		\end{equation}
		Notice, therefore, that $\widetilde{\Phi}_{\varepsilon} \in L^{2}(\Omega; \mathbb{R})$ and $\widetilde{\bs{u}}_{\varepsilon} \in \mathcal{V}(\Omega)$, where we have introduced
		\begin{equation} \label{espaciosstar}
		\mathcal{V}(\Omega) \doteq \left\lbrace \bs{v} \in W^{1,2}(\Omega; \mathbb{R}^{3}) \ \Bigg\rvert \ 
		\begin{aligned}
			& \nabla \cdot \bs{v}=0 \ \ \mbox{in} \ \ \Omega \, ; \qquad \bs{v} = \bs{0} \ \ \mbox{on} \ \ \Gamma_{W} \, ; \, \\[3pt]
			&  \bs{v} \times \bs{\nu} = \bs{0} \ \ \mbox{on} \ \ \Gamma_{k} \qquad \forall k \in \{1,\ldots, M\} \
		\end{aligned}
		\right\rbrace  \, ,
		\end{equation}
		which is a closed subspace of $W^{1,2}(\Omega; \mathbb{R}^{3})$ under the Dirichlet scalar product of the gradients. Moreover,
		\begin{equation}\label{extvel}
			\| \nabla \widetilde{\bs{u}}_{\varepsilon} \|_{L^{2}(\Omega)} = \| \nabla \bs{u}_{\varepsilon} \|_{L^{2}(\Omega_{\varepsilon})} \qquad \text{and} \qquad \| \widetilde{\Phi}_{\varepsilon} \|_{L^{2}(\Omega)} = \| \Phi_{\varepsilon} \|_{L^{2}(\Omega_{\varepsilon})} \, .
		\end{equation}
		
		Firstly, we must estimate the constants $p_{1}^{\varepsilon},\ldots,p_{M-1}^{\varepsilon} \in \mathbb{R}$ in terms of the Dirichlet norm of $\bs{u}_{\varepsilon}$ in $\Omega_{\varepsilon}$. For this, take any $\ell \in \{1,\ldots,M-1\}$; in virtue of Lemma~\ref{fluxcarrier}, consider a flux carrier  $\bs{\Upsilon}_{\! \varepsilon} \in W^{2,2}(\Omega_{\varepsilon}; \mathbb{R}^{3})$ satisfying the following properties:
		\begin{equation} \label{vecpsiunit}
			\left\{
			\begin{aligned}
				& \nabla \cdot \bs{\Upsilon}_{\! \varepsilon} =0 \ \ \mbox{in} \ \ \Omega_{\varepsilon} \, ;  \quad  \bs{\Upsilon}_{\! \varepsilon} = \bs{0} \ \ \mbox{on} \ \ \Gamma_{W} \cup \partial K_{\varepsilon} \, ; \quad \bs{\Upsilon}_{\! \varepsilon} \times \bs{\nu} = \bs{0} \ \ \mbox{on} \ \ \Gamma_{k} \quad \forall k \in \{1,\ldots,M\} \, ; \\[4pt]
				& \int_{\Gamma_{\ell}} \bs{\Upsilon}_{\! \varepsilon} \cdot \bs{\nu} = 1 \, ; \quad \int_{\Gamma_{M}} \bs{\Upsilon}_{\! \varepsilon} \cdot \bs{\nu} = -1 \, ; \quad \int_{\Gamma_{k}} \bs{\Upsilon}_{\! \varepsilon} \cdot \bs{\nu} = 0 \quad \forall k \in \{1,\ldots,M-1\} \setminus \{\ell\} \, ; \\[4pt]
				& \| \nabla \bs{\Upsilon}_{\! \varepsilon} \|_{L^{2}(\Omega_{\varepsilon})} \leq C \, .
			\end{aligned}
			\right.
		\end{equation}
		Multiply the first identity in \eqref{nsstokes1aprox}$_1$ by $\bs{\Upsilon}_{\! \varepsilon}$ and integrate by parts in $\Omega_{\varepsilon}$, each term separately, to obtain
		\begin{equation}\label{epsbyparts0}
				- \int_{\Omega_{\varepsilon}} \Delta \bs{u}_{\varepsilon} \cdot \bs{\Upsilon}_{\! \varepsilon}  = \int_{\Omega_{\varepsilon}} \nabla \bs{u}_{\varepsilon} : \nabla \bs{\Upsilon}_{\! \varepsilon} - \sum_{k=1}^{M} \left( \int_{\Gamma_{k}} \dfrac{\partial \bs{u}_{\varepsilon}}{\partial \bs{\nu}}  \cdot \bs{\Upsilon}_{\! \varepsilon} \right) \, .
		\end{equation}
		Regarding the pressure term, from \eqref{nsstokes1aprox}$_3$-\eqref{nsstokes1aprox}$_4$-\eqref{vecpsiunit} we infer
		\begin{equation} \label{epsbyparts2}
			\int_{\Omega_{\varepsilon}} \nabla \Phi_{\varepsilon} \cdot \bs{\Upsilon}_{\! \varepsilon} =  \int_{\partial \Omega_{\varepsilon}} \Phi_{\varepsilon} (\bs{\Upsilon}_{\! \varepsilon} \cdot \bs{\nu}) = p^{\varepsilon}_{\ell} + \sum_{k=1}^{M} \left( \int_{\Gamma_{k}} (\bs{\Upsilon}_{\! \varepsilon} \cdot \bs{\nu}) \dfrac{\partial}{\partial \bs{\nu}}(\bs{u}_{\varepsilon} \cdot \bs{\nu}) \right) \, .
		\end{equation}
		By adding the identities \eqref{epsbyparts0}-\eqref{epsbyparts2}, and recalling \eqref{firstidentity0} in Remark \ref{firstidentity}, we get
		\begin{equation} \label{epsbyparts5}
			\int_{\Omega_{\varepsilon} } \nabla \bs{u}_{\varepsilon} : \nabla \bs{\Upsilon}_{\! \varepsilon} + \int_{\Omega_{\varepsilon}} ((\nabla \times \bs{u}_{\varepsilon}) \times \bs{u}_{\varepsilon}) \cdot \bs{\Upsilon}_{\! \varepsilon} + p^{\varepsilon}_{\ell} = \int_{\Omega_{\varepsilon}} \bs{f} \cdot \bs{\Upsilon}_{\! \varepsilon} \, .
		\end{equation}
		Define $\bs{\widetilde{\Upsilon}}_{\! \varepsilon} \in W^{1,2}(\Omega; \mathbb{R}^{3})$ as in \eqref{extension}$_2$, and observe that $\| \nabla \bs{\widetilde{\Upsilon}}_{\! \varepsilon} \|_{L^{2}(\Omega)}=\| \nabla \bs{\Upsilon}_{\! \varepsilon} \|_{L^{2}(\Omega_{\varepsilon})}$.
		Applying the H\"older, Sobolev and Poincaré inequalities (in $\Omega$) in \eqref{epsbyparts5}, as well as \eqref{extvel}-\eqref{vecpsiunit}, entails  
		\begin{equation} \label{epsbyparts6}
			\begin{aligned}
				| p^{\varepsilon}_{\ell} | & = \left| \int_{\Omega} \nabla \bs{u}_{\varepsilon} : \nabla \bs{\Upsilon}_{\! \varepsilon} + \int_{\Omega} ((\nabla \times \widetilde{\bs{u}}_{\varepsilon} ) \times \widetilde{\bs{u}}_{\varepsilon} ) \cdot \bs{\widetilde{\Upsilon}}_{\! \varepsilon} - \int_{\Omega} \bs{f} \cdot \bs{\widetilde{\Upsilon}}_{\! \varepsilon} \right| \\[6pt]
				& \leq \| \nabla \bs{u}_{\varepsilon} \|_{L^{2}(\Omega_{\varepsilon})} \| \nabla \bs{\Upsilon}_{\! \varepsilon} \|_{L^{2}(\Omega_{\varepsilon})} + \| \nabla \widetilde{\bs{u}}_{\varepsilon} \|_{L^{2}(\Omega)} \| \widetilde{\bs{u}}_{\varepsilon} \|_{L^{4}(\Omega)} \| \bs{\widetilde{\Upsilon}}_{\! \varepsilon} \|_{L^{4}(\Omega)} + \| \bs{f} \|_{L^{2}(\Omega)} \| \bs{\widetilde{\Upsilon}}_{\! \varepsilon} \|_{L^{2}(\Omega)} \\[6pt]
				& \leq \| \nabla \bs{u}_{\varepsilon} \|_{L^{2}(\Omega_{\varepsilon})} \| \nabla \bs{\Upsilon}_{\! \varepsilon} \|_{L^{2}(\Omega_{\varepsilon})} + C \| \nabla \widetilde{\bs{u}}_{\varepsilon} \|^{2}_{L^{2}(\Omega)} \| \nabla \bs{\widetilde{\Upsilon}}_{\! \varepsilon} \|_{L^{2}(\Omega)} + C \| \bs{f} \|_{L^{2}(\Omega)} \| \nabla \bs{\widetilde{\Upsilon}}_{\! \varepsilon} \|_{L^{2}(\Omega)} \\[6pt]
				& \leq C \left( \| \nabla \bs{u}_{\varepsilon}  \|^{2}_{L^{2}(\Omega_{\varepsilon})} + \| \nabla \bs{u}_{\varepsilon}  \|_{L^{2}(\Omega_{\varepsilon})} + \| \bs{f}  \|_{L^{2}(\Omega)} \right) \qquad \forall \ell \in \{1,\ldots,M-1\} \, .
			\end{aligned}
		\end{equation}
		
		Secondly, our goal is to derive an analogous estimate to \eqref{epsbyparts6} for the $L^{2}(\Omega_{\varepsilon})$-norm of the Bernoulli pressure. In order to do so, owing to Lemma~\ref{bogtype}, consider a vector field $\bs{J}_{\! \varepsilon} \in W^{1,2}(\Omega_{\varepsilon}; \mathbb{R}^3)$ such that
		\begin{equation} \label{vecjeeps}
			\left\{
			\begin{aligned}
			& \nabla \cdot \bs{J}_{\! \varepsilon} = \Phi_{\varepsilon}  \ \ \mbox{in} \ \ \Omega_{\varepsilon} \, ; \qquad \bs{J}_{\! \varepsilon} = \bs{0} \ \ \mbox{on} \ \ \Gamma_{W} \cup \partial K_{\varepsilon} \cup \Gamma_{k} \quad \forall k \in \{1, \ldots, M-1\} \, ; \\[5pt]
			& \bs{J}_{\! \varepsilon} \times \bs{\nu} = \bs{0} \ \ \mbox{on} \ \ \Gamma_{M} \, ; \qquad \| \nabla \bs{J}_{\! \varepsilon} \|_{L^{2}(\Omega_{\varepsilon})} \leq C \| \Phi_{\varepsilon}  \|_{L^{2}(\Omega_{\varepsilon})} \, .
			\end{aligned}
			\right.
		\end{equation}
		Multiplying the first identity in \eqref{nsstokes1aprox}$_1$ by $\bs{J}_{\! \varepsilon}$, integrating by parts in $\Omega_{\varepsilon}$ as in \eqref{epsbyparts0}-\eqref{epsbyparts2} and enforcing \eqref{vecjeeps}, yields
		\begin{equation} \label{epsbyparts7}
			\int_{\Omega_{\varepsilon}} \nabla \bs{u}_{\varepsilon} : \nabla \bs{J}_{\! \varepsilon} + \int_{\Omega_{\varepsilon}} ((\nabla \times \bs{u}_{\varepsilon}) \times \bs{u}_{\varepsilon}) \cdot \bs{J}_{\! \varepsilon} - \| \Phi_{\varepsilon} \|^{2}_{L^{2}(\Omega_{\varepsilon})} = \int_{\Omega_{\varepsilon}} \bs{f} \cdot \bs{J}_{\! \varepsilon} \, .
		\end{equation}
		Define $\bs{\widetilde{J}}_{\! \varepsilon} \in W^{1,2}(\Omega; \mathbb{R}^{3})$ as in \eqref{extension}$_2$, and observe that $\| \nabla \bs{\widetilde{J}}_{\! \varepsilon} \|_{L^{2}(\Omega)}=\| \nabla \bs{J}_{\! \varepsilon} \|_{L^{2}(\Omega_{\varepsilon})}$. We subsequently apply the H\"older, Sobolev and Poincaré inequalities (in $\Omega$) in \eqref{epsbyparts7}, as well as \eqref{extvel}-\eqref{vecjeeps}, to write
		\begin{equation} \label{bypartsje3}
			\begin{aligned}
				\| \Phi_{\varepsilon} \|^{2}_{L^{2}(\Omega_{\varepsilon})} & = \int_{\Omega_{\varepsilon}} \nabla \bs{u}_{\varepsilon} : \nabla \bs{J}_{\! \varepsilon} + \int_{\Omega} ((\nabla \times \widetilde{\bs{u}}_{\varepsilon}) \times \widetilde{\bs{u}}_{\varepsilon}) \cdot \bs{\widetilde{J}}_{\! \varepsilon} - \int_{\Omega} \bs{f} \cdot \bs{\widetilde{J}}_{\! \varepsilon} \\[6pt]
				& \leq \| \nabla \bs{u}_{\varepsilon} \|_{L^{2}(\Omega_{\varepsilon})} \| \nabla \bs{J}_{\! \varepsilon} \|_{L^{2}(\Omega_{\varepsilon})} + \| \nabla \widetilde{\bs{u}}_{\varepsilon} \|_{L^{2}(\Omega)} \| \widetilde{\bs{u}}_{\varepsilon} \|_{L^{4}(\Omega)} \| \bs{\widetilde{J}}_{\! \varepsilon} \|_{L^{4}(\Omega)} + \| \bs{f} \|_{L^{2}(\Omega)} \| \bs{\widetilde{J}}_{\! \varepsilon} \|_{L^{2}(\Omega)} \\[6pt]
				& \leq \| \nabla \bs{u}_{\varepsilon} \|_{L^{2}(\Omega_{\varepsilon})} \| \nabla \bs{J}_{\! \varepsilon} \|_{L^{2}(\Omega_{\varepsilon})} + C \| \nabla \widetilde{\bs{u}}_{\varepsilon} \|^{2}_{L^{2}(\Omega)} \| \nabla \bs{\widetilde{J}}_{\! \varepsilon} \|_{L^{2}(\Omega)} + C \| \bs{f} \|_{L^{2}(\Omega)} \| \nabla \bs{\widetilde{J}}_{\! \varepsilon} \|_{L^{2}(\Omega)} \\[6pt]
				& \leq C \left( \| \nabla \bs{u}_{\varepsilon} \|^{2}_{L^{2}(\Omega_{\varepsilon})} + \| \nabla \bs{u}_{\varepsilon} \|_{L^{2}(\Omega_{\varepsilon})} + \| \bs{f} \|_{L^{2}(\Omega)} \right) \| \Phi_{\varepsilon} \|_{L^{2}(\Omega_{\varepsilon})} \, ,
			\end{aligned}
		\end{equation}
		thereby yielding
		\begin{equation} \label{unipress2}
			\| \Phi_{\varepsilon} \|_{L^{2}(\Omega_{\varepsilon})} \leq C \left( \| \nabla \bs{u}_{\varepsilon} \|^{2}_{L^{2}(\Omega_{\varepsilon})} + \| \nabla \bs{u}_{\varepsilon} \|_{L^{2}(\Omega_{\varepsilon})} + \| \bs{f} \|_{L^{2}(\Omega)} \right) \, .
		\end{equation}
		
		By contradiction, suppose now that the norms $\| \nabla \bs{u}_{\varepsilon} \|_{L^{2}(\Omega_{\varepsilon})}$ are \textit{not} uniformly bounded with respect to $\varepsilon \in I_{*}$. Then, there must exist a sub-sequence (not being relabeled) such that
		\begin{equation} \label{divergent}
			\lim_{\varepsilon \to 0^{+}} \mathcal{Z}_{\varepsilon} = + \infty  \quad \text{with} \ \ \mathcal{Z}_{\varepsilon} \doteq \| \nabla \bs{u}_{\varepsilon} \|_{L^{2}(\Omega_{\varepsilon})} \quad \forall \varepsilon \in I_{*} \, .
		\end{equation}
		The estimates in \eqref{epsbyparts6}-\eqref{unipress2} (see also \eqref{extvel}) enable us to establish that, along this divergent sub-sequence \eqref{divergent}, the following sequences are all uniformly bounded with respect to $\varepsilon \in I_{*}$:
		$$
		( \widehat{\bs{u}}_{\varepsilon} )_{\varepsilon \in I_{*}} \doteq \left( \dfrac{\widetilde{\bs{u}}_{\varepsilon}}{\mathcal{Z}_{\varepsilon}} \right)_{\! \! \varepsilon \in I_{*}} \subset \mathcal{V}(\Omega)  \, ; \qquad ( \widehat{\Phi}_{\varepsilon} )_{\varepsilon \in I_{*}} \doteq  \left( \dfrac{\widetilde{\Phi}_{\varepsilon}}{\mathcal{Z}^{2}_{\varepsilon}} \right)_{\! \! \varepsilon \in I_{*}} \subset L^{2}(\Omega; \mathbb{R}) \, ; \qquad (\widehat{p}^{\: \varepsilon}_{\ell})_{\varepsilon \in I_{*}} \doteq \left( \dfrac{p_{\ell}^{\varepsilon}}{\mathcal{Z}^{2}_{\varepsilon}} \right)_{\! \! \varepsilon \in I_{*}} \subset \mathbb{R} \, .
		$$
		There exist $\widehat{\bs{u}} \in \mathcal{V}(\Omega)$, $\widehat{\Phi} \in L^{2}(\Omega; \mathbb{R})$ and $\widehat{p}_{\ell} \in \mathbb{R}$ such that the following convergences hold as $\varepsilon \to 0^{+}$:
		\begin{equation} \label{convergencesn11}
			\begin{aligned} 
				& \widehat{\bs{u}}_\varepsilon \rightharpoonup \widehat{\bs{u}} \ \ \ \text{weakly in} \ W^{1,2}(\Omega; \mathbb{R}^3) \, ;& \qquad &\widehat{\bs{u}}_\varepsilon \to \widehat{\bs{u}} \ \ \ \text{strongly in} \ L^{4}(\Omega; \mathbb{R}^3) \, ;  \\[3pt]
				& \widehat{\Phi}_\varepsilon \rightharpoonup \widehat{\Phi} \ \ \ \text{weakly in} \ L^{2}(\Omega; \mathbb{R}) \, ;& \qquad & \widehat{p}^{\: \varepsilon}_{\ell} \to \widehat{p}_{\ell} \ \ \text{in} \ \ \mathbb{R} \, ,
			\end{aligned}
		\end{equation}
		for every $\ell \in \{1,\ldots, M-1\}$, along sub-sequences that are not being relabeled.
		
		Let $\bs{A}_{\varepsilon} \in \mathcal{C}^{\infty}(\overline{\Omega}; \mathbb{R}^{3})$ be the vector field arising from Lemma~\ref{smallfield}. Multiplying the first identity in \eqref{nsstokes1aprox}$_1$ by $\bs{u}_{\varepsilon} - \bs{A}_{\varepsilon}$, integrating by parts in $\Omega_{\varepsilon}$ and dividing the resulting identity by $\mathcal{Z}^{2}_{\varepsilon}$, furnishes
		\begin{equation} \label{epsbyparts8}
			1 - \dfrac{1}{\mathcal{Z}_{\varepsilon}} \int_{\Omega} \nabla \widehat{\bs{u}}_{\varepsilon} : \nabla \bs{A}_{\varepsilon} - \int_{\Omega} ((\nabla \times \widehat{\bs{u}}_{\varepsilon}) \times \widehat{\bs{u}}_{\varepsilon}) \cdot \bs{A}_{\varepsilon} + \sum_{\ell=1}^{M-1} \mathcal{F}_{\ell} \, \widehat{p}^{\: \varepsilon}_{\ell}  = \dfrac{1}{\mathcal{Z}_{\varepsilon}} \int_{\Omega} \bs{f} \cdot \left(\widehat{\bs{u}}_{\varepsilon} - \dfrac{\bs{A}_{\varepsilon}}{\mathcal{Z}_{\varepsilon}} \right) \, .
		\end{equation}
		Applying the H\"older, Sobolev and Poincaré inequalities (in $\Omega$), as well as \eqref{smallfieldpropbounds}, we can estimate:
		\begin{equation} \label{absurdum1}
			\begin{aligned}
				& \left| \dfrac{1}{\mathcal{Z}_{\varepsilon}} \int_{\Omega} \nabla \widehat{\bs{u}}_{\varepsilon} : \nabla \bs{A}_{\varepsilon} \right| \leq \dfrac{1}{\mathcal{Z}_{\varepsilon}} \| \nabla \widehat{\bs{u}}_{\varepsilon} \|_{L^{2}(\Omega)} \| \nabla \bs{A}_{\varepsilon} \|_{L^{2}(\Omega)}  \leq \dfrac{C}{\mathcal{Z}_{\varepsilon}} \, , \\[6pt]
				& \left| \int_{\Omega} ((\nabla \times \widehat{\bs{u}}_{\varepsilon}) \times \widehat{\bs{u}}_{\varepsilon}) \cdot \bs{A}_{\varepsilon} \right| \leq \| \nabla \times \widehat{\bs{u}}_{\varepsilon} \|_{L^{2}(\Omega)} \| \widehat{\bs{u}}_{\varepsilon} \|_{L^{4}(\Omega)} \| \bs{A}_{\varepsilon} \|_{L^{4}(\Omega)} \leq C \varepsilon^{3/4} \, , \\[6pt]
				&\left| \dfrac{1}{\mathcal{Z}_{\varepsilon}} \int_{\Omega} \bs{f} \cdot \left(\widehat{\bs{u}}_{\varepsilon} - \dfrac{\bs{A}_{\varepsilon}}{\mathcal{Z}_{\varepsilon}} \right)  \right| \leq \dfrac{C}{\mathcal{Z}_{\varepsilon}} \| \bs{f}  \|_{L^{2}(\Omega)} \left\| \nabla \left(\widehat{\bs{u}}_{\varepsilon} - \dfrac{\bs{A}_{\varepsilon}}{\mathcal{Z}_{\varepsilon}} \right) \right\|_{L^{2}(\Omega)} \leq \dfrac{C}{\mathcal{Z}_{\varepsilon}} \| \bs{f}  \|_{L^{2}(\Omega)} \, .
			\end{aligned}
		\end{equation}
		Consequently, owing to \eqref{absurdum1}, we let $\varepsilon \to 0^{+}$ in \eqref{epsbyparts8} (along the sub-sequences \eqref{convergencesn11}) to deduce 
		\begin{equation} \label{vlambdalimn}
			1 = - \sum_{\ell=1}^{M-1} \mathcal{F}_{\ell} \, \widehat{p}_{\ell} \, .
		\end{equation}
		A contradiction will be reached in \eqref{vlambdalimn} after proving that $\widehat{p}_{1} = \ldots = \widehat{p}_{M-1} = 0$.
		
		Let $\phi \in \mathcal{C}_{0}^{\infty}(\Omega; \mathbb{R})$ be a scalar function.  Denoting by $\{\widehat{\bs{e}}_{1},\widehat{\bs{e}}_{2},\widehat{\bs{e}}_{3}\}$ the canonical basis in $\R^3$, we put
		$$
		\bs{\omega}_{\! j}^{\eps} \doteq \mathcal{R}_{\eps}(\widehat{\bs{e}}_{j}) \in W^{1,2}(\Omega; \mathbb{R}^3) \qquad \forall j \in \{1,2,3\} \, .
		$$
		Theorem \ref{pro:corrector} entails that the following properties hold:
			\begin{equation} \label{stokescorr}
				\begin{aligned}
					& \nabla \cdot \bs{\omega}_{\! j}^{\eps} =0 \ \ \mbox{in} \ \ \Omega \, ; \qquad \bs{\omega}_{\! j}^{\eps} = \bs{0} \ \ \mbox{in} \ \ \overline{K^{\varepsilon}_{n}} \qquad \forall n \in \{1,...,N(\varepsilon)\} \,  ; \\[6pt]
					& \sup_{\varepsilon \in (0, \varepsilon_{\star})} \| \nabla \bs{\omega}_{\! j}^{\eps}\|_{L^{2}(\Omega)} \leq C \, ; \qquad \bs{\omega}_{\! j}^{\eps} \to \widehat{\bs{e}}_{j} \ \ \ \text{strongly in} \ L^{4}(\Omega; \mathbb{R}^3) \quad \text{as} \quad \varepsilon \to 0^{+} \, .
				\end{aligned}
			\end{equation}
			We multiply both sides of the first identity in \eqref{nsstokes1aprox}$_1$ by $\phi \, \bs{\omega}_{\! j}^{\eps} \in W_{0}^{1,2}(\Omega_{\varepsilon}; \mathbb{R}^{3})$, and integrate by parts in $\Omega_{\varepsilon}$, each term separately, in the following way:
			\begin{equation} \label{bypartsn111}
				\begin{aligned}
					- \int_{\Omega_{\varepsilon}} \Delta \bs{u}_{\varepsilon} \cdot \phi \, \bs{\omega}_{\! j}^{\eps} = \int_{\Omega_{\varepsilon}} \nabla \bs{u}_{\varepsilon} \cdot ( \nabla \phi \otimes \bs{\omega}_{\! j}^{\eps} + \phi \nabla \bs{\omega}_{\! j}^{\eps}) = \int_{\Omega} \nabla \widetilde{\bs{u}}_{\varepsilon} \cdot ( \nabla \phi \otimes \bs{\omega}_{\! j}^{\eps} + \phi \nabla \bs{\omega}_{\! j}^{\eps}) \, .
				\end{aligned}
			\end{equation}
			Concerning the nonlinear term, we simply put
			\begin{equation} \label{bypartsn222}
				\int_{\Omega_{\varepsilon}} ((\nabla \times \bs{u}_{\varepsilon}) \times \bs{u}_{\varepsilon}) \cdot \phi \, \bs{\omega}_{\! j}^{\eps} = \int_{\Omega} ((\nabla \times \widetilde{\bs{u}}_{\varepsilon}) \times \widetilde{\bs{u}}_{\varepsilon}) \cdot \phi \, \bs{\omega}_{\! j}^{\eps} \, .
			\end{equation}
			Regarding the pressure term, from \eqref{stokescorr}$_1$ we infer
			\begin{equation} \label{bypartsn333}
				\begin{aligned}
					\int_{\Omega_{\varepsilon}} \nabla \Phi_{\varepsilon} \cdot \phi \, \bs{\omega}_{\! j}^{\eps} = - \int_{\Omega_{\varepsilon}} \Phi_{\varepsilon} (\nabla \phi \cdot \bs{\omega}_{\! j}^{\eps} ) = - \int_{\Omega} \widetilde{\Phi}_{\varepsilon} (\nabla \phi \cdot \bs{\omega}_{\! j}^{\eps} ) \, .
				\end{aligned}
			\end{equation}
			By adding the identities \eqref{bypartsn111}-\eqref{bypartsn222}-\eqref{bypartsn333}, and then dividing the result by $\mathcal{Z}^{2}_{\varepsilon}$, we obtain
			\begin{equation} \label{bypartsn444}
				\dfrac{1}{\mathcal{Z}_{\varepsilon}} \int_{\Omega} \nabla \widehat{\bs{u}}_{\varepsilon} \cdot ( \nabla \phi \otimes \bs{\omega}_{\! j}^{\eps} + \phi \nabla \bs{\omega}_{\! j}^{\eps}) + \int_{\Omega} ((\nabla \times \widehat{\bs{u}}_{\varepsilon}) \times \widehat{\bs{u}}_{\varepsilon}) \cdot \phi \, \bs{\omega}_{\! j}^{\eps} - \int_{\Omega} \widehat{\Phi}_{\varepsilon} (\nabla \phi \cdot \bs{\omega}_{\! j}^{\eps} ) = \dfrac{1}{\mathcal{Z}_{\varepsilon}^{2}} \int_{\Omega_{\varepsilon}} \bs{f} \cdot \phi \, \bs{\omega}_{\! j}^{\eps} \, , 
			\end{equation} 
			for every $\varepsilon \in I_{*}$, along the sub-sequences \eqref{convergencesn11}. From the properties stated in \eqref{stokescorr} we deduce
			\begin{equation} \label{nsstokeslambda0}
				\lim_{\varepsilon \to 0^{+}} \dfrac{1}{\mathcal{Z}_{\varepsilon}} \int_{\Omega} \nabla \widehat{\bs{u}}_{\varepsilon} \cdot ( \nabla \phi \otimes \bs{\omega}_{\! j}^{\eps} + \phi \nabla \bs{\omega}_{\! j}^{\eps}) = \lim_{\varepsilon \to 0^{+}} \dfrac{1}{\mathcal{Z}_{\varepsilon}^{2}} \int_{\Omega_{\varepsilon}} \bs{f} \cdot \phi \, \bs{\omega}_{\! j}^{\eps} = 0 \, .
			\end{equation}
			In order to handle the remaining integrals appearing in \eqref{bypartsn444}, we write 
			\begin{equation} \label{nsstokeslambdak3}
				\begin{aligned}	
					\int_{\Omega} ((\nabla \times \widehat{\bs{u}}_{\varepsilon}) \times \widehat{\bs{u}}_{\varepsilon}) \cdot \phi \, \bs{\omega}_{\! j}^{\eps} & = \int_{\Omega} ((\nabla \times \widehat{\bs{u}}_{\varepsilon}) \times \widehat{\bs{u}}) \cdot \phi \, \widehat{\bs{e}}_{j} + \int_{\Omega} ((\nabla \times \widehat{\bs{u}}_{\varepsilon}) \times \widehat{\bs{u}}) \cdot \phi (\bs{\omega}_{\! j}^{\eps} - \widehat{\bs{e}}_{j}) \\[6pt]
					& \hspace{4mm} +  \int_{\Omega} ((\nabla \times \widehat{\bs{u}}_{\varepsilon}) \times  (\widehat{\bs{u}}_{\varepsilon} - \widehat{\bs{u}} )) \cdot \phi \, \bs{\omega}_{\! j}^{\eps}  \, ,
				\end{aligned}
			\end{equation}
			alongside with
			\begin{equation} \label{nsstokeslambdak4}
				\int_{\Omega} \widehat{\Phi}_{\varepsilon} (\nabla \phi \cdot \bs{\omega}_{\! j}^{\eps} ) = \int_{\Omega} \widehat{\Phi}_{\varepsilon} (\nabla \phi \cdot \widehat{\bs{e}}_{j} ) + \int_{\Omega} \widehat{\Phi}_{\varepsilon} (\nabla \phi \cdot (\bs{\omega}_{\! j}^{\eps} - \widehat{\bs{e}}_{j} )) \, ,
			\end{equation}
			for every $\varepsilon \in I_{*}$. Consequently, the convergences in \eqref{convergencesn11}-\eqref{stokescorr}$_2$,  entail
			\begin{equation} \label{nsstokeslambdak5}
				\begin{aligned}	
					& \lim_{\varepsilon \to 0^{+}} \int_{\Omega} ((\nabla \times \widehat{\bs{u}}_{\varepsilon}) \times \widehat{\bs{u}}_{\varepsilon}) \cdot \phi \, \bs{\omega}_{\! j}^{\eps} = \int_{\Omega} ((\nabla \times \widehat{\bs{u}}) \times \widehat{\bs{u}}) \cdot \phi \, \widehat{\bs{e}}_{j} \, , \\[6pt]
					& \lim_{\varepsilon \to 0^{+}} \int_{\Omega}\widehat{\Phi}_{\varepsilon} (\nabla \phi \cdot \bs{\omega}_{\! j}^{\eps} ) = \int_{\Omega} \widehat{\Phi} (\nabla \phi \cdot \widehat{\bs{e}}_{j} ) = \int_{\Omega} \widehat{\Phi} (\nabla \cdot ( \phi \, \widehat{\bs{e}}_{j} )) \, .
				\end{aligned}
			\end{equation}
			Observing \eqref{nsstokeslambda0}-\eqref{nsstokeslambdak5}, one can take the limit as $\varepsilon \to 0^{+}$ in \eqref{bypartsn444} to deduce that 
			$$
			\int_{\Omega} ((\nabla \times \widehat{\bs{u}}) \times \widehat{\bs{u}}) \cdot \phi \, \widehat{\bs{e}}_{j} - \int_{\Omega} \widehat{\Phi} (\nabla \cdot ( \phi \, \widehat{\bs{e}}_{j} )) = 0 \qquad \forall \phi \in \mathcal{C}^{\infty}_{0}(\Omega; \mathbb{R}) \, , \quad \forall j \in \{1,2,3\} \, ;
			$$
			this last equality shows that the pair $(\widehat{\bs{u}}, \widehat{\Phi}) \in \mathcal{V}(\Omega) \times L^{2}(\Omega; \mathbb{R})$ satisfies, in distributional form, the following Euler-type equation in $\Omega$:
			\begin{equation} \label{limiteulerlip}
				(\nabla \times \widehat{\bs{u}}) \times \widehat{\bs{u}} + \nabla \widehat{\Phi} = \bs{0} \, , \quad  \nabla\cdot \widehat{\bs{u}}=0 \ \ \mbox{in} \ \ \Omega \, .
			\end{equation}
			Since $(\nabla \times \widehat{\bs{u}}) \times \widehat{\bs{u}} \in L^{3/2}(\Omega; \mathbb{R}^{3})$ (due to the Sobolev embedding $W^{1,2}(\Omega; \mathbb{R}^{3}) \subset L^{6}(\Omega; \mathbb{R}^{3})$), identity \eqref{limiteulerlip} proves that actually $\widehat{\Phi} \in W^{1,3/2}(\Omega; \mathbb{R})$. Recalling \eqref{vectorcalculus}, we set 
			$$
			\widehat{\Phi}_{*} \doteq \widehat{\Phi} - \dfrac{1}{2}|\widehat{\bs{u}}|^{2} \ \ \text{in} \ \ \Omega \, ,
			$$ 
			and thus the pair $(\widehat{\bs{u}}, \widehat{\Phi}_{*}) \in \mathcal{V}(\Omega) \times W^{1,3/2}(\Omega; \mathbb{R})$ satisfies, in strong form, the Euler equation in $\Omega$:
			$$
			(\widehat{\bs{u}} \cdot \nabla) \widehat{\bs{u}} + \nabla \widehat{\Phi}_{*} = \bs{0} \, , \quad  \nabla\cdot \widehat{\bs{u}}=0 \ \ \mbox{in} \ \ \Omega \, .
			$$
			Since $\widehat{\bs{u}} = \bs{0}$ on $\Gamma_{W}$, the Bernoulli Law \cite[Lemma~4]{kapitanskii1983spaces} (see \cite[Theorem~2.2]{amick1984existence} and \cite[Theorem~1]{korobkov2011bernoulli} as well) states that $\widehat{\Phi}_{*}$ must be constant on $\Gamma_{W}$. Then, there exists $\widehat{p}_{*} \in \mathbb{R}$ such that 
			\begin{equation} \label{bernoulli1}
				\widehat{\Phi}= \widehat{p}_{*} \ \ \text{almost everywhere on} \ \ \Gamma_{W} \, .
			\end{equation}
			Now, the embedding $W^{1,2}(\Omega_{\varepsilon}; \mathbb{R}^{3}) \subset L^{6}(\Omega_{\varepsilon}; \mathbb{R}^{3})$ implies $\bs{f} - (\nabla \times \bs{u}_{\varepsilon}) \times \bs{u}_{\varepsilon} \in L^{3/2}(\Omega_{\varepsilon}; \mathbb{R}^{3})$, so from the H\"older and Sobolev inequalities (applied in $\Omega$) we estimate 
			\begin{equation} \label{catta1}
				\begin{aligned}
					\| \bs{f} - (\nabla \times \bs{u}_{\varepsilon}) \times \bs{u}_{\varepsilon} \|_{L^{3/2}(\Omega_{\varepsilon})} & \leq \| \bs{f} \|_{L^{3/2}(\Omega)} + \| (\nabla \times \widetilde{\bs{u}}_{\varepsilon}) \times \widetilde{\bs{u}}_{\varepsilon} \|_{L^{3/2}(\Omega)} \\[3pt]
					& \leq \| \bs{f} \|_{L^{3/2}(\Omega)} + \| \nabla \times \widetilde{\bs{u}}_{\varepsilon} \|_{L^{2}(\Omega)} \| \widetilde{\bs{u}}_{\varepsilon} \|_{L^{6}(\Omega)} \\[3pt]
					& \leq C \left( \| \bs{f} \|_{L^{3/2}(\Omega)} + \| \nabla \bs{u}_{\varepsilon} \|^{2}_{L^{2}(\Omega_{\varepsilon})} \right) \, .
				\end{aligned}
			\end{equation}
			We introduce the following subdomains of $\Omega$:
			$$
			\widetilde{\Omega}_{k} \doteq \left\lbrace \bs{\xi} \in \Omega_{k} \ \vert \ \textup{dist}( \bs{\xi} , \Gamma_{k}) < \delta_{3} \right\rbrace \qquad \forall k \in \{1,\ldots,M\} \, .
			$$
			Observe that the pair $(\bs{u}_{\varepsilon}, \Phi_{\varepsilon}) \in W^{2,3/2}(\Omega_{\varepsilon}; \mathbb{R}^{3}) \times W^{1,3/2}(\Omega_{\varepsilon}; \mathbb{R})$ is also a strong solution to the Stokes system \eqref{nsstokes1aprox}$_1$ in $\Omega_{\varepsilon}$, with a right-hand side given by $\bs{f} - (\nabla \times \bs{u}_{\varepsilon}) \times \bs{u}_{\varepsilon} $. If we apply the same extension argument of \cite[Theorem~3.2]{sperone2022} (see also \cite[Remark 5.4.1]{panasenko2024multiscale}), we can then invoke the usual local regularity results for the Stokes equations (see \cite[Teorema, page 311]{cattabriga1961problema} or \cite[Theorem~IV.4.1]{galdi2011introduction}) and the estimates \eqref{unipress2}-\eqref{catta1} to furnish, for every $k \in \{1,\ldots,M\}$,
			\begin{equation} \label{regularity1}
				\begin{aligned}
					& \| \bs{u}_{\varepsilon} \|_{W^{2,3/2}(\widetilde{\Omega}_{k})} + \| \Phi_{\varepsilon} \|_{W^{1,3/2}(\widetilde{\Omega}_{k})} \\[3pt]
					& \hspace{-4mm} \leq C \left( \| \bs{f} - (\nabla \times \bs{u}_{\varepsilon}) \times \bs{u}_{\varepsilon} \|_{L^{3/2}(\Omega_{\varepsilon})} + \| \nabla \bs{u}_{\varepsilon} \|_{L^{3/2}(\Omega_{\varepsilon})} + \| \Phi_{\varepsilon} \|_{L^{3/2}(\Omega_{\varepsilon})} \right) \\[3pt]
					& \hspace{-4mm} \leq C \left(1 + \| \bs{f} \|_{L^{2}(\Omega)} + \| \nabla \bs{u}_{\varepsilon} \|^{2}_{L^{2}(\Omega_{\varepsilon})} + \| \nabla \bs{u}_{\varepsilon} \|_{L^{2}(\Omega_{\varepsilon})} \right) \, .
				\end{aligned}
			\end{equation}
			We emphasize that, in view of assumption \eqref{extrahyp}, the constant $C>0$ entering the right-hand side of \eqref{regularity1} is independent of $\varepsilon \in I_{*}$, since the strips $\widetilde{\Omega}_{1}, \ldots, \widetilde{\Omega}_{M}$ do not contain any particles. 
			
			Along the weakly convergent sub-sequence \eqref{convergencesn11}$_2$, and in the light of \eqref{regularity1}, we infer that the following sequences are uniformly bounded with respect to $\varepsilon \in I_{*}$:
			$$
			\begin{aligned}
				(\widehat{\bs{v}}_{\varepsilon})_{\varepsilon \in I_{*}} \doteq \left( \dfrac{\widetilde{\bs{u}}_{\varepsilon}}{\mathcal{Z}^{2}_{\varepsilon}} \right)_{\! \! \varepsilon \in I_{*}} \subset W^{2,3/2}(\widetilde{\Omega}_{k}; \mathbb{R}^{3}) \qquad \text{and} \qquad (\widehat{\Phi}_{\varepsilon})_{\varepsilon \in I_{*}} \subset W^{1,3/2}(\widetilde{\Omega}_{k}; \mathbb{R}) \qquad \forall k \in \{1,\ldots,M\} \, .
			\end{aligned}
			$$ 
			Thus, for any $k \in \{1,\ldots,M\}$, there exist $\widehat{\bs{v}}_{k} \in W^{2,3/2}(\widetilde{\Omega}_{k}; \mathbb{R}^{3})$ and $\widehat{\Phi}_{k} \in W^{1,3/2}(\widetilde{\Omega}_{k}; \mathbb{R})$ such that the following convergences hold as $\varepsilon \to 0^{+}$:
			\begin{equation} \label{convergencesiol}
				\begin{aligned} 
					& \widehat{\bs{v}}_{\varepsilon} \to \widehat{\bs{v}}_{k} \ \ \ \text{strongly in} \ W^{1,3/2}(\widetilde{\Omega}_{k};\mathbb{R}^3) \, ;& \qquad &\dfrac{\partial \widehat{\bs{v}}_{\varepsilon}}{\partial \bs{\nu}} \to \dfrac{\partial \widehat{\bs{v}}_{k}}{\partial \bs{\nu}} \ \ \ \text{strongly in} \ L^{1}(\partial \widetilde{\Omega}_{k};\mathbb{R}^3) \, ; \\[6pt]
					& \widehat{\Phi}_{\varepsilon} \rightharpoonup \widehat{\Phi}_{k} \ \ \ \text{weakly in} \ W^{1,3/2}(\widetilde{\Omega}_{k}; \mathbb{R}) \, ;& \qquad & \widehat{\Phi}_{\varepsilon} \to \widehat{\Phi}_{k} \ \ \ \text{strongly in} \ L^{2}(\widetilde{\Omega}_{k}; \mathbb{R}) \, \, ; \\[6pt]
					& \widehat{\Phi}_{\varepsilon} \to \widehat{\Phi}_{k} \ \ \ \text{strongly in} \ L^{1}(\partial \widetilde{\Omega}_{k}; \mathbb{R}) \, ,
				\end{aligned}
			\end{equation}
			along sub-sequences that are not being relabeled, see also \cite[Theorem~6.2]{necas2011direct}. An application of the H\"older inequality (in $\widetilde{\Omega}_{k}$, for $k \in \{1,\ldots,M\}$) readily implies that 
			\begin{equation} \label{convergencesiol00}
				\widehat{\bs{v}}_{k} \equiv \bs{0} \ \ \mbox{in} \ \ \widetilde{\Omega}_{k} \qquad \forall k \in \{1,\ldots,M\} \, ,
			\end{equation}
			since, in virtue of \eqref{convergencesn11}$_1$, $\| \nabla \widehat{\bs{v}}_{\varepsilon} \|_{L^{2}(\Omega)} \to 0$ as $\varepsilon \to 0^{+}$. Furthermore, in view of \eqref{nsstokes1aprox}$_3$-\eqref{nsstokes1aprox}$_4$-\eqref{convergencesiol00}, the convergences in \eqref{convergencesn11}-\eqref{convergencesiol} yield $\widehat{\Phi}_{k} = \widehat{p}_{k}$ on $\Gamma_{k}$, for $k \in \{1,\ldots,M-1\}$, and $\widehat{\Phi}_{M}= 0$ on $\Gamma_{M}$. But, given that $\widehat{\Phi}_\varepsilon \rightharpoonup \widehat{\Phi}$ weakly in $L^{2}(\widetilde{\Omega}_{k}; \mathbb{R})$ as $\varepsilon \to 0^{+}$, for every $k \in \{1,\ldots,M\}$, uniqueness of the weak limit entails $\widehat{\Phi} = \widehat{\Phi}_k$ in $\widetilde{\Omega}_{k}$, for every $k \in \{1,\ldots,M\}$. Consequently, since $\widehat{\Phi} \in W^{1,3/2}(\Omega; \mathbb{R})$,
			\begin{equation} \label{berlawlip0}
			\widehat{\Phi} = \widehat{p}_{k} \ \ \mbox{on} \ \ \Gamma_{k} \qquad \forall k \in \{1,\ldots,M-1\} \, ; \qquad \widehat{\Phi}= 0 \ \ \mbox{on} \ \ \Gamma_{M} \, .
			\end{equation}
			
			In order to prove that $\widehat{p}_{1} = \ldots = \widehat{p}_{M-1} = 0$, we adapt the arguments of \cite[Theorem 3.1]{korobkov2020solvability} to our current configuration (namely, cylindrical ends with general cross sections).
			
			Put $\widehat{p}_{M} \doteq 0$ and fix $k \in \{1,\ldots,M\}$. Denote by $\bs{n}_{k}( \bs s)\in \R^2$ the outward unit normal to $\partial \Theta_k$ at a point $ \bs s \in \partial \Theta_k$, so that $ \bs{n}_{k} \in \mathcal{C}^1(\partial \Theta_k; \R^2)$, as $\partial \Theta_k$ is of class $\mathcal{C}^2$. There exists a sufficiently small $\delta_*>0$ (say $0< \delta_* \ll \text{diam}(\Theta_k)$) such that, after putting
			$$
			\Theta_k^\delta \doteq\{ \bs y \in \Theta_k \,|\, \text{dist}(\bs y, \partial \Theta_k) < \delta \, \} \qquad \forall \delta \in (0,\delta_*] \, ,
			$$
			the mapping 
			$$
			\Psi_{k} :\begin{array}{rcl}
				\partial \Theta_k \times [0,\delta_*] & \to & \overline{\Theta_k^{\delta_*} }\\[4pt]
				(\bs s,t) & \mapsto & 
				\bs y \doteq \bs s- t \bs{n}_{k}( \bs s)
			\end{array}
			$$
			is a $\mathcal{C}^1$-diffeomorphism whose Jacobian can be uniformly bounded (also away from zero) with respect to
			$\delta \in (0,\delta_*)$. Then, given $\delta \in (0, \delta_{*})$, define the \textit{boundary layer} of width $\delta$, included in $\Omega_k$, as
			\begin{equation} \label{bl3}
				\Omega^{\delta}_{k} \doteq \Theta^{\delta}_{k} \times (0,\ell_k) = \Psi_{k}(\partial \Theta_{k} \times (0,\delta))  \times (0,\ell_k) \, ,
			\end{equation}
			as well as the parallel surfaces
			\begin{equation} \label{bl4}
				S_{k}^{\delta} \doteq \{\bs{y} \in \Theta_k \ | \ \text{dist}(\bs{y},\partial\Theta_{k}) = \delta \, \} = \Psi_{k}(\partial \Theta_{k}\times \{ \delta \}) \, .
			\end{equation}
			Since $\widehat{\Phi} \in W^{1,3/2}(\Omega;\mathbb{R})$, for almost every $(\bs{s},\delta,z) \in \partial\Theta_{k} \times (0,\delta_{*}) \times (0,\ell_{k})$ we may write
			$$
			\begin{aligned}
				\widehat{\Phi}(\Psi_{k}(\bs{s},\delta),z) - \widehat{p}_{*} & = \widehat{\Phi}(\Psi_{k}(\bs{s},\delta),z) - \widehat{\Phi}(\Psi_{k}(\bs{s},0),z) \\[6pt]
				& = \int_{0}^{\delta} \dfrac{\partial}{\partial \tau} \widehat{\Phi}(\Psi_{k}(\bs{s},\tau),z) \, d\tau \\[6pt]
				& = - \int_{0}^{\delta} \nabla_{\! \bs{y}} \widehat{\Phi}(\Psi_{k}(\bs{s},\tau),z) \cdot \bs{n}_{k}(\bs{s}) \, d\tau \, ,
			\end{aligned}
			$$
			thereby implying
			$$
			| \widehat{\Phi}(\Psi_{k}(\bs{s},\delta),z) - \widehat{p}_{*} | \leq \int_{0}^{\delta} | \nabla \widehat{\Phi}(\Psi_{k}(\bs{s},\tau),z) | \, d\tau \qquad \forall (\bs{s},\delta,z) \in \partial\Theta_{k} \times (0,\delta_{*}) \times (0,\ell_{k})-\text{a.e.}.
			$$
			Upon integration of this last inequality over $(\bs{s},z) \in \partial\Theta_{k} \times (0,\ell_{k})$ we notice, as a consequence of \eqref{bl3}-\eqref{bl4} and a change of variables, that
			\begin{equation} \label{finalp-1}
				\int\limits_{S_{k}^{\delta} \times (0,\ell_k)} | \widehat{\Phi} - \widehat{p}_{*} | \leq C \int_{\Omega^{\delta}_{k}}|\nabla\widehat{\Phi}| \qquad \forall \delta \in (0, \delta_{*}) \, ,
			\end{equation}
			where $C>0$ is independent of $\delta \in (0, \delta_{*})$. 
			
			On the other hand, recalling that $\widehat{\bs{u}}$ vanishes on $\partial\Theta_k\times (0,\ell_k)$, the following Poincaré inequality holds:
			\begin{equation}\label{poincare}
				\|\widehat{\bs{u}}\|_{L^2(\Omega^{\delta}_{k})}\leq C \delta \|\nabla\widehat{\bs{u}}\|_{L^2(\Omega^{\delta}_{k})} \qquad \forall \delta \in (0, \delta_{*}) \, ,
			\end{equation}
			where $C>0$ is independent of $\delta \in (0, \delta_{*})$. As the Euler equation \eqref{limiteulerlip} holds in $\Omega^{\delta}_{k}$, from the H\"older inequality and \eqref{poincare} we further infer
			\begin{equation} \label{eulerdelta}
				\int_{\Omega^{\delta}_{k}}|\nabla\widehat{\Phi}|= \int_{\Omega^{\delta}_{k}}|(\nabla \times \widehat{\bs{u}}) \times \widehat{\bs{u}}|\leq C \|\widehat{\bs{u}}\|_{L^2(\Omega^{\delta}_{k})}\|\nabla\widehat{\bs{u}}\|_{L^2(\Omega^{\delta}_{k})}\leq C\delta\|\nabla\widehat{\bs{u}}\|_{L^2(\Omega^{\delta}_{k})}^2 \qquad \forall \delta \in (0, \delta_{*}) \, .
			\end{equation}
			Now, the Coarea formula \cite[Theorem 1.5.17]{korobkov2024steady} entails
			$$
			\int\limits_{0}^{\delta} \left( \, \int\limits_{S_{k}^{t} \times (0,\ell_k)} | \nabla\widehat{\Phi} | \right) dt = \int_{\Omega^{\delta}_{k}}|\nabla\widehat{\Phi}| \qquad \forall \delta \in (0, \delta_{*}) \, ,
			$$
			and thus, the Mean Value Theorem for Lebesgue Integrals \cite[Theorem 3.1]{pouso2011mean} ensures
			$$
			\left| \left\{ t \in (0, \delta) \ \Bigg| \ \int\limits_{S_{k}^{t} \times (0,\ell_k)} | \nabla\widehat{\Phi} | \leq \dfrac{1}{\delta} \int_{\Omega^{\delta}_{k}}|\nabla\widehat{\Phi}| \, \right\} \right| > 0 \qquad \forall \delta \in (0, \delta_{*}) \, .
			$$
			Therefore, for every $\delta \in (0, \delta_{*})$ we can find $h_{\delta} \in (0,\delta)$ such that
			\begin{equation} \label{deltaeps1}
				\int\limits_{S_{k}^{h_{\delta}} \times (0,\ell_k)} | \nabla\widehat{\Phi} | \leq \dfrac{1}{\delta} \int_{\Omega^{\delta}_{k}}|\nabla\widehat{\Phi}| \leq C \|\nabla\widehat{\bs{u}}\|_{L^2(\Omega^{\delta}_{k})}^2 \, .
			\end{equation}
			as a consequence of \eqref{eulerdelta}. Now, take $\delta \in (0, \delta_{*})$ and $h_{\delta} \in (0,\delta)$ satisfying \eqref{deltaeps1}. For almost every $(\bs{y},z) \in S_{k}^{h_{\delta}} \times (0,\ell_k)$ we may integrate in local coordinates to obtain
			\begin{equation} \label{deltaeps2}
				\widehat{\Phi}(\bs{y},z) - \widehat{\Phi}(\bs{y},\ell_{k}) = - \int_{z}^{\ell_{k}} \dfrac{\partial \widehat{\Phi}}{\partial z}(\bs{y},\tau) \, d\tau \, .
			\end{equation}
			Since $(\bs{y},\ell_{k}) \in \Gamma_{k}$, from \eqref{berlawlip0}-\eqref{deltaeps2} we deduce
			$$
			| \widehat{\Phi}(\bs{y},z) - \widehat{p}_{k} | \leq \int_{0}^{\ell_{k}} | \nabla \widehat{\Phi}(\bs{y},\tau) | d\tau \qquad \forall (\bs{y},z) \in S_{k}^{h_{\delta}} \times (0,\ell_k)-\text{a.e.} \, .
			$$
			Upon integration of this last inequality over $S_{k}^{h_{\delta}} \times (0,\ell_k)$ we notice, in view of \eqref{deltaeps1}, that
			\begin{equation} \label{finalp0}
				\int\limits_{S_{k}^{h_{\delta}} \times (0,\ell_k)} | \widehat{\Phi} - \widehat{p}_{k} | \leq \ell_{k} \int\limits_{S_{k}^{h_{\delta}} \times (0,\ell_k)} | \nabla\widehat{\Phi} | \leq \dfrac{\ell_k}{\delta} \int_{\Omega^{\delta}_{k}}|\nabla\widehat{\Phi}| \leq C \|\nabla\widehat{\bs{u}}\|_{L^2(\Omega^{\delta}_{k})}^2 \qquad \forall \delta \in (0, \delta_{*})  \, .
			\end{equation} 
			
			In the light of \eqref{finalp-1}-\eqref{finalp0}, take $\delta \in (0, \delta_{*})$ and $h_{\delta} \in (0,\delta)$ satisfying \eqref{deltaeps1}. Since for almost every $(\bs{y},z) \in S_{k}^{h_{\delta}} \times (0,\ell_k)$ we have
			$$
			| \widehat{p}_{*} - \widehat{p}_{k} | \leq | \widehat{p}_{*} - \widehat{\Phi}(\bs{y},z) | + | \widehat{p}_{k} - \widehat{\Phi}(\bs{y},z)| \, ,
			$$
			we can integrate this last inequality $S_{k}^{h_{\delta}} \times (0,\ell_k)$ to obtain
			$$
			\ell_{k} | S_{k}^{h_{\delta}} | | \widehat{p}_{*} - \widehat{p}_{k} | \leq C \left( \|\nabla\widehat{\Phi}\|_{L^1(\Omega^{h_{\delta}}_{k})} + \|\nabla\widehat{\bs{u}}\|_{L^2(\Omega^{\delta}_{k})}^2 \right) \qquad \forall \delta \in (0, \delta_{*})  \, ,
			$$
	where one can take the limit as $\delta \to 0^{+}$ to deduce $\ell_{k} | \partial \Theta_{k} | | \widehat{p}_{*} - \widehat{p}_{k} | \leq 0$ for every $k \in \{1,\ldots,M\}$. Therefore,
			\begin{equation} \label{finalp}
			\widehat{p}_{*} = \widehat{p}_{1} = \ldots = \widehat{p}_{M} = 0 \, .
			\end{equation} 
			A contradiction is reached between \eqref{vlambdalimn}-\eqref{finalp}, and consequently, the norms $\| \nabla \bs{u}_{\varepsilon} \|_{L^{2}(\Omega_{\varepsilon})}$ must be uniformly bounded with respect to $\varepsilon \in I_{*}$. Combined with \eqref{epsbyparts6}-\eqref{unipress2}, this concludes the proof.
		\end{proof}
		
		\begin{remark} \label{rem1}
			The additional assumption \eqref{extrahyp}, which requires the existence of thin strips near $\Gamma_{k}$ that are never perforated in the process as $\varepsilon \to 0^{+}$, is only invoked to specify the boundary values of the limit Bernoulli pressure $\widehat{\Phi} \in W^{1,3/2}(\Omega)$ on $\Gamma_{k}$, for $k \in \{1,\ldots,M\}$. Such information cannot be directly extracted from the weak convergence in $L^{2}(\Omega; \mathbb{R})$, while it does not seem straightforward to build a $W^{1,3/2}(\Omega; \mathbb{R})$-uniform extension for the Bernoulli pressure $\Phi_{\varepsilon}$ inside the holes; in fact, following the approach of \cite[Lemma~1.7]{masmoudi2002homogenization} (based on local regularity estimates for the Stokes problem) one obtains
			$$
			\| \bs{u}_{\varepsilon} \|_{W^{2,3/2}(\Omega_{\varepsilon})} + \| \Phi_{\varepsilon} \|_{W^{1,3/2}(\Omega_{\varepsilon})} \leq \dfrac{C}{\sqrt{\varepsilon}} \left(1 + \| \bs{f} \|_{L^{2}(\Omega)} + \| \nabla \bs{u}_{\varepsilon} \|^{2}_{L^{2}(\Omega_{\varepsilon})} + \| \nabla \bs{u}_{\varepsilon} \|_{L^{2}(\Omega_{\varepsilon})} \right) \qquad \forall \varepsilon \in I_{*} \, , 
			$$
			for some constant $C>0$ independent of $\varepsilon \in I_{*}$. It is left open the possibility of recovering the result of Theorem~\ref{epslevel} without hypothesis \eqref{extrahyp}.
		\end{remark}
	
		The second main result of this section provides uniform bounds (with respect to $\varepsilon \in I_{*}$) for the solutions of the prescribed pressure drop problem \eqref{nsstokespd1}. Compared to Theorem~\ref{epslevel}, in this case the additional assumption \eqref{extrahyp} is not required and the corresponding proof is considerably simpler.
		
		\begin{theorem} \label{epslevelpd}
			Let $\Omega \subset \mathbb{R}^{3}$ be an admissible domain, $\bs{f} \in L^{2}(\Omega; \mathbb{R}^{3})$, $(\bs{\mu}^{\varepsilon}_{n})^{N(\eps)}_{n=1} \subset \mathbb{R}^{3}$ and $(p_{k})^{M}_{k=1} \subset \mathbb{R}$. There exists (at least) one weak solution $\bs{u}_{\varepsilon} \in W^{2,2}(\Omega_{\varepsilon}; \mathbb{R}^{3}) \cap \mathcal{V}(\Omega_{\varepsilon})$ of the prescribed pressure drop problem \eqref{nsstokespd1} and an associated Bernoulli pressure $\Phi_{\varepsilon} \in W^{1,2}(\Omega_{\varepsilon}; \mathbb{R})$ such that the pair $(\bs{u}_{\varepsilon},\Phi_{\varepsilon})$ satisfies \eqref{nsstokespd1} in strong form. 
			
			Furthermore, under assumption \eqref{kineticuni}, there exists $\varepsilon_{\circ} \in I_{*}$ for which the uniform bound
			\begin{equation} \label{uboundpd}
				\sup_{\varepsilon \in (0, \varepsilon_{\circ})} \left( \| \nabla \bs{u}_{\varepsilon} \|_{L^{2}(\Omega_{\varepsilon})} + \| \Phi_{\varepsilon} \|_{L^{2}(\Omega_{\varepsilon})} \right) \leq C_{*} \, ,
			\end{equation}
			holds for some constant $C_{*} > 0$ that depends on $\Omega$, $\bs{f}$, $(p_{k})^{M}_{k=1}$, $\mathcal{S}_{*}$ and $\{ \delta_{i} \}^{2}_{i=0}$.
		\end{theorem}
		\noindent
		\begin{proof}
			In what follows, $C > 0$ will always denote a generic constant that depends on $\Omega$, $(p_{k})^{M}_{k=1}$, $\mathcal{S}_{*}$ and $\{ \delta_{i} \}^{2}_{i=0}$ (independently of $\varepsilon \in I_{*}$), but that may change from line to line. 
			\par
			Given $\bs{f} \in L^{2}(\Omega; \mathbb{R}^{3})$,  $(\bs{\mu}^{\varepsilon}_{n})^{N(\eps)}_{n=1} \subset \mathbb{R}^{3}$,  $(p_{k})^{M}_{k=1} \subset \mathbb{R}$ and $\varepsilon \in I_{*}$, a direct extension of \cite[Theorem~3.2]{korobkov2020solvability} ensures the existence of at least one weak solution $\bs{u}_{\varepsilon} \in \mathcal{V}(\Omega_{\varepsilon})$ of problem \eqref{nsstokespd1}. Then, the same argument of \cite[Theorem~3.2]{sperone2022} can be applied to deduce $\bs{u}_{\varepsilon} \in W^{2,2}(\Omega_{\varepsilon}; \mathbb{R}^{3})$ and the existence of an associated pressure $\Phi_{\varepsilon} \in W^{1,2}(\Omega_{\varepsilon}; \mathbb{R})$ satisfying
			\begin{equation}\label{nsstokes1aproxpd}
				\left\{
				\begin{aligned}
					& -\Delta \bs{u}_{\varepsilon}+ (\nabla \times \bs{u}_{\varepsilon}) \times \bs{u}_{\varepsilon} + \nabla \Phi_{\varepsilon}=\bs{f} \, ,\ \quad  \nabla\cdot \bs{u}_{\varepsilon}=0 \ \ \mbox{in} \ \ \Omega_{\varepsilon} \, ; \\[3pt]
					& \bs{u}_{\varepsilon}=\bs{0} \ \ \mbox{on} \ \ \Gamma_{W} \, ; \quad \bs{u}_{\varepsilon}=\bs{\mu}^{\varepsilon}_{n} \ \ \mbox{on} \ \ \partial K^{\varepsilon}_{n} \qquad \forall n \in \{1,\ldots,N(\varepsilon)\} \, ; \\[3pt]
					& \bs{u}_{\varepsilon} \times \bs{\nu} = \bs{0} \, , \quad - \dfrac{\partial}{\partial \bs{\nu}}(\bs{u}_{\varepsilon} \cdot \bs{\nu}) + \Phi_{\varepsilon} = p_{k} \qquad \forall k \in \{1,\ldots,M\} \, \, ,
				\end{aligned}
				\right.
			\end{equation}
			in strong form, see also \cite[Section 2]{korobkov2020solvability}. Let $\bs{A}_{\varepsilon} \in \mathcal{C}^{\infty}(\overline{\Omega}; \mathbb{R}^{3})$ be the vector field arising from Lemma~\ref{smallfield} and define $\bs{v}_{\varepsilon} \doteq \bs{u}_{\varepsilon} - \bs{A}_{\varepsilon}$ in $\Omega_{\varepsilon}$, so that $\bs{v}_{\varepsilon} \in \mathcal{W}(\Omega_{\varepsilon})$. In view of \eqref{smallfieldpropbounds}$_1$, it suffices to prove the following uniform bound:
			\begin{equation} \label{uboundpdshift}
				\sup_{\varepsilon \in (0, \varepsilon_{\circ})} \left( \| \nabla \bs{v}_{\varepsilon} \|_{L^{2}(\Omega_{\varepsilon})} + \| \Phi_{\varepsilon} \|_{L^{2}(\Omega_{\varepsilon})} \right) < + \infty \, .
			\end{equation}
			As a consequence of \eqref{nstokesdebilpd}, we notice that $\bs{v}_{\varepsilon}$ satisfies the following weak formulation:
			\begin{equation} \label{nstokesdebilpdshift}
				\begin{aligned}
					& \int_{\Omega_{\varepsilon}} \nabla \bs{v}_{\varepsilon} : \nabla \bs{\varphi} + \int_{\Omega_{\varepsilon}} \nabla \bs{A}_{\varepsilon} : \nabla \bs{\varphi} + \int_{\Omega_{\varepsilon}} ((\nabla \times \bs{v}_{\varepsilon}) \times \bs{v}_{\varepsilon}) \cdot \bs{\varphi}+ \int_{\Omega_{\varepsilon}} ((\nabla \times \bs{v}_{\varepsilon}) \times \bs{A}_{\varepsilon}) \cdot \bs{\varphi}   \\[6pt]
					& \hspace{-4mm} + \int_{\Omega_{\varepsilon}} ((\nabla \times \bs{A}_{\varepsilon}) \times \bs{v}_{\varepsilon}) \cdot \bs{\varphi} + \int_{\Omega_{\varepsilon}} ((\nabla \times \bs{A}_{\varepsilon}) \times \bs{A}_{\varepsilon}) \cdot \bs{\varphi} + \sum_{k=1}^{M} p_{k} \left( \int_{\Gamma_{k}} \bs{\varphi} \cdot \bs{\nu} \right) = \int_{\Omega_{\varepsilon}} \bs{f} \cdot \bs{\varphi} \, ,
				\end{aligned}
			\end{equation}
			for every $\bs{\varphi} \in \mathcal{W}(\Omega_{\varepsilon})$. As in the proof of Theorem~\ref{epslevel}, we put
			\begin{equation} \label{extensionpd}
				\widetilde{\bs{v}}_{\varepsilon} \doteq 
				\begin{cases}
					\bs{v}_{\varepsilon} & \quad \text{in} \ \ \Omega_{\varepsilon} \, ,\\[3pt]
					\bs{0} & \quad  \text{in} \ \ \overline{K_{\varepsilon}}  \, ,
				\end{cases}
				\qquad
				\text{and}
				\qquad 
				\widetilde{\Phi}_{\varepsilon} \doteq 
				\begin{cases}
					\Phi_{\varepsilon} & \quad \text{in} \ \ \Omega_{\varepsilon} \, ,\\[3pt]
					0 & \quad  \text{in} \ \ \overline{K_{\varepsilon}}  \, ,
				\end{cases}
			\end{equation}
			and notice that $\widetilde{\Phi}_{\varepsilon} \in L^{2}(\Omega; \mathbb{R})$ and $\widetilde{\bs{v}}_{\varepsilon} \in \mathcal{V}(\Omega)$, see \eqref{espaciosstar}. By taking $\bs{\varphi} = \bs{v}_{\varepsilon}$ in the weak formulation \eqref{nstokesdebilpdshift} we obtain
			\begin{equation} \label{pd1}
				\begin{aligned}
					& \| \nabla \bs{v}_{\varepsilon}  \|^{2}_{L^{2}(\Omega_{\varepsilon})} + \int_{\Omega_{\varepsilon}} \nabla \bs{A}_{\varepsilon} : \nabla \bs{v}_{\varepsilon} + \int_{\Omega} ((\nabla \times \widetilde{\bs{v}}_{\varepsilon}) \times \bs{A}_{\varepsilon}) \cdot \widetilde{\bs{v}}_{\varepsilon}   \\[6pt]
					& \hspace{-4mm} + \int_{\Omega} ((\nabla \times \bs{A}_{\varepsilon}) \times \bs{A}_{\varepsilon}) \cdot \widetilde{\bs{v}}_{\varepsilon} + \sum_{k=1}^{M} p_{k} \left( \int_{\Gamma_{k}} \bs{v}_{\varepsilon} \cdot \bs{\nu} \right) = \int_{\Omega} \bs{f} \cdot \widetilde{\bs{v}}_{\varepsilon} \, .
				\end{aligned}
			\end{equation}
			Applying the H\"older, Sobolev and Poincaré inequalities (in $\Omega$), as well as \eqref{smallfieldpropbounds}, we can estimate:
			\begin{equation} \label{ppdbound0}
				\begin{aligned}
					& \left| \int_{\Omega_{\varepsilon}} \nabla \bs{A}_{\varepsilon} : \nabla \bs{v}_{\varepsilon} \right| \leq C \| \nabla \bs{v}_{\varepsilon}  \|_{L^{2}(\Omega_{\varepsilon})} \, , \qquad \left| \int_{\Omega} \bs{f} \cdot \widetilde{\bs{v}}_{\varepsilon} \right| \leq C \| \bs{f}  \|_{L^{2}(\Omega)} \| \nabla \bs{v}_{\varepsilon}  \|_{L^{2}(\Omega_{\varepsilon})} \, , \\[6pt]
					& \left| \int_{\Omega} ((\nabla \times \widetilde{\bs{v}}_{\varepsilon}) \times \bs{A}_{\varepsilon}) \cdot \widetilde{\bs{v}}_{\varepsilon}  \right| \leq \| \nabla \times \widetilde{\bs{v}}_{\varepsilon} \|_{L^{2}(\Omega)} \| \widetilde{\bs{v}}_{\varepsilon} \|_{L^{4}(\Omega)} \| \bs{A}_{\varepsilon} \|_{L^{4}(\Omega)} \leq C \varepsilon^{3/4}  \| \nabla \bs{v}_{\varepsilon}  \|^{2}_{L^{2}(\Omega_{\varepsilon})} \, , \\[6pt]
					& \left| \int_{\Omega} ((\nabla \times \bs{A}_{\varepsilon}) \times \bs{A}_{\varepsilon}) \cdot \widetilde{\bs{v}}_{\varepsilon}  \right| \leq \| \nabla \times \bs{A}_{\varepsilon} \|_{L^{2}(\Omega)}  \| \bs{A}_{\varepsilon} \|_{L^{4}(\Omega)} \| \widetilde{\bs{v}}_{\varepsilon} \|_{L^{4}(\Omega)} \leq C \| \nabla \bs{v}_{\varepsilon}  \|_{L^{2}(\Omega_{\varepsilon})} \, .
				\end{aligned}
			\end{equation}
			On the other hand, the trace inequality (applied in $\Omega$) gives us
			\begin{equation} \label{pd3}
				\left| \int_{\Gamma_{k}} \bs{v}_{\varepsilon} \cdot \bs{\nu} \right| \leq \| \widetilde{\bs{v}}_{\varepsilon}  \|_{L^{1}(\Gamma_{k})} \leq C \| \widetilde{\bs{v}}_{\varepsilon}  \|_{L^{2}(\Gamma_{k})} \leq C \| \nabla \widetilde{\bs{v}}_{\varepsilon} \|_{L^{2}(\Omega_{\varepsilon})} \qquad \forall k \in \{1,\ldots,M\} \, .
			\end{equation}
			Then, select a sufficiently small $\varepsilon_{\circ} \in I_{*}$ such that
			\begin{equation} \label{pd4}
				\varepsilon^{3/4} < \dfrac{1}{C} \qquad \forall \varepsilon \in (0, \varepsilon_{\circ}) \, ,
			\end{equation}
			with $C>0$ being the constant entering the right-hand side of \eqref{ppdbound0}$_2$. After plugging \eqref{ppdbound0}-\eqref{pd3}-\eqref{pd4} into \eqref{pd1} we obtain
			\begin{equation} \label{pdu}
				\| \nabla \bs{v}_{\varepsilon} \|_{L^{2}(\Omega_{\varepsilon})} \leq C \left(\| \bs{f} \|_{L^{2}(\Omega)} + 1 \right) \qquad \forall \varepsilon \in (0, \varepsilon_{\circ}) \, .
			\end{equation}
			Concerning the pressure $\Phi_{\varepsilon} \in L^{2}(\Omega_{\varepsilon})$, from Lemma \ref{bogtype} we find a vector field $\bs{J}_{\! \varepsilon} \in W^{1,2}(\Omega_{\varepsilon}; \mathbb{R}^3)$ satisfying
			\begin{equation}
	\left\{
\begin{aligned}
	& \nabla \cdot \bs{J}_{\! \varepsilon} = \Phi_{\varepsilon}  \ \ \mbox{in} \ \ \Omega_{\varepsilon} \, ; \qquad \bs{J}_{\! \varepsilon} = \bs{0} \ \ \mbox{on} \ \ \Gamma_{W} \cup \partial K_{\varepsilon} \cup \Gamma_{k} \quad \forall k \in \{1, \ldots, M-1\} \, ; \\[5pt]
	& \bs{J}_{\! \varepsilon} \times \bs{\nu} = \bs{0} \ \ \mbox{on} \ \ \Gamma_{M} \, ; \qquad \| \nabla \bs{J}_{\! \varepsilon} \|_{L^{2}(\Omega_{\varepsilon})} \leq C \| \Phi_{\varepsilon}  \|_{L^{2}(\Omega_{\varepsilon})} \, .
\end{aligned}
\right.
			\end{equation}
			If we multiply the first identity in \eqref{nsstokes1aproxpd}$_1$ by $\bs{J}_{\! \varepsilon}$ and integrate by parts in $\Omega_{\varepsilon}$, arguing exactly as in \eqref{epsbyparts7}-\eqref{bypartsje3}-\eqref{unipress2} we derive the bound
			\begin{equation} \label{unipresspd}
				\| \Phi_{\varepsilon} \|_{L^{2}(\Omega_{\varepsilon})} \leq C \left( \| \nabla \bs{u}_{\varepsilon} \|^{2}_{L^{2}(\Omega_{\varepsilon})} + \| \nabla \bs{u}_{\varepsilon} \|_{L^{2}(\Omega_{\varepsilon})} + \| \bs{f} \|_{L^{2}(\Omega)} \right) \, .
			\end{equation}
			A combination of \eqref{uboundpdshift}-\eqref{pdu}-\eqref{unipresspd} finishes the proof.
		\end{proof}

\subsection{Proof of Theorem \ref{thm_hom}}		
We conclude this section with the proof of Theorem \ref{thm_hom} (the corresponding proof of Theorem \ref{thm_hom_pd} is omitted for the sake of brevity, since it resembles the proof of Theorem~\ref{thm_hom} with minor modifications).
\newline 
\newline
\noindent
\begin{proof}
We put
$$
\mathcal{U}_{*}(\Omega_{\eps}) \doteq \left\lbrace \bs{v} \in W^{1,2}(\Omega_{\varepsilon}; \mathbb{R}^{3}) \ \Bigg\rvert \ 
\begin{aligned}
	& \hspace{3.1cm} \bs{v} = \bs{0} \ \ \mbox{on} \ \ \Gamma_{W} \cup \partial K_{\varepsilon} \, ; \, \\[3pt]
	&  \bs{v} \times \bs{\nu} = \bs{0} \ \ \mbox{on} \ \ \Gamma_{k} \quad \text{and} \quad \int_{\Gamma_{k}} \bs{v} \cdot \bs{\nu} = 0 \qquad \forall k \in \{1,\ldots, M\} \
\end{aligned}
\right\rbrace 
$$
\bigskip
$$
\mathcal{U}_{*}(\Omega) \doteq \left\lbrace \bs{v} \in W^{1,2}(\Omega; \mathbb{R}^{3}) \ \Bigg\rvert \ 
\begin{aligned}
	& \hspace{3.1cm} \bs{v} = \bs{0} \ \ \mbox{on} \ \ \Gamma_{W} \, ; \, \\[3pt]
	&  \bs{v} \times \bs{\nu} = \bs{0} \ \ \mbox{on} \ \ \Gamma_{k} \quad \text{and} \quad \int_{\Gamma_{k}} \bs{v} \cdot \bs{\nu} = 0 \qquad \forall k \in \{1,\ldots, M\} \
\end{aligned}
\right\rbrace 
$$
and $\mathcal{V}(\Omega) \subset W^{1,2}(\Omega; \mathbb{R}^{3})$ as in \eqref{espaciosstar}. Thus, $\mathcal{U}_{*}(\Omega)$ and $\mathcal{V}(\Omega)$ are closed subspaces of $W^{1,2}(\Omega; \mathbb{R}^{3})$, with $(\widetilde{\bs{u}}_{\eps})_{\eps \in I_{*}} \subset \mathcal{V}(\Omega)$. If $\bs{u} \in W^{1,2}(\Omega; \mathbb{R}^{3})$ is a weak accumulation point of $(\widetilde{\bs{u}}_{\eps})_{\eps \in I_{*}}$ in $W^{1,2}(\Omega; \mathbb{R}^{3})$, then necessarily $\bs{u} \in \mathcal{V}(\Omega)$; the boundary conditions for $\bs{u}$ in \eqref{homog.eqpnf}$_2$-\eqref{homog.eqpnf}$_3$-\eqref{homog.eqpnf}$_4$ are thus verified. Furthermore, by the compactness of the embeddings $W^{1,2}(\Omega; \mathbb{R}^{3}) \subset L^{4}(\Omega; \mathbb{R}^{3})$ and $W^{1,2}(\Omega; \mathbb{R}^{3}) \subset L^{1}(\partial \Omega; \mathbb{R}^{3})$, we may assume that the following convergences hold as $\eps \to 0^{+}$:
\begin{equation} \label{convergencesn22}
\widetilde{\bs{u}}_{\eps} \to \bs{u} \ \ \ \text{strongly in} \ L^{4}(\Omega; \mathbb{R}^{3}) \, ;  \qquad \widetilde{\bs{u}}_{\eps} \to \bs{u} \ \ \ \text{strongly in} \ L^{1}(\partial \Omega; \mathbb{R}^{3}) \, .
\end{equation}
Recall, from \eqref{nsstokes1}$_5$, 
$$
\int_{\Gamma_{k}} \widetilde{\bs{u}}_{\eps} \cdot \bs{\nu} = \mathcal{F}_{k} \qquad \forall \varepsilon \in I_{*} \, , \ \ \forall k \in \{1,\ldots,M\} \, ,
$$
so that \eqref{convergencesn22}$_2$ gives us
$$
\left| \int_{\Gamma_{k}} \bs{u} \cdot \bs{\nu} -  \mathcal{F}_{k} \right| = \left| \int_{\Gamma_{k}} (\bs{u} - \widetilde{\bs{u}}_{\eps}) \cdot \bs{\nu} \right| \leq \| \widetilde{\bs{u}}_{\eps} - \bs{u} \|_{L^{1}(\Gamma_{k})} \leq \| \widetilde{\bs{u}}_{\eps} - \bs{u} \|_{L^{1}(\partial \Omega)} \to 0 \quad \text{as} \ \ \varepsilon \to 0^{+} \, .
$$
Therefore, $\bs{u}$ verifies the flux condition \eqref{homog.eqpnf}$_5$.

Now, take any vector field $\bs{\varphi} \in \mathcal{U}_{*}(\Omega) \cap \mathcal{C}^\infty(\overline{\Omega}; \mathbb{R}^{3})$; owing to Theorem \ref{pro:corrector}, $\mathcal R_\eps(\bs \varphi) \in \mathcal{U}_{*}(\Omega_\eps)$, since $\mathcal R_\eps(\bs \varphi) = \bs{\varphi}$ on $\partial \Omega$. Consequently, the weak formulation of \eqref{nsstokes1} including the Bernoulli pressure (see \cite[Section 4]{korobkov2020solvability}) reads,
\begin{equation} \label{finlimit0}
\int_{\Omega} \nabla \widetilde{\bs{u}}_{\eps} : \nabla \mathcal R_\eps(\bs \varphi) + \int_{\Omega} ((\nabla \times \widetilde{\bs{u}}_{\eps}) \times \widetilde{\bs{u}}_{\eps}) \cdot \mathcal{R}_{\eps}(\bs{\varphi}) - \int_{\Omega} \widetilde{\Phi}_{\eps} ( \nabla \cdot \mathcal R_\eps(\bs \varphi)) = \int_{\Omega} \bs{f} \cdot \mathcal R_\eps(\bs \varphi) \, ,
\end{equation}
for every $\varepsilon \in I_{*}$. It follows from \eqref{W_eps_resistance} that
\begin{equation} \label{finlimit1}
\lim_{\varepsilon \to 0^{+}} \left( \int_{\Omega} \nabla \widetilde{\bs{u}}_{\eps} : \nabla \mathcal{R}_{\eps}(\bs{\varphi}) - \int_{\Omega} \widetilde{\Phi}_{\eps} ( \nabla \cdot \mathcal R_\eps(\bs \varphi)) \right) = \int_\Omega \nabla \bs u : \nabla \bs \varphi + \int_{\Omega} (\mathcal{G} \bs{u} - \bs{\mathcal{J}}) \cdot \bs{\varphi} - \int_{\Omega} \Phi ( \nabla \cdot \bs \varphi) \, ,
\end{equation}
while \eqref{propreps0}$_4$ entails that $R_\eps(\bs \varphi) \to \bs{\varphi}$ strongly in $L^{2}(\Omega; \mathbb{R}^{3})$ as $\varepsilon \to 0^{+}$, so that
\begin{equation} \label{finlimit2}
	\lim_{\varepsilon \to 0^{+}} \int_{\Omega} \bs{f} \cdot \mathcal R_\eps(\bs \varphi) = \int_{\Omega} \bs{f} \cdot \bs{\varphi} \, .
\end{equation}
Additionally, arguing as in \eqref{nsstokeslambdak3}, we can invoke the strong convergences in \eqref{propreps0}$_4$-\eqref{convergencesn22}$_1$ to deduce
\begin{equation} \label{finlimit3}
\lim_{\varepsilon \to 0^{+}} \int_{\Omega} ((\nabla \times \widetilde{\bs{u}}_{\eps}) \times \widetilde{\bs{u}}_{\eps}) \cdot \mathcal{R}_{\eps}(\bs{\varphi}) = \int_{\Omega} ((\nabla \times \bs{u}) \times \bs{u})  \cdot  \bs{\varphi}.
\end{equation}
Therefore, appealing to \eqref{finlimit1}-\eqref{finlimit2}-\eqref{finlimit3}, we can take the limit as $\varepsilon \to 0^{+}$ on both sides of identity \eqref{finlimit0} to deduce, for every $\bs{\varphi} \in \mathcal{U}_{*}(\Omega) \cap \mathcal{C}^\infty(\overline{\Omega}; \mathbb{R}^{3})$,
\begin{equation}
\int_{\Omega} \nabla \bs{u} : \nabla\bs{\varphi} + \int_{\Omega} ((\nabla \times \bs{u}) \times \bs{u}) \cdot  \bs{\varphi}  + \int_{\Omega} (\mathcal{G} \bs{u} - \bs{\mathcal{J}}) \cdot \bs{\varphi} - \int_{\Omega} \Phi ( \nabla \cdot \bs \varphi) = \int_{\Omega} \bs{f} \cdot   \bs{\varphi} \, .
\end{equation}
By density, this shows that $(\bs{u}, \Phi)$ is a weak solution to the prescribed net flux problem \eqref{homog.eqpnf} in $\Omega$.
\end{proof}
		
		\appendix
		\section{Appendix}\label{appendix_whole}

		\subsection{The Stokes problem in the exterior of a single particle}\label{appendix_sec}
		Let $\Omega \subset \mathbb{R}^{3}$ be an admissible domain. Given $\varepsilon \in I_{*}$, $n \in \{1,\ldots,N(\varepsilon)\}$ and $\bs{\varphi} \in W^{1,2}(\Omega; \mathbb{R}^{3})$, consider the unique generalized solution $(\bs{\psi}^{\varepsilon}_{n}, p^{\varepsilon}_{n}) \in D^{1,2}(\R^3 \setminus \overline{K_{n}^{\eps}}; \R^{3}) \times L^{2}(\R^3 \setminus \overline{K_{n}^{\eps}}; \R)$ (in the sense of \cite[Definition~V.1.1]{galdi2011introduction}, and where $D^{1,2}(\R^3 \setminus \overline{K_{n}^{\eps}}; \R^{3})$ is defined in \eqref{D^1,2}) to the following Stokes problem in the exterior of $K_{n}^{\eps}$:
		\begin{equation}\label{extstokes00}
			\left\{
			\begin{aligned}
				& -\Delta \bs{\psi}^{\varepsilon}_{n} + \nabla p^{\varepsilon}_{n}  = \bs{0} \, , \quad  \nabla\cdot \bs{\psi}^{\varepsilon}_{n}  = 0 \ \ \mbox{in} \ \ \R^3 \setminus \overline{K_{n}^{\eps}} \, , \\[5pt]
				& \bs{\psi}^{\varepsilon}_{n} = \bs{\varphi} \ \ \mbox{on} \ \ \partial K_{n}^{\eps}  \, , \\[5pt]
				& \bs{\psi}^{\varepsilon}_{n} \to \bs{0} \ \ \mbox{as} \ \ | \bs{\xi} | \to + \infty \, ,
			\end{aligned}
			\right.
		\end{equation}
		whose existence can be guaranteed by standard methods such as \cite[Theorem~V.2.1]{galdi2011introduction} or \cite[Theorem~7.1]{maremonti1999stokes}. Moreover, the velocity component of this generalized solution can be characterized by the property
		\begin{equation} \label{weakexteenergy}
		\int\limits_{\R^3 \setminus \overline{K_{n}^{\eps}}} | \nabla \bs{\psi}^{\varepsilon}_{n} |^{2} = \inf_{\bs{v} \in D^{1,2}(\R^3 \setminus \overline{K_{n}^{\eps}}; \R^{3})} \left\{ \, \int\limits_{\R^3 \setminus \overline{K_{n}^{\eps}}} | \nabla \bs{v} |^{2} \ \Bigg| \ \nabla \cdot \bs{v} = 0 \ \ \text{in} \ \ \R^3 \setminus \overline{K_{n}^{\eps}} \, ; \ \ \bs{v} = \bs{\varphi} \ \ \mbox{on} \ \ \partial K_{n}^{\eps} \, \right\} \, ,
		\end{equation}
		see, for example, \cite[Theorem~2]{hillairet2019effect}. In the next result we provide energy bounds and decay estimates for the velocity field $\bs{\psi}^{\varepsilon}_{n}$ which are \textit{independent} of the shape of the particle $K_{n}^{\eps}$ (that is, independent of $n \in \{1,\ldots,N(\varepsilon)\}$).
		
		\begin{lemma}\label{lemma_Abis}
		Let $\Omega \subset \mathbb{R}^{3}$ be an admissible domain. Given $\varepsilon \in I_{*}$, $n \in \{1,\ldots,N(\varepsilon)\}$ and a vector field $\bs{\varphi} \in W^{1,2}(\Omega; \mathbb{R}^{3})$, let $(\bs{\psi}^{\varepsilon}_{n}, p^{\varepsilon}_{n}) \in D^{1,2}(\R^3 \setminus \overline{K_{n}^{\eps}}; \R^{3}) \times L^{2}(\R^3 \setminus \overline{K_{n}^{\eps}}; \R)$ be the unique weak solution to the exterior problem \eqref{extstokes00}. There holds the energy bound
			\begin{equation} \label{energy_stokes_1}
				\|\nabla \bs{\psi}^{\varepsilon}_{n}\|_{L^2(\R^3\setminus \overline{K_{n}^{\eps}})} \leq C_{*} \left( \dfrac{1}{\eps^{3}} \|\bs  \varphi\|_{L^2(B(\bs{\xi}_{n}^{\eps},\delta_0 \eps^3))} + \|\nabla \bs \varphi\|_{L^2(B(\bs{\xi}_{n}^{\eps},\delta_0 \eps^3))} \right) \, , 
			\end{equation}
		as well as the decay estimate
			\begin{equation}\label{decay_stokes_1}
				|\bs{\psi}^{\varepsilon}_{n}(\bs{\xi})| +|\bs{\xi}-\bs{\xi}_{n}^{\eps}| \left( |\nabla \bs{\psi}^{\varepsilon}_{n}(\bs{\xi})| + |p^{\varepsilon}_{n}(\bs{\xi})| \right) \leq \dfrac{C_{*} \, \eps^{3/2}}{|\bs{\xi}-\bs{\xi}_{n}^{\eps}|} \left( \dfrac{1}{\eps^{3}}\| \bs \varphi\|_{L^2(B(\bs{\xi}_{n}^{\eps},\delta_0 \eps^3))} + \|\nabla \bs \varphi\|_{ L^2(B(\bs{\xi}_{n}^{\eps},\delta_0 \eps^3))} \right) \, ,
			\end{equation}
		for every $\bs{\xi} \in \mathbb{R}^{3} \setminus \overline{B(\bs{\xi}_{n}^{\eps},2\delta_0 \eps^3)}$, for some constant $C_{*} > 0$ that depends exclusively on $\delta_{0}$ (independently of $\varepsilon \in I_{*}$ and $n \in \{1,\ldots,N(\varepsilon)\}$).
		\end{lemma}
		\noindent
		\begin{proof}
		Let $\varepsilon \in I_{*}$ and $n \in \{1,\ldots,N(\varepsilon)\}$. Recalling \eqref{perforationinitial} and setting
		$$
		\bs{\psi}_{n} (\bs{\xi}) \doteq \bs{\psi}^{\varepsilon}_{n} (\bs{\xi}^{\varepsilon}_{n} + \varepsilon^{3} \bs{\xi}) \qquad \text{and} \qquad p_{n} (\bs{\xi}) \doteq \varepsilon^{3} p^{\varepsilon}_{n}(\bs{\xi}^{\varepsilon}_{n} + \varepsilon^{3} \bs{\xi}) \qquad \forall \bs{\xi} \in \R^3 \setminus \overline{K_n} \, ,
		$$
		it can be easily seen that $(\bs{\psi}_{n}, p_{n}) \in D^{1,2}(\R^3 \setminus \overline{K_n}; \R^{3}) \times L^{2}(\R^3 \setminus \overline{K_n}; \R)$ is the unique weak solution to the following Stokes problem in the exterior of $K_n$:
		\begin{equation}\label{extstokes11}
			\left\{
			\begin{aligned}
				& -\Delta \bs \psi_n + \nabla p_n  = \bs{0} \, , \quad  \nabla\cdot \bs \psi_n  = 0 \ \ \mbox{in} \ \ \R^3 \setminus \overline{K_{n}} \, , \\[5pt]
				& \bs \psi_n = \bs{\varphi}_{n} \ \ \mbox{on} \ \ \partial K_n  \, , \\[5pt]
				& \bs \psi_n \to \bs{0} \ \ \mbox{as} \ \ | \bs{\xi} | \to + \infty \, ,
			\end{aligned}
			\right.
		\end{equation}
		where $\bs{\varphi}_{n} \in W^{1,2}(B(\bs{0}, \delta_{2}); \mathbb{R}^{3})$ is given by
		$$
		\bs{\varphi}_{n} (\bs{\xi}) \doteq \bs{\varphi} (\bs{\xi}^{\varepsilon}_{n} + \varepsilon^{3} \bs{\xi}) \qquad \forall \bs{\xi} \in B(\bs{0}, \delta_{2}) \, .
		$$
		In particular, by density, there holds
		\begin{equation}\label{weakextstokes11}
		\int\limits_{\R^3 \setminus \overline{K_{n}}} \nabla \bs{\psi}_{n} : \nabla \bs{\phi} - \int\limits_{\R^3 \setminus \overline{K_{n}}} p_{n} (\nabla \cdot \bs{\phi}) = 0 \qquad \forall \bs{\phi} \in D^{1,2}_{0}(\R^3 \setminus \overline{K_{n}}; \mathbb{R}^{3}) \, .
		\end{equation}
		Accordingly, the characterization \eqref{weakexteenergy} reads
			\begin{equation} \label{weakexteenergy2}
			\int\limits_{\R^3 \setminus \overline{K_{n}}} | \nabla \bs{\psi}_{n} |^{2} = \inf_{\bs{v} \in D^{1,2}(\R^3 \setminus \overline{K_{n}}; \R^{3})} \left\{ \, \int\limits_{\R^3 \setminus \overline{K_{n}}} | \nabla \bs{v} |^{2} \ \Bigg| \ \nabla \cdot \bs{v} = 0 \ \ \text{in} \ \ \R^3 \setminus \overline{K_{n}} \, ; \ \ \bs{v} = \bs{\varphi}_{n} \ \ \mbox{on} \ \ \partial K_{n} \, \right\} \, .
		\end{equation}
		
	 In what follows, $C > 0$ will always denote a generic constant that depends exclusively on $\delta_{0}$ (independently of $\varepsilon \in I_{*}$ and $n \in \{1,\ldots,N(\varepsilon)\}$), but that may change from line to line.
	 
	 Define the vector field $\bs{T}_{n} \in \mathcal{C}^{\infty}(\mathbb{R}^{3} \setminus \{\bs{0}\}; \mathbb{R}^{3}) \cap W^{1,2}(\R^3 \setminus \overline{B(\bs{0}, \delta_{0})}; \mathbb{R}^{3})$ by
	 \begin{equation} \label{gravfield1}
	 \bs{T}_{n}(\bs{\xi}) \doteq - \dfrac{1}{4 \pi} \left( \int_{B(\bs{0}, \delta_{0}) \setminus \overline{K_{n}}} \nabla \cdot \bs{\varphi}_{n} \right) \dfrac{\bs{\xi}}{| \bs{\xi} |^{3}} \qquad \forall \bs{\xi} \in \mathbb{R}^{3} \setminus \{\bs{0}\} \, .
	\end{equation}
	 Since $\bs{T}_{n}$ is divergence-free in $\mathbb{R}^{3} \setminus \{\bs{0}\}$, a direct computation shows that
	 \begin{equation} \label{gravfield2}
	 \int\limits_{\partial B(\bs{0}, \delta_{0})} \bs{T}_{n} \cdot \widehat{\bs{\nu}} = - \int\limits_{B(\bs{0}, \delta_{0}) \setminus \overline{K_{n}}} \nabla \cdot \bs{\varphi}_{n} \, ; \qquad \|\bs T_{n} \|_{W^{1,2}(\R^3 \setminus \overline{B(\bs{0}, \delta_{0})})} \leq C \|\nabla \bs{\varphi}_{n} \|_{L^{2}(B(\bs{0}, \delta_{0}))} \, ,
	 \end{equation}
	 with $\widehat{\bs{\nu}} \in \mathbb{R}^{3}$ being the outward unit normal to $B(\bs{0}, \delta_{0})$.  Consequently, owing to \cite[Exercise III.3.5]{galdi2011introduction}, consider a vector field $\bs{w}_{n} \in W^{1,2}(B(\bs{0}, \delta_{0}) \setminus \overline{K_{n}}; \R^{3})$ satisfying
	 \begin{equation} \label{divext01}
	 	\left\{
	 	\begin{aligned}
	 		& \nabla \cdot \bs{w}_{n} = - \nabla \cdot \bs{\varphi}_{n} \ \ \mbox{in} \ \ B(\bs{0}, \delta_{0}) \setminus \overline{K_{n}} \, ; \quad \bs{w}_{n} = \bs{0} \ \ \mbox{on} \ \ \partial K_{n}  \, ; \quad \bs{w}_{n} = \bs{T}_{n} \ \ \mbox{on} \ \ \partial B(\bs{0}, \delta_{0})  \, ; \\[6pt]
	 		& \|\nabla \bs w_{n} \|_{L^2(B(\bs{0}, \delta_{0}) \setminus \overline{K_{n}})} \leq C \left( \| \nabla \cdot \bs{\varphi}_{n} \|_{L^2(B(\bs{0}, \delta_{0}) \setminus \overline{K_{n}})} + \|\bs{T}_{n} \|_{W^{1/2,2}(\partial B(\bs{0}, \delta_{0}))} \right) \, .
	 	\end{aligned} 
	 	\right.
	 \end{equation}
	 Recall that our assumption on the uniform John character of the sets $(B(\bs{0}, \delta_{0}) \setminus \overline{K_{n}})_{n \in \mathbb{N}}$ ensures that the constant entering the right-hand side of \eqref{divext01}$_2$ can be bounded independently of $n \in \{1,\ldots,N(\varepsilon)\}$. In view of \cite[Exercise III.3.8]{galdi2011introduction}, take now a vector field $\bs{v}_{n} \in D^{1,2}(\R^3 \setminus \overline{B(\bs{0}, \delta_{0})}; \R^{3})$ verifying
		\begin{equation} \label{divext1}
		\left\{
		\begin{aligned}
		& \nabla \cdot \bs{v}_{n} = 0 \ \ \mbox{in} \ \ \R^3 \setminus \overline{B(\bs{0}, \delta_{0})} \, ; \qquad \bs{v}_{n} = \bs{w}_{n} + \bs{\varphi}_{n} \ \ \mbox{on} \ \ \partial B(\bs{0}, \delta_{0})  \, ; \\[6pt]
		& \|\nabla \bs v_{n} \|_{L^2(\R^3 \setminus \overline{B(\bs{0}, \delta_{0})})} \leq C \| \bs{w}_{n} + \bs{\varphi}_{n} \|_{W^{1/2,2}(\partial B(\bs{0}, \delta_{0}))} \, .
		\end{aligned} 
		\right.
	    \end{equation}
	    Extending $\bs{v}_{n} = \bs{w}_{n} + \bs{\varphi}_{n}$ inside $B(\bs{0}, \delta_{0}) \setminus \overline{K_{n}}$ (keeping the same notation), we obtain a divergence-free vector field $\bs{v}_{n} \in D^{1,2}(\R^3 \setminus \overline{K_{n}}; \R^{3})$ such that $\bs{v}_{n} = \bs{\varphi}_{n}$ on $\partial K_{n}$. Moreover, owing to \eqref{gravfield2}$_2$-\eqref{divext1}$_2$ and the trace inequality (applied in $B(\bs{0}, \delta_{0})$), there holds the estimate
	     \begin{equation} \label{divext2}
	     \begin{aligned}
	     \|\nabla \bs v_{n} \|^{2}_{L^2(\R^3 \setminus \overline{K_{n}})} & = \|\nabla \bs v_{n} \|^{2}_{L^2(\R^3 \setminus \overline{B(\bs{0}, \delta_{0})})} + \|\nabla \bs v_{n} \|^{2}_{L^2(B(\bs{0}, \delta_{0}) \setminus \overline{K_{n}})} \\[6pt]
	     & \leq C \left( \|\bs{\varphi}_{n} \|_{W^{1/2,2}(\partial B(\bs{0}, \delta_{0}))}^{2} +\|\bs{w}_{n} \|_{W^{1/2,2}(\partial B(\bs{0}, \delta_{0}))}^{2}  + \|\nabla \bs{\varphi}_{n}  \|^{2}_{L^2(B(\bs{0}, \delta_{0}))} \right) \\[6pt]
	     & \leq C \left( \|\bs{\varphi}_{n} \|_{L^{2}(B(\bs{0}, \delta_{0}))}^{2} + \|\nabla \bs{\varphi}_{n}  \|^{2}_{L^2(B(\bs{0}, \delta_{0}))} \right) \, .
	     \end{aligned}
	     \end{equation}
	     Observing that $\bs{v}_{n}$ is an admissible candidate in the characterization \eqref{weakexteenergy2}, from \eqref{divext2} we infer
	     \begin{equation} \label{divext3}
	     	\begin{aligned}
	     	\|\nabla \bs{\psi}_{n} \|^{2}_{L^2(\R^3 \setminus \overline{K_{n}})} \leq \|\nabla \bs v_{n} \|^{2}_{L^2(\R^3 \setminus \overline{K_{n}})} \leq C \left( \|\bs{\varphi}_{n} \|_{L^{2}(B(\bs{0}, \delta_{0}))}^{2} + \|\nabla \bs{\varphi}_{n}  \|^{2}_{L^2(B(\bs{0}, \delta_{0}))} \right) \, .
	     	\end{aligned}
	     \end{equation}
	    
	Applying the changes of variables
	     \begin{equation} \label{changevar}
	     \bs{\xi} \in \R^3 \setminus \overline{K_{n}} \ \ \longleftrightarrow \ \ \bs{\xi}^{\varepsilon}_{n} + \varepsilon^{3} \bs{\xi} \in \R^3 \setminus \overline{K^{\eps}_{n}} \, ; \qquad \bs{\xi} \in B(\bs{0}, \delta_{0}) \ \ \longleftrightarrow \ \ \bs{\xi}^{\varepsilon}_{n} + \varepsilon^{3} \bs{\xi} \in B(\bs{\xi}_{n}^{\eps},\delta_0 \eps^3) \, ,
	    \end{equation}
	    we easily deduce
	     \begin{equation} \label{divext4}
	     \begin{aligned}
	     & \|\nabla \bs{\psi}_{n} \|^{2}_{L^2(\R^3 \setminus \overline{K_{n}})} =	\dfrac{1}{\eps^{3}}	\|\nabla \bs{\psi}^{\eps}_{n} \|^{2}_{L^2(\R^3 \setminus \overline{K^{\eps}_{n}})} \, , \\[6pt]
	     & \|\bs{\varphi}_{n} \|_{L^{2}(B(\bs{0}, \delta_{0}))}^{2} = \dfrac{1}{\eps^{9}} \|\bs{\varphi} \|_{L^{2}(B(\bs{\xi}_{n}^{\eps},\delta_0 \eps^3))}^{2} \, , \qquad \|\nabla \bs{\varphi}_{n} \|_{L^{2}(B(\bs{0}, \delta_{0}))}^{2} = \dfrac{1}{\eps^{3}} \|\nabla \bs{\varphi} \|^{2}_{L^2(B(\bs{\xi}_{n}^{\eps},\delta_0 \eps^3))}  \, .
	     	\end{aligned}
	     \end{equation}
	    Upon insertion of \eqref{divext4} into \eqref{divext3}, the energy bound \eqref{energy_stokes_1} is obtained.
	      
		We then turn to the proof of \eqref{decay_stokes_1} which, even though is analogous to the one of \cite[Lemma~5]{hillairet2019effect}, is provided here for the sake of completeness. By interior elliptic regularity (see \cite[Theorem~IV.4.3]{galdi2011introduction}) we deduce $(\bs{\psi}_{n}, p_{n}) \in \mathcal{C}^{\infty}(\R^3 \setminus B(\bs{0}, \delta_{0} ); \R^{3}) \times \mathcal{C}^{\infty}(\R^3 \setminus B(\bs{0}, \delta_{0} ); \R)$. Moreover, \cite[Lemma~1.1]{galdi1990existence} ensures the existence of a constant vector $\bs{\psi}_{0} \in \R^3$ such that $\bs{\psi}_{n} - \bs{\psi}_{0} \in L^{6}(\R^3 \setminus \overline{K_n}; \R^{3})$ and
		\begin{equation} \label{simader}
	    \| \bs{\psi}_{n} - \bs{\psi}_{0}  \|_{L^6(\R^3 \setminus \overline{K_{n}})} \leq C \|\nabla \bs{\psi}_{n} \|_{L^2(\R^3 \setminus \overline{K_{n}})} \, ; \qquad \lim\limits_{r \to +\infty} \int\limits_{\partial B(\bs{0},1)} | \bs{\psi}_{n}(r \bs{\xi}) - \bs{\psi}_{0} |^{2} = 0 \, .
		\end{equation}
		Nevertheless, recalling, from \cite[Definition V.1.1]{galdi2011introduction} that
		\begin{equation} \label{simader2}
		\lim\limits_{r \to +\infty} \int\limits_{\partial B(\bs{0},1)} | \bs{\psi}_{n}(r\bs{\xi}) | = 0 \, ,
		\end{equation}
		we can combine \eqref{simader}-\eqref{simader2} with the H\"older inequality to deduce that $\bs{\psi}_{0} = \bs{0}$.
		 
		In view of assumption \eqref{perforation}$_1$, we can choose $\varrho \in (0,\delta_{0})$ (depending exclusively on $\delta_{0}$) such that
		$$
		\overline{K_{n}} \subset B(\bs{0}, \varrho) \subset B(\bs{0}, \delta_{0} )  \qquad \forall n \in \{1,...,N(\varepsilon)\} \, .
		$$
		Consider, then, the annular regions
		$$
		A(\delta_{0}) \doteq B \! \left( \bs{0}, \dfrac{3}{2}\delta_{0} \right) \setminus \overline{B(\bs{0}, \delta_{0})} \qquad \text{and} \qquad E(\delta_{0}) \doteq B( \bs{0}, 2\delta_{0}) \setminus \overline{B(\bs{0}, \varrho)} \, ,
		$$
		so that $\overline{	A(\delta_{0})} \subset E(\delta_{0})$ and $\overline{E(\delta_{0})} \subset \R^3 \setminus \overline{K_{n}}$. We may invoke the usual interior regularity estimates for the Stokes equations (see \cite[Theorem~IV.4.1]{galdi2011introduction}) to derive the bound
		\begin{equation} \label{intbound0}
			\| \bs{\psi}_{n} \|_{W^{4,2}(A(\delta_{0}))} +\| p_{n}\|_{W^{3,2}(A(\delta_{0}))} \leq C \left( \| \bs{\psi}_{n} \|_{W^{1,2}(E(\delta_{0}))} + \| p_{n} \|_{L^{2}(E(\delta_{0}))} \right) \, .
		\end{equation}
		A trivial application of the H\"older inequality, as well as \eqref{simader}$_1$ (recall that $\bs{\psi}_{0} = \bs{0}$), entails
		\begin{equation} \label{intbound1}
		\begin{aligned}
		\| \bs{\psi}_{n} \|^{2}_{W^{1,2}(E(\delta_{0}))} & = \| \bs{\psi}_{n} \|^{2}_{L^{2}(E(\delta_{0}))} + \| \nabla \bs{\psi}_{n}  \|^{2}_{L^{2}(E(\delta_{0}))} \\[6pt]
		& \leq C \left( \| \bs{\psi}_{n} \|^{2}_{L^{6}(E(\delta_{0}))} + \| \nabla \bs{\psi}_{n}  \|^{2}_{L^{2}(E(\delta_{0}))} \right) \\[6pt]
		& \leq C \left( \| \bs{\psi}_{n} \|^{2}_{L^{6}(\R^3 \setminus \overline{K_{n}})} + \| \nabla \bs{\psi}_{n}  \|^{2}_{L^{2}(\R^3 \setminus \overline{K_{n}})} \right) \\[6pt]
		& \leq C \| \nabla \bs{\psi}_{n}  \|^{2}_{L^{2}(\R^3 \setminus \overline{K_{n}})} \, .
		\end{aligned}
		\end{equation}
On the other hand, \cite[Theorem~III.3.6]{galdi2011introduction} ensures the existence of $\bs{X}_{\! n} \in D_{0}^{1,2}(\R^3 \setminus \overline{B(\bs{0}, \varrho)}; \R^{3})$ such that
\begin{equation} \label{divext11}
	\nabla \cdot \bs{X}_{\! n} = p_{n} \ \ \mbox{in} \ \ \R^3 \setminus \overline{B(\bs{0}, \varrho)} \, ; \qquad  \|\nabla \bs{X}_{\! n}  \|_{L^2(\R^3 \setminus \overline{B(\bs{0}, \varrho)})} \leq C \|p_{n} \|_{L^{2}(\R^3 \setminus \overline{B(\bs{0}, \varrho)})} \, .
\end{equation}   
After a trivial extension (by zero) to the interior of $B(\bs{0}, \varrho)$, we can take $\bs{\phi} = \bs{X}_{\! n}$ in \eqref{weakextstokes11}; this furnishes, as a consequence of the H\"older inequality and \eqref{divext11}$_2$,
\begin{equation} 
	\|p_{n} \|_{L^{2}(\R^3 \setminus \overline{B(\bs{0}, \varrho)})} \leq C \|\nabla \bs{\psi}_{n} \|_{L^2(\R^3 \setminus \overline{B(\bs{0}, \varrho)})}  \, ,
\end{equation} 	  		
and thus,
\begin{equation} \label{divext12}
\| p_{n} \|_{L^{2}(E(\delta_{0}))} \leq	\|p_{n} \|_{L^{2}(\R^3 \setminus \overline{B(\bs{0}, \varrho)})} \leq C \|\nabla \bs{\psi}_{n} \|_{L^2(\R^3 \setminus \overline{B(\bs{0}, \varrho)})} \leq C \|\nabla \bs{\psi}_{n} \|_{L^2(\R^3 \setminus \overline{K_{n}})}  \, .
\end{equation} 
From \eqref{intbound0}-\eqref{intbound1}-\eqref{divext12} and by Sobolev embedding, we have $(\bs{\psi}_{n}, p_{n}) \in \mathcal{C}^{2}(\overline{A(\delta_{0}) }; \R^{3}) \times \mathcal{C}^{1}(\overline{A(\delta_{0}) }; \R)$ together with 
		\begin{equation} \label{intbound2}
			\| \bs{\psi}_{n} \|_{\mathcal{C}^{2}(\overline{A(\delta_{0})})} +\|p_{n}\|_{\mathcal{C}^{1}(\overline{A(\delta_{0})})}  \leq C \| \nabla \bs{\psi}_{n}  \|_{L^{2}(\R^3 \setminus \overline{K_{n}})} \, .
		\end{equation}
		Consider now a cutoff function $\chi_{0} \in \mathcal{C}^{\infty}(\mathbb{R}^{3}, [0,1])$ such that
		\begin{equation} \label{cutoffapp}
		\chi_{0} \equiv 0 \ \ \text{in} \ \ \overline{B(\bs{0}, \delta_{0} )} \qquad \text{and} \qquad \chi_{0} \equiv 1 \ \ \text{in} \ \ \mathbb{R}^{3} \setminus B \! \left( \bs{0}, \dfrac{3}{2}\delta_{0} \right) \, .
		\end{equation}
		Introducing the functions $(\widehat{\bs{\psi}}_{n}, \widehat{p}_{n}, \bs{f}_{n}, g_{n}) \in \mathcal{C}^{\infty}(\R^3; \R^{3}) \times \mathcal{C}^{\infty}(\R^3; \R) \times \mathcal{C}_{0}^{\infty}(\R^3; \R^{3}) \times \mathcal{C}_{0}^{\infty}(\R^3; \R)$ by
		\begin{equation}
		\widehat{\bs{\psi}}_{n} \doteq \chi_{0} \bs{\psi}_{n} \, , \quad \widehat{p}_{n} \doteq \chi_{0} p_{n} \, , \quad \bs{f}_{\! n} \doteq -(\Delta \chi_{0}) \bs{\psi}_{n} - 2(\nabla \chi_{0} \cdot \nabla) \bs{\psi}_{n} + p_{n} \nabla \chi_{0} \, , \quad g_{n} = \nabla \chi_{0} \cdot \bs{\psi}_{n} \quad \text{in} \quad \mathbb{R}^{3} \, ,
		\end{equation}
		we notice that $\text{supp}(\bs{f}_{\! n}) \subset \overline{A(\delta_{0})}$, $\text{supp}(g_{n}) \subset \overline{A(\delta_{0})}$ and, due to \eqref{intbound2},
		\begin{equation} \label{intbound3}
			\| \bs{f}_{\! n} \|_{\mathcal{C}^{2}(\overline{A(\delta_{0})})} +\|g_{n}\|_{\mathcal{C}^{1}(\overline{A(\delta_{0})})} \leq C \left( \| \bs{\psi}_{n} \|_{\mathcal{C}^{2}(\overline{A(\delta_{0})})} +\|p_{n}\|_{\mathcal{C}^{1}(\overline{A(\delta_{0})})} \right) \leq C \| \nabla \bs{\psi}_{n}  \|_{L^{2}(\R^3 \setminus \overline{K_{n}})} \, .
		\end{equation}
		Moreover, a straightforward computation shows that the following non-homogeneous Stokes system, in the whole space $\mathbb{R}^{3}$, is satisfied in the classical sense:
		\begin{equation}\label{wholestokes11}
		-\Delta \widehat{\bs{\psi}}_{n} + \nabla \widehat{p}_{n}  = \bs{f}_{\! n} \, , \quad  \nabla\cdot \widehat{\bs{\psi}}_{n}  = g_{n} \ \ \mbox{in} \ \ \R^3 \, .
		\end{equation}
		The uniqueness statement of \cite[Theorem~IV.2.2]{galdi2011introduction} implies the representation
		$$
		\left\{
		\begin{aligned}
		& \widehat{\bs{\psi}}_{n}(\bs{\xi}) = \nabla(F_{L} \ast g_{n})(\bs{\xi}) + \int_{\R^3} \bs{F}_{\! S}(\bs{\xi} - \bs{y}) \cdot \left[\bs{f}_{\! n}(\bs{y}) - \Delta(\nabla(F_{L} \ast g_{n}))(\bs{y}) \right] d\bs{y} \qquad \forall \bs{\xi} \in \mathbb{R}^{3} \, , \\[6pt]
		& \widehat{p}_{n}(\bs{\xi}) = - \int_{\R^3} \bs{q}_{S}(\bs{\xi} - \bs{y}) \cdot \left[\bs{f}_{\! n}(\bs{y}) - \Delta(\nabla(F_{L} \ast g_{n}))(\bs{y}) \right] d\bs{y} \qquad \forall \bs{\xi} \in \mathbb{R}^{3} \, ,
		\end{aligned}
		\right.
		$$
		where $F_{L} \in \mathcal{C}^{\infty}(\mathbb{R}^{3} \setminus \{\bs{0}\}; \mathbb{R})$ and $(\bs{F}_{\! S}, \bs{q}_{S}) \in \mathcal{C}^{\infty}(\mathbb{R}^{3} \setminus \{\bs{0}\}; \mathbb{R}^{3}) \times \mathcal{C}^{\infty}(\mathbb{R}^{3} \setminus \{\bs{0}\}; \mathbb{R}^{3})$ denote, respectively, the fundamental solutions to the Laplace and Stokes equations in $\mathbb{R}^{3}$, see \cite[Formula (IV.2.4)]{galdi2011introduction}. Since
		$$
		| \bs{\xi} - \bs{y} | \geq \dfrac{| \bs{\xi}|}{8} \qquad \forall \bs{\xi} \in \mathbb{R}^{3} \setminus \overline{B(\bs{0}, 2\delta_{0} )} \, , \ \ \forall \bs{y} \in A(\delta_{0}) \, ,
		$$
		we may replicate the calculations carried out in \cite[Section IV.2]{galdi2011introduction} to derive the decay estimate
		\begin{equation}\label{wholestokes12}
			|\widehat{\bs{\psi}}_{n}(\bs{\xi})| +|\bs{\xi}| \left( |\nabla \widehat{\bs{\psi}}_{n}(\bs{\xi})| + | \widehat{p}_{n}(\bs{\xi}) |  \right) \leq \dfrac{C}{|\bs{\xi}|} \left( \| \bs{f}_{\! n} \|_{\mathcal{C}^{2}(\overline{A(\delta_{0})})} +\|g_{n}\|_{\mathcal{C}^{1}(\overline{A(\delta_{0})})} \right) \quad \forall \bs{\xi} \in \mathbb{R}^{3} \setminus \overline{B(\bs{0}, 2\delta_{0} )} \, .
		\end{equation}
		In the light of \eqref{cutoffapp}-\eqref{intbound3}-\eqref{wholestokes12}, we may then write 
		\begin{equation}\label{wholestokes3}
			|\bs{\psi}_{n}(\bs{\xi})| +|\bs{\xi}| \left( |\nabla \bs{\psi}_{n}(\bs{\xi})| + |p_{n}(\bs{\xi})| \right) \leq \dfrac{C}{|\bs{\xi}|} \| \nabla \bs{\psi}_{n}  \|_{L^{2}(\R^3 \setminus \overline{K_{n}})} \qquad \forall \bs{\xi} \in \mathbb{R}^{3} \setminus \overline{B(\bs{0}, 2\delta_{0} )} \, .
		\end{equation}
		The decay estimate \eqref{decay_stokes_1} then follows from \eqref{energy_stokes_1}-\eqref{divext3}-\eqref{divext4}-\eqref{wholestokes3}. This concludes the proof.
		\end{proof}
		
		Particular attention is devoted to the case when the boundary datum in \eqref{extstokes00}$_2$ is a constant vector. Accordingly, given $\bs{V} \in \R^3$, denote by $(\bs{\psi}_{n}[\bs{V}], p_{n}[\bs{V}]) \in D^{1,2}(\R^3 \setminus \overline{K_{n}}; \R^{3}) \times L^{2}(\R^3 \setminus \overline{K_{n}}; \R)$ the unique generalized solution to the following Stokes problem in the exterior of $K_{n}$:
	\begin{equation}\label{stokeslet0}
	\left\{
	\begin{aligned}
		& -\Delta \bs{\psi}_{n}[\bs{V}] + \nabla p_{n}[\bs{V}]  = \bs{0} \, , \quad  \nabla\cdot \bs{\psi}_{n}[\bs{V}]  = 0 \ \ \mbox{in} \ \ \R^3 \setminus \overline{K_{n}} \, , \\[5pt]
		& \bs{\psi}_{n}[\bs{V}] = \bs{V} \ \ \mbox{on} \ \ \partial K_{n}  \, , \\[5pt]
		& \bs{\psi}_{n}[\bs{V}] \to \bs{0} \ \ \mbox{as} \ \ | \bs{\xi} | \to + \infty \, .
	\end{aligned}
	\right.
\end{equation}
By elliptic regularity (up to the boundary $\partial K_{n}$), as in expressed in \cite[Theorem~IV.5.2]{galdi2011introduction}, we know that 
$$
(\bs{\psi}_{n}[\bs{V}], p_{n}[\bs{V}]) \in \mathcal{C}^{2}(\R^3 \setminus K_{n}; \R^{3}) \times \mathcal{C}^{1}(\R^3 \setminus K_{n}; \R) \, ,
$$
so that the \textit{drag force} exerted by the Stokes flow $(\bs{\psi}_{n}[\bs{V}], p_{n}[\bs{V}])$ over $K_{n}$ can be computed as 
\begin{equation} \label{drag1}
		\bs{F}_{\! n}[\bs{V}] \doteq \int\limits_{\partial K_{n}} \left( \dfrac{\partial \bs{\psi}^{\varepsilon}_{n}[\bs{V}]}{\partial \bs{\nu}} - p_{n}[\bs{V}] \bs{\nu} \right) \, .
\end{equation}
In accordance with previous notation, $\bs{\nu} \in \mathbb{R}^{3}$ is the outward unit normal to $\R^3 \setminus \overline{K_{n}}$, which is directed towards the interior of $K_{n}$. Due to the linearity of the Stokes system \eqref{stokeslet0}, the drag force $\bs{F}_{\! n}[\bs{V}]$ is linearly dependent on $\bs V \in \mathbb{R}^{3}$. This enables us to define the \textit{Stokes resistance matrix} associated to \eqref{stokeslet0} as the representative matrix $\mathcal{G}_{n} \in \R^{3\times3}$ of the linear transformation $\bs{F}_{\! n} : \R^3 \longrightarrow \R^3$:
		\begin{equation}\label{def_resistance_matrix}
			\bs{F}_{\! n}[\bs{V}] = \mathcal{G}_{n} \bs{V} \qquad \forall \bs{V} \in \mathbb{R}^{3} \, .
		\end{equation}
In the case of spherical particles, the explicit form of the Stokes resistance matrix is known since the work of Stokes \cite{stokes1851effect} (see also the contribution by Boggio \cite{boggio1910sul}). In our framework, since $K_{n} \subset B(\bs{0},\delta_0)$ for any $n \in \N$, we are able to provide bounds which are independent of the shape of $K_{n}$.

		\begin{lemma}\label{lemma_A}
		 Let $(K_{n})_{n \in \mathbb{N}}$ be the sequence of domains defined in Section \ref{perfdomassumptions}. For every $n \in \mathbb{N}$ we have
			\begin{equation}\label{energy_stokeslet}
			\bs{F}_{\! n}[\bs{V}] \cdot \bs{W} = \mathcal{G}_{n} \bs V\cdot\bs  W = \int\limits_{\R^3 \setminus \overline{K_{n}}} \nabla \bs{\psi}_{n}[\bs{V}] : \nabla \bs{\psi}_{n}[\bs{W}]  \qquad \forall \bs{V}, \bs{W} \in \mathbb{R}^{3} \, ,
			\end{equation} 
		so that	the Stokes resistance matrix $\mathcal{G}_{n} \in \R^{3\times3}$ is symmetric and positive-definite. Moreover, there holds the bound
			\begin{equation}\label{bound_R_n}
				|\mathcal{G}_{n}| \leq C_{*}  \, ,
			\end{equation} 
		for some constant $C_{*} > 0$ that depends exclusively on $\delta_{0}$ (independent of $n \in  \N$).
		\end{lemma}
		\noindent
		\begin{proof}
Let  $n \in \N$. We set
$$
\bs{\phi}_{n}[\bs{V}]  \doteq\bs{V} - \bs{\psi}_{n}[\bs{V}]  \qquad \text{and} \qquad \pi_{n}[\bs{V}]  \doteq -  p_{n}[\bs{V}] \qquad \forall \bs{\xi} \in \R^3 \setminus \overline{K_n} \, ,
$$
so that $(\bs{\phi}_{n}[\bs{V}], \pi_{n}[\bs{V}]) \in D_{0}^{1,2}(\R^3 \setminus \overline{K_n}; \R^{3}) \times L^{2}(\R^3 \setminus \overline{K_n}; \R)$ is the unique weak solution to the following Stokes problem in the exterior of $K_n$:
\begin{equation}\label{extstokes11drag}
	\left\{
	\begin{aligned}
		& -\Delta \bs \phi_n[\bs{V}] + \nabla \pi_n[\bs{V}]  = \bs{0} \, , \quad  \nabla\cdot \bs \phi_n[\bs{V}]  = 0 \ \ \mbox{in} \ \ \R^3 \setminus \overline{K_{n}} \, , \\[5pt]
		& \bs \phi_n[\bs{V}] = \bs{0} \ \ \mbox{on} \ \ \partial K_n  \, , \\[5pt]
		& \bs \phi_n[\bs{V}] \to \bs{V} \ \ \mbox{as} \ \ | \bs{\xi} | \to + \infty \, .
	\end{aligned}
	\right.
\end{equation}
Denote by $\widehat{\bs{\nu}} \in \mathbb{R}^{3}$ the outward unit normal to $\partial K_{n}$. Given any pair of indexes $i,j \in \{1,2,3\}$, the following formula was proved in \cite[Lemma~2.3.5]{allaire1}:
\begin{equation}  \label{stokesmatrix0}
\int\limits_{\partial K_{n}} \left( \dfrac{\partial \bs{\phi}_{n}[\widehat{\bs{e}}_{i}]}{\partial \widehat{\bs{\nu}}} - \pi_{n}[\widehat{\bs{e}}_{i}] \widehat{\bs{\nu}} \right) \cdot \widehat{\bs{e}}_{j} = \int\limits_{\R^3 \setminus \overline{K_{n}}} \nabla \bs{\phi}_{n}[\widehat{\bs{e}}_{i}] : \nabla \bs{\phi}_{n}[\widehat{\bs{e}}_{j}] \, .
\end{equation}
Hence, 
\begin{equation} \label{stokesmatrix}
	\int\limits_{\partial K_{n}} \left( \dfrac{\partial \bs{\psi}_{n}[\widehat{\bs{e}}_{i}]}{\partial \bs{\nu}} - p_{n}[\widehat{\bs{e}}_{i}] \bs{\nu} \right) \cdot \widehat{\bs{e}}_{j} = \int\limits_{\R^3 \setminus \overline{K_{n}}} \nabla \bs{\psi}^{\varepsilon}_{n}[\widehat{\bs{e}}_{i}] : \nabla \bs{\psi}_{n}[\widehat{\bs{e}}_{j}] \, .
\end{equation}
Identity \eqref{energy_stokeslet} is directly obtained from \eqref{stokesmatrix}, recalling \eqref{def_resistance_matrix}. In particular, in view of \eqref{def_resistance_matrix}-\eqref{energy_stokeslet}, this implies that the matrix $\mathcal{G}_{n}$ is symmetric and positive-definite, since
$$
\left\{
\begin{aligned}
& \mathcal{G}_{n} \widehat{\bs{e}}_{i} \cdot \widehat{\bs{e}}_{j} = \int\limits_{\R^3 \setminus \overline{K_{n}}} \nabla \bs{\psi}_{n}[\widehat{\bs{e}}_{i}] : \nabla \bs{\psi}_{n}[\widehat{\bs{e}}_{j}] = \int\limits_{\R^3 \setminus \overline{K_{n}}} \nabla \bs{\psi}_{n}[\widehat{\bs{e}}_{j}] : \nabla \bs{\psi}_{n}[\widehat{\bs{e}}_{i}]  = \mathcal{G}_{n} \widehat{\bs{e}}_{j} \cdot \widehat{\bs{e}}_{i} \qquad \forall i,j \in \{1,2,3\} \, , \\[8pt]
& \mathcal{G}_{n} \bs{V} \cdot \bs{V} = \bs{F}_{\! n}[\bs{V}] \cdot \bs{V} = \int\limits_{\R^3 \setminus \overline{K_{n}}} | \nabla \bs{\psi}_{n}[\bs{V}] |^{2} > 0 \qquad \forall \bs{V} \in \mathbb{R}^{3} \setminus \{\bs{0}\} \, .
\end{aligned}
\right.
$$
Finally, from \eqref{energy_stokeslet} and the H\"older inequality we deduce, for every $i,j \in \{1,2,3\}$,
$$
| \mathcal{G}_{n} \widehat{\bs{e}}_{i} \cdot \widehat{\bs{e}}_{j} | = \left| \, \int\limits_{\R^3 \setminus \overline{K_{n}}} \nabla \bs{\psi}_{n}[\widehat{\bs{e}}_{i}] : \nabla \bs{\psi}_{n}[\widehat{\bs{e}}_{j}] \right| \leq \|\nabla \bs{\psi}_{n}[\widehat{\bs{e}}_{i}] \|_{L^2(\R^3\setminus \overline{K_{n}})} \|\nabla \bs{\psi}_{n}[\widehat{\bs{e}}_{j}] \|_{L^2(\R^3\setminus \overline{K_{n}})} \, ,
$$
which, in turn, furnishes \eqref{bound_R_n} as a consequence of the energy estimate \eqref{energy_stokes_1}.
\end{proof}

\begin{remark} \label{scalematrix}
For $\bs{V} \in \mathbb{R}^{3}$, let $(\bs{\phi}_{n}^\eps[\bs{V}], \pi_{n}^\eps[\bs{V}]) \in D^{1,2}(\R^3 \setminus \overline{K_{n}^{\eps}}; \R^{3}) \times L^{2}(\R^3 \setminus \overline{K_{n}^{\eps}}; \R)$ be the unique generalized solution to the following Stokes problem in the exterior of $K^{\eps}_{n}$:
		\begin{equation}\label{stokeslet00}
		\left\{
		\begin{aligned}
			& -\Delta \bs{\psi}^\eps_{n}[\bs{V}] + \nabla p_{n}^\eps[\bs{V}]  = \bs{0} \, , \quad  \nabla\cdot \bs{\psi}^\eps_{n}[\bs{V}]  = 0 \ \ \mbox{in} \ \ \R^3 \setminus \overline{K_{n}^\eps} \, , \\[5pt]
			& \bs{\psi}_{n}^\eps[\bs{V}] = \bs{V} \ \ \mbox{on} \ \ \partial K_{n}^\eps  \, , \\[5pt]
			& \bs{\psi}_{n}^\eps[\bs{V}] \to \bs{0} \ \ \mbox{as} \ \ | \bs{\xi} | \to + \infty \, .
		\end{aligned}
		\right.
	\end{equation}
Employing the change of variables \eqref{changevar} from the proof of Lemma \ref{lemma_Abis}, the drag force is given by 
	\begin{equation}
		\bs{F}^{\eps}_{\! n}[\bs{V}] \doteq \int\limits_{\partial K_{n}^\eps} \left( \dfrac{\partial \bs{\psi}^{\varepsilon}_{n}[\bs{V}]}{\partial \bs{\nu}} - p_{n}[\bs{V}] \bs{\nu} \right) = \eps^3 \mathcal G_n \bs{V} \qquad \forall n \in \{1,\ldots,N(\varepsilon)\} \, .
	\end{equation}
\end{remark}

\subsection{Limits of capacity densities}\label{empmeasures}
	The purpose of this section is to give further insight into the capacity densities defined in \eqref{empmeasuresdef},
	as well as their limits (in a suitable sense, described below), providing a proper justification to the assumption \eqref{hyp_cv_measures}. In order to do so, let us introduce the operators
	$\widetilde{\mathcal{G}}_{\eps} \in \mathcal{L}(L^{2}(\Omega; \R^{3\times 3}), \R^{3\times 3})$ and $\widetilde{\mathcal{J}}_{\eps} \in (L^{2}(\Omega; \R^3))^{*}$ according to the formulas
	\begin{equation} \label{empmeasuresdefl2}
		\left\{
		\begin{aligned}
			& \widetilde{\mathcal{G}}_{\eps}(M) \doteq \eps^3  \sum_{n=1}^{N(\eps)}\mathcal{G}_{n}\left( \dfrac{1}{| B(\bs{\xi}_{n}^{\eps},\delta_{1} \eps) |} \, \int\limits_{B(\bs{\xi}_{n}^{\eps},\delta_{1} \eps)} M \right) \qquad \forall M \in L^{2}(\Omega; \R^{3\times 3}) \, ; \\[6pt]
			& \widetilde{\mathcal{J}}_{\eps}(\bs{\varphi}) \doteq \eps^3 \sum_{n=1}^{N(\eps)} \mathcal{G}_{n}\bs{\mu}^{\varepsilon}_{n} \cdot \left( \dfrac{1}{| B(\bs{\xi}_{n}^{\eps},\delta_{1} \eps) |} \, \int\limits_{B(\bs{\xi}_{n}^{\eps},\delta_{1} \eps)} \bs{\varphi} \right) \qquad \forall \bs{\varphi} \in L^{2}(\Omega; \R^3) \, .
		\end{aligned}
		\right.
	\end{equation}
	In the following result we show that, under hypothesis \eqref{kineticuni}, all the sequences of capacity densities presented so far are uniformly bounded with respect to $\eps \in I_{*}$. 
	
	\begin{proposition} \label{emboundsprop1}
		Let $\Omega \subset \mathbb{R}^{3}$ be an admissible domain and $\bs{\mu}^{\varepsilon}_{1},...,\bs{\mu}^{\varepsilon}_{N(\varepsilon)} \in \mathbb{R}^{3}$ satisfying \eqref{kineticuni}. There hold the uniform bounds
		\begin{equation} 
		\left\{
		\begin{aligned}
		& \sup_{\varepsilon \in I_{*}} \left( \| \mathcal{G}_{\eps} \|_{\mathcal{L}(W^{1,\infty}(\Omega; \R^{3\times 3}), \R^{3\times 3})} + \| \widetilde{\mathcal{G}}_{\eps} \|_{\mathcal{L}(L^{2}(\Omega; \R^{3\times 3}), \R^{3\times 3})} \right) \leq C_{*} \, , \\[6pt]
		& \sup_{\varepsilon \in I_{*}} \left( \| \mathcal{J}_{\eps} \|_{(W^{1,\infty}(\Omega; \R^3))^{*}} + \| \widetilde{\mathcal{J}}_{\eps} \|_{(L^{2}(\Omega; \R^3))^{*}} \right) \leq C_{*} \, \mathcal{S}_{*} \, ,
		\end{aligned}
		\right.
		\end{equation}
		for some constant $C_{*} > 0$ that depends on $\Omega$ and $\{ \delta_{i} \}^{2}_{i=0}$, but is independent of $\varepsilon \in I_{*}$.
	\end{proposition}
	\noindent
	\begin{proof}
	In what follows, $C > 0$ will always denote a generic constant that depends on $\Omega$ and $\{ \delta_{i} \}^{2}_{i=0}$ (independently of $\varepsilon \in I_{*}$), but that may change from line to line.
	
	Given any $\bs{\varphi} \in W^{1,\infty}(\Omega; \R^3)$, due to the Cauchy-Schwartz inequality and \eqref{count}-\eqref{kineticuni}-\eqref{bound_R_n}, we have
	\begin{equation}
	\left| \mathcal{J}_{\eps}(\bs{\varphi}) \right| \leq C \eps^{3} \| \bs{\varphi} \|_{L^{\infty}(\Omega)} \sum_{n=1}^{N(\eps)}  |\bs{\mu}^{\varepsilon}_{n} | \leq C \eps^{3} \| \bs{\varphi} \|_{L^{\infty}(\Omega)} \sqrt{N(\eps) \sum_{n=1}^{N(\eps)}  |\bs{\mu}^{\varepsilon}_{n} |^{2}} \leq C \mathcal{S}_{*} \| \bs{\varphi} \|_{L^{\infty}(\Omega)} \, .
	\end{equation}
Similarly, invoking the Jensen inequality, for every $\bs{\varphi} \in L^{2}(\Omega; \R^3)$ we may write
\begin{equation}
	\left| \widetilde{\mathcal{J}}_{\eps}(\bs{\varphi}) \right| \leq C \sum_{n=1}^{N(\eps)}  |\bs{\mu}^{\varepsilon}_{n} | \| \bs{\varphi} \|_{L^{1}(B(\bs{\xi}_{n}^{\eps},\delta_{1} \eps))} \leq C \sqrt{\sum_{n=1}^{N(\eps)}  |\bs{\mu}^{\varepsilon}_{n} |^{2}} \, \sqrt{\eps^{3} \sum_{n=1}^{N(\eps)}  \| \bs{\varphi} \|^{2}_{L^{2}(B(\bs{\xi}_{n}^{\eps},\delta_{1} \eps))}} \leq C \mathcal{S}_{*} \| \bs{\varphi} \|_{L^{2}(\Omega)} \, .
\end{equation}
In a completely analogous fashion we obtain the estimates
\begin{equation}
	\left| \mathcal{G}_{\eps}(M) \right| \leq C\| M \|_{L^{\infty}(\Omega)} \quad \forall M \in W^{1,\infty}(\Omega; \R^{3 \times 3}) \, ; \qquad \left| \widetilde{\mathcal{G}}_{\eps}(M) \right| \leq C \| M \|_{L^{2}(\Omega)} \quad \forall M \in L^{2}(\Omega; \R^{3 \times 3}) \, ,
\end{equation}
and the proof is concluded.
	\end{proof}

Combined with the Riesz Representation Theorem, the result of Proposition~\ref{emboundsprop1} implies the existence of elements 
$$
\mathcal{G} \in \mathcal{L}(W^{1,\infty}(\Omega; \R^{3\times 3}), \R^{3\times 3}) \, , \quad \mathcal{J} \in (W^{1,\infty}(\Omega; \R^3))^{*} \, , \quad \widetilde{\mathcal{G}} \in L^2(\Omega;\R^{3\times 3}) \, , \quad \widetilde{\bs{\mathcal{J}}} \in L^2(\Omega;\R^{ 3}) \, ,
$$
such that the following convergences hold as $\eps \to 0^{+}$:
\begin{equation}\label{hyp_cv_measuresapp}
	\left\{
	\begin{aligned}
		& \mathcal{G}_{\eps} \rightarrow \mathcal{G} \ \ \ \text{in the strong operator topology of} \ \mathcal{L}(W^{1,\infty}(\Omega; \R^{3\times 3}),\R^{3\times 3}) \, ; \\[6pt]
		& \mathcal{J}_{\eps} \rightharpoonup \mathcal{J} \ \ \ *-\text{weakly in} \ (W^{1,\infty}(\Omega; \R^3))^{*} \, ; \\[6pt]
		& \widetilde{\mathcal{G}}_{\eps} \rightharpoonup\widetilde{\mathcal{G}} \ \ \ \text{weakly in} \ L^{2}(\Omega; \R^{3 \times 3}) \, ; \qquad \mathcal{J}_{\eps} \rightharpoonup \widetilde{\bs{\mathcal{J}}} \ \ \ \text{weakly in} \ L^{2}(\Omega; \R^3) \, ,
	\end{aligned}
	\right.
\end{equation}
along sub-sequences that are not being relabeled, see \eqref{hyp_cv_measures2}. Now, for every $\eps \in I_{*}$ we have 
	\begin{equation} \label{empmeasuresdefl22}
	\left\{
	\begin{aligned}
		& \widetilde{\mathcal{G}}_{\eps}(M) - \mathcal{G}_{\eps}(M) = \eps^3 \sum_{n=1}^{N(\eps)}\mathcal{G}_{n}\left( \dfrac{1}{| B(\bs{\xi}_{n}^{\eps},\delta_{1} \eps) |} \, \int\limits_{B(\bs{\xi}_{n}^{\eps},\delta_{1} \eps)} (M - M(\bs{\xi}_{n}^{\eps})) \right) \quad \forall M \in W^{1, \infty}(\Omega; \R^{3\times 3}) \, ; \\[6pt]
		& \widetilde{\mathcal{J}}_{\eps}(\bs{\varphi}) - \mathcal{J}_{\eps}(\bs{\varphi}) = \eps^3  \sum_{n=1}^{N(\eps)}\mathcal{G}_{n}\bs{\mu}_{n}^{\eps} \left( \dfrac{1}{| B(\bs{\xi}_{n}^{\eps},\delta_{1} \eps) |} \, \int\limits_{B(\bs{\xi}_{n}^{\eps},\delta_{1} \eps)} (\bs{\varphi} - \bs{\varphi}(\bs{\xi}_{n}^{\eps})) \right) \quad \forall \bs{\varphi} \in W^{1, \infty}(\Omega; \R^3) \, .
	\end{aligned}
	\right.
\end{equation}		
In what follows, $C > 0$ will always denote a generic constant that depends on $\Omega$ and $\{ \delta_{i} \}^{2}_{i=0}$ (independently of $\varepsilon \in I_{*}$), but that may change from line to line.

Given any $\bs{\varphi} \in W^{1, \infty}(\Omega; \R^3)$, there holds
\begin{equation} \label{empimeasures0}
\int\limits_{B(\bs{\xi}_{n}^{\eps},\delta_{1} \eps)} |\bs{\varphi} - \bs{\varphi}(\bs{\xi}_{n}^{\eps}) | \leq C \eps^{4} \| \nabla \bs{\varphi} \|_{L^{\infty}(\Omega)} \qquad \forall n \in \{1,\ldots,N(\eps)\} \, .
\end{equation}
Therefore, arguing as in the proof of Proposition~\ref{emboundsprop1}, the Cauchy-Schwartz inequality and \eqref{count}-\eqref{kineticuni}-\eqref{bound_R_n}-\eqref{empimeasures0} entail
\begin{equation} \label{empimeasures1}
\begin{aligned}
	\left| \eps^3 \sum_{n=1}^{N(\eps)}\mathcal{G}_{n}\bs{\mu}_{n}^{\eps} \left( \dfrac{1}{| B(\bs{\xi}_{n}^{\eps},\delta_{1} \eps) |} \, \int\limits_{B(\bs{\xi}_{n}^{\eps},\delta_{1} \eps)} (\bs{\varphi} - \bs{\varphi}(\bs{\xi}_{n}^{\eps})) \right) \right| \leq C \mathcal{S}_{*} \eps \| \nabla \bs{\varphi} \|_{L^{\infty}(\Omega)} \qquad \forall \eps \in I_{*} \, .
\end{aligned}
\end{equation}		
Similarly one can derive, for every $M \in W^{1,\infty}(\Omega; \R^{3 \times 3})$,
\begin{equation} \label{empimeasures2}
\begin{aligned}
	\left| \eps^3\sum_{n=1}^{N(\eps)} \mathcal{G}_{n}\left( \dfrac{1}{| B(\bs{\xi}_{n}^{\eps},\delta_{1} \eps) |} \, \int\limits_{B(\bs{\xi}_{n}^{\eps},\delta_{1} \eps)} (M - M(\bs{\xi}_{n}^{\eps})) \right) \right| \leq C \eps \| \nabla M \|_{L^{\infty}(\Omega)} \qquad \forall \eps \in I_{*} \, .
\end{aligned}
\end{equation}
In view of \eqref{empmeasuresdefl22}-\eqref{empimeasures1}-\eqref{empimeasures2}, we obtain that, along the sub-sequences in \eqref{hyp_cv_measuresapp}, there holds
	\begin{equation}\label{empmeasures3}
	\left\{
	\begin{aligned}
		& \mathcal{G}(M) = \int_{\Omega} \widetilde{\mathcal{G}} M  \qquad \forall M \in W^{1,\infty}(\Omega; \R^{3\times 3}) \, ; \\[6pt]
		& \mathcal{J}(\bs{\varphi} ) = \int_{\Omega} \widetilde{\bs{\mathcal{J}}} \cdot \bs{\varphi}  \qquad \forall \bs{\varphi} \in W^{1,\infty}(\Omega; \R^{3}) \, .
	\end{aligned}
	\right.
\end{equation}

\begin{remark}
On the one hand, the result of Proposition~\ref{emboundsprop1} ensures that the sequences of capacity densities defined in \eqref{empmeasuresdef} always have convergent sub-sequences in the sense of \eqref{hyp_cv_measuresapp}. Therefore, assumption \eqref{hyp_cv_measures} merely asserts that the whole sequences converge to such limits. On the other hand, \eqref{empmeasures3} guarantees that these limits belong to suitable Lebesgue spaces, thereby justifying \eqref{hyp_cv_measures2} as well. We refer to \cite[Assumptions (A2)-(A3)]{Hillairet2018} and \cite[Equations (17)-(18)]{hillairet2019effect} for related discussions.
\end{remark}
		
We conclude this section by studying the asymptotic behavior of a particular capacity density that plays an important role in the proof of Theorem \ref{pro:corrector}.		

\begin{lemma}
Let $\Omega \subset \mathbb{R}^{3}$ be an admissible domain. Suppose that $(\bs{v}_\eps)_{\varepsilon \in I_{*}} \subset W^{1,2}(\Omega; \mathbb{R}^{3})$ is a sequence that weakly converges to $\bs{v} \in W^{1,2}(\Omega; \mathbb{R}^{3})$, as $\varepsilon \to 0^{+}$, in $W^{1,2}(\Omega; \mathbb{R}^{3})$. Setting
$$ 
A^{\eps}_{n} \doteq B(\bs{\xi}_{n}^{\eps},\delta_1 \eps) \setminus  \overline{B \! \left( \bs{\xi}_{n}^{\eps}, \dfrac{\delta_1 \eps}{2} \right)} \qquad \forall n \in \{1, \ldots, N(\eps) \} \, ,
$$
there holds, under assumption \eqref{hyp_cv_measures},
\begin{equation}\label{W_eps_resistanceapp}
	\lim_{\varepsilon \to 0^{+}} \eps^3 \sum_{n=1}^{N(\eps)}  \mathcal{G}_{n} \bs{\varphi}(\bs{\xi}_{n}^{\eps}) \cdot \left( \dfrac{1}{| A^{\eps}_{n}|} \int_{A^{\eps}_{n}} \bs{v}_\eps \right) = \int_{\Omega} \mathcal{G} \bs{v} \cdot \bs{\varphi} \qquad \forall \bs{\varphi} \in W^{1,\infty}(\Omega; \R^{3}) \, .
\end{equation}
\end{lemma}
\noindent
\begin{proof}
In what follows, $C > 0$ will always denote a generic constant that depends on $\Omega$ and $\{ \delta_{i} \}^{2}_{i=0}$ (independently of $\varepsilon \in I_{*}$), but that may change from line to line.

Define the bilinear application $\mathcal{F}_{\eps} : W^{1,\infty}(\Omega; \R^{3}) \times W^{1,2}(\Omega; \R^{3}) \longrightarrow \mathbb{R}$ by the formula
$$
\mathcal{F}_{\eps}(\bs{\varphi}, \bs{w}) \doteq \eps^3 \sum_{n=1}^{N(\eps)} \mathcal{G}_{n}\bs{\varphi}(\bs{\xi}_{n}^{\eps}) \cdot \left( \dfrac{1}{| A^{\eps}_{n}|} \int_{A^{\eps}_{n}} \bs{w} \right) \qquad \forall (\bs{\varphi}, \bs{w}) \in W^{1,\infty}(\Omega; \R^{3}) \times W^{1,2}(\Omega; \R^{3}) \, .
$$
An application of the Cauchy-Schwartz and Jensen inequalities, as well as \eqref{count}-\eqref{bound_R_n}, provide the following bound, for every $(\bs{\varphi}, \bs{w}) \in W^{1,\infty}(\Omega; \R^{3}) \times W^{1,2}(\Omega; \R^{3})$:
\begin{equation} \label{bilinear1}
| \mathcal{F}_{\eps}(\bs{\varphi}, \bs{w}) | \leq C \eps^{3} \| \bs{\varphi} \|_{L^{\infty}(\Omega)} \sqrt{N(\eps) \sum_{n=1}^{N(\eps)} \dfrac{1}{| A^{\eps}_{n}|} \| \bs{w} \|^{2}_{L^{2}(B(\bs{\xi}_{n}^{\eps},\delta_1 \eps))}} \leq C \| \bs{\varphi} \|_{L^{\infty}(\Omega)}   \| \bs{w} \|_{L^{2}(\Omega)} \, ,
\end{equation}
so that $\mathcal{F}_{\eps}$ is a continuous bilinear form on $W^{1,\infty}(\Omega; \R^{3}) \times W^{1,2}(\Omega; \R^{3})$.

Now, fix any vector field $\bs{\varphi} \in W^{1,\infty}(\Omega; \R^{3})$. By linearity, we have
\begin{equation} \label{bilinear11}
\mathcal{F}_{\eps}(\bs{\varphi}, \bs{v}_{\eps}) = \mathcal{F}_{\eps}(\bs{\varphi}, \bs{v}_\eps - \bs{v}) + \mathcal{F}_{\eps}(\bs{\varphi}, \bs{v}) \qquad \forall \varepsilon \in I_{*} \, .
\end{equation}
By the Rellich-Kondrachov Theorem we then have $\bs{v}_{\eps} \rightarrow \bs{v}$ strongly in $L^{2}(\Omega; \mathbb{R}^{3})$, as $\varepsilon \to 0^{+}$. In view of \eqref{bilinear1}, this implies
\begin{equation} \label{bilinear2}
\lim_{\varepsilon \to 0^{+}} \mathcal{F}_{\eps}(\bs{\varphi}, \bs{v}_\eps - \bs{v}) = 0 \, ,
\end{equation}
so that, due to \eqref{bilinear11}, it remains to prove the following limit:
\begin{equation} \label{bilinear3}
	\lim_{\varepsilon \to 0^{+}} \mathcal{F}_{\eps}(\bs{\varphi}, \bs{v}) = \int_{\Omega} \mathcal{G} \bs{v} \cdot \bs{\varphi} \, .
\end{equation}
Since $\mathcal{F}_{\eps}$ is a continuous bilinear form on $W^{1,\infty}(\Omega; \R^{3}) \times W^{1,2}(\Omega; \R^{3})$ and $\mathcal{G}\in L^2(\Omega; \R^3\times \R^3)$, in order to show \eqref{bilinear3} we can suppose, by density, that $\bs{v} \in \mathcal{C}^{\infty}(\overline{\Omega}; \mathbb{R}^{3})$. Denoting by $\langle \cdot , \cdot \rangle_{\infty}$ the duality product between $(W^{1,\infty}(\Omega; \R^{3}))^{*}$ and $W^{1,\infty}(\Omega; \R^{3})$, we have
\begin{equation} \label{bilinear4}
\begin{aligned}
 \mathcal{F}_{\eps}(\bs{\varphi}, \bs{v}) & = \eps^3 \sum_{n=1}^{N(\eps)}\mathcal{G}_{n}\bs{\varphi}(\bs{\xi}_{n}^{\eps}) \cdot \left( \dfrac{1}{| A^{\eps}_{n}|} \int_{A^{\eps}_{n}} (\bs{v} - \bs{v}(\bs{\xi}_{n}^{\eps})) \right) + \eps^3 \sum_{n=1}^{N(\eps)}\mathcal{G}_{n}\bs{v}(\bs{\xi}_{n}^{\eps}) \cdot \bs{\varphi}(\bs{\xi}_{n}^{\eps}) \\[6pt]
 & =\eps^3  \sum_{n=1}^{N(\eps)}\mathcal{G}_{n}\bs{\varphi}(\bs{\xi}_{n}^{\eps}) \cdot \left( \dfrac{1}{| A^{\eps}_{n}|} \int_{A^{\eps}_{n}} (\bs{v} - \bs{v}(\bs{\xi}_{n}^{\eps})) \right) + \left\langle \eps^3 \sum_{n=1}^{N(\eps)} \delta_{\bs{\xi}_{n}^{\eps}} \mathcal{G}_{n}\bs{v}, \bs{\varphi} \right\rangle_{\! \! \! \infty} \, .
\end{aligned}
\end{equation}
Firstly, from the Mean Value Theorem and \eqref{count}-\eqref{bound_R_n} we readily deduce the bound
\begin{equation} \label{bilinear5}
\left|\eps^3  \sum_{n=1}^{N(\eps)}\mathcal{G}_{n}\bs{\varphi}(\bs{\xi}_{n}^{\eps}) \cdot \left( \dfrac{1}{| A^{\eps}_{n}|} \int_{A^{\eps}_{n}} (\bs{v} - \bs{v}(\bs{\xi}_{n}^{\eps})) \right) \right| \leq C \eps \| \bs{\varphi} \|_{L^{\infty}(\Omega)}   \| \nabla \bs{v} \|_{L^{\infty}(\Omega)} \, .
\end{equation}
Secondly, in the light of \eqref{hyp_cv_measures}, we certainly have
\begin{equation} \label{bilinear6}
	\lim_{\varepsilon \to 0^{+}} \left\langle \eps^3 \sum_{n=1}^{N(\eps)} \delta_{\bs{\xi}_{n}^{\eps}} \mathcal{G}_{n}\bs{v}, \bs{\varphi} \right\rangle_{\! \! \! \infty} = \int_{\Omega} \mathcal{G} \bs{v} \cdot \bs{\varphi} \, .
\end{equation}
The limit \eqref{bilinear3} is then obtained from \eqref{bilinear4}-\eqref{bilinear5}-\eqref{bilinear6}. This concludes the proof.
\end{proof}
		
\bigskip
		
\noindent
{\bf Acknowledgements.} The work of Richard H\"ofer and Amina Mecherbet is supported by the  Project SUSPENSIONS, funded by the DFG (\emph{German Research Foundation}) and ANR (\textit{Agence Nationale de la Recherche}, grants DFG-545373458 and ANR-24-CE92-0028, respectively. The research of Gianmarco Sperone is supported by the \textit{Chilean National Agency for Research and Development} (ANID) through the \textit{Fondecyt Iniciación} grant 11250322. 
\par\smallskip
\noindent
{\bf Data availability statement.} Data sharing not applicable to this article as no datasets were generated or analyzed during the current study.
\par\smallskip
\noindent
{\bf Conflict of interest statement}.  The Authors declare that they have no conflict of interest.

		\phantomsection
		\addcontentsline{toc}{section}{References}
		\bibliographystyle{abbrv}
		\bibliography{references}

\begin{thebibliography}{10}

\bibitem{acosta2017divergence}
G.~Acosta and R.~G. Dur{\'a}n.
\newblock {\em Divergence {O}perator and {R}elated {I}nequalities}.
\newblock Springer, 2017.

\bibitem{allaire1}
G.~Allaire.
\newblock Homogenization of the {N}avier-{S}tokes equations in open sets
  perforated with tiny holes {I}. {A}bstract framework, a volume distribution
  of holes.
\newblock {\em Archive for Rational Mechanics and Analysis}, 113:209--259,
  1991.

\bibitem{allaire2}
G.~Allaire.
\newblock Homogenization of the {N}avier-{S}tokes equations in open sets
  perforated with tiny holes {II}. {N}on-critical sizes of the holes for a
  volume distribution and a surface distribution of holes.
\newblock {\em Archive for Rational Mechanics and Analysis}, 113:261--298,
  1991.

\bibitem{allaire3}
G.~Allaire.
\newblock Homogenization of the {N}avier-{S}tokes equations with a slip
  boundary condition.
\newblock {\em Communications on Pure and Applied Mathematics}, 44(6):605--641,
  1991.

\bibitem{amick1977steady}
C.~J. Amick.
\newblock Steady solutions of the {N}avier-{S}tokes equations in unbounded
  channels and pipes.
\newblock {\em Annali della Scuola Normale Superiore di Pisa - Classe di
  Scienze}, 4(3):473--513, 1977.

\bibitem{amick1984existence}
C.~J. Amick.
\newblock Existence of solutions to the nonhomogeneous steady {N}avier-{S}tokes
  equations.
\newblock {\em Indiana University Mathematics Journal}, 33(6):817--830, 1984.

\bibitem{bear2013dynamics}
J.~Bear.
\newblock {\em Dynamics of {F}luids in {P}orous {M}edia}.
\newblock Courier Corporation, 2013.

\bibitem{boggio1910sul}
T.~Boggio.
\newblock Sul moto stazionario lento di una sfera in un liquido viscoso.
\newblock {\em Rendiconti del Circolo Matematico di Palermo (1884-1940)},
  30(1):65--81, 1910.

\bibitem{boyer2012mathematical}
F.~Boyer and P.~Fabrie.
\newblock {\em {M}athematical {T}ools for the {S}tudy of the {I}ncompressible
  {N}avier-{S}tokes {E}quations and {R}elated {M}odels}, volume 183.
\newblock Springer Science \& Business Media, 2012.

\bibitem{bradley1987petroleum}
H.~B. Bradley.
\newblock {\em Petroleum {E}ngineering {H}andbook}.
\newblock Society of Petroleum Engineers, Richardson, Texas, 1987.

\bibitem{bravin2024collective}
M.~Bravin, E.~Feireisl, A.~Roy, and A.~Zarnescu.
\newblock On the collective effect of a large system of heavy particles
  immersed in a {N}ewtonian fluid.
\newblock {\em ArXiv preprint arXiv:2407.08595}, 2024.

\bibitem{brinkman1949calculation}
H.~C. Brinkman.
\newblock A calculation of the viscous force exerted by a flowing fluid on a
  dense swarm of particles.
\newblock {\em Flow, Turbulence and Combustion}, 1(1):27--34, 1949.

\bibitem{CarrapatosoHillairet20}
K.~Carrapatoso and M.~Hillairet.
\newblock On the derivation of a {S}tokes--{B}rinkman problem from {S}tokes
  equations around a random array of moving spheres.
\newblock {\em Communications in Mathematical Physics}, 373(1):265--325, 2020.

\bibitem{cattabriga1961problema}
L.~Cattabriga.
\newblock Su un problema al contorno relativo al sistema di equazioni di
  {S}tokes.
\newblock {\em Rendiconti del Seminario Matematico della Università di
  Padova}, 31:308--340, 1961.

\bibitem{cioranescu2018strange}
D.~Cioranescu and F.~Murat.
\newblock A strange term coming from nowhere.
\newblock In {\em Topics in the Mathematical Modelling of Composite Materials},
  chapter~4, pages 45--93. Springer, 1997.

\bibitem{conca1985application}
C.~Conca.
\newblock On the application of the homogenization theory to a class of
  problems arising in fluid mechanics.
\newblock {\em Journal de Mathématiques Pures et Appliquées}, 64(1):31--75,
  1985.

\bibitem{darcy1856fontaines}
H.~Darcy.
\newblock {\em Les {F}ontaines {P}ubliques de la {V}ille de {D}ijon}.
\newblock Victor Dalmont, Paris, 1856.

\bibitem{dautray1999mathematicalvol2}
R.~Dautray and J.-L. Lions.
\newblock {\em Mathematical {A}nalysis and {N}umerical {M}ethods for {S}cience
  and {T}echnology - {V}olume 2: {F}unctional and {V}ariational {Methods}
  {II}}.
\newblock Springer Science \& Business Media, 1999.

\bibitem{DesvillettesGolseRicci08}
L.~Desvillettes, F.~Golse, and V.~Ricci.
\newblock The mean-field limit for solid particles in a {N}avier-{S}tokes flow.
\newblock {\em Journal of Statistical Physics}, 131(5):941--967, 2008.

\bibitem{diening2017inverse}
L.~Diening, E.~Feireisl, and Y.~Lu.
\newblock The inverse of the divergence operator on perforated domains with
  applications to homogenization problems for the compressible
  {N}avier--{S}tokes system.
\newblock {\em ESAIM: Control, Optimisation and Calculus of Variations},
  23(3):851--868, 2017.

\bibitem{ene1973equations}
H.~I. Ene and E.~S{\'a}nchez-Palencia.
\newblock {\'E}quations et conditions aux limites pour un mod{\`e}le de milieu
  poreux.
\newblock {\em Comptes Rendus de l'Acad{\'e}mie des Sciences. S{\'e}rie A.
  Sciences Math{\'e}matique}, 277:A257--A259, 1973.

\bibitem{farwig1994generalized}
R.~Farwig and H.~Sohr.
\newblock Generalized resolvent estimates for the {S}tokes system in bounded
  and unbounded domains.
\newblock {\em Journal of the Mathematical Society of Japan}, 46(4):607--643,
  1994.

\bibitem{feireisl2021homogenization}
E.~Feireisl, Y.~Lu, and Y.~Sun.
\newblock Homogenization of a non-homogeneous heat conducting fluid.
\newblock {\em Asymptotic Analysis}, 125(3-4):327--346, 2021.

\bibitem{feireisl2016homogenization}
E.~Feireisl, Y.~Namlyeyeva, and {\v{S}}.~Ne{\v{c}}asov{\'a}.
\newblock Homogenization of the evolutionary {N}avier--{S}tokes system.
\newblock {\em Manuscripta Mathematica}, 149:251--274, 2016.

\bibitem{galdi2011introduction}
G.~P. Galdi.
\newblock {\em An {I}ntroduction to the {M}athematical {T}heory of the
  {N}avier-{S}tokes {E}quations: {S}teady-{S}tate {P}roblems}.
\newblock Springer Science \& Business Media, 2011.

\bibitem{galdi2008hemodynamical}
G.~P. Galdi, A.~M. Robertson, R.~Rannacher, and S.~Turek.
\newblock {\em Hemodynamical {F}lows: {M}odeling, {A}nalysis and {S}imulation
  ({O}berwolfach {S}eminars)}.
\newblock Springer Science \& Business Media, 2008.

\bibitem{galdi1990existence}
G.~P. Galdi and C.~G. Simader.
\newblock Existence, uniqueness and ${L}^{q}$-estimates for the {S}tokes
  problem in an exterior domain.
\newblock {\em Archive for Rational Mechanics and Analysis}, 112(4):291--318,
  1990.

\bibitem{Giunti21}
A.~Giunti.
\newblock Derivation of {D}arcy’s law in randomly perforated domains.
\newblock {\em Calculus of Variations and Partial Differential Equations},
  60(5):1--30, 2021.

\bibitem{GiuntiHoefer21}
A.~Giunti and R.~M. H\"ofer.
\newblock Homogenisation for the {Stokes} equations in randomly perforated
  domains under almost minimal assumptions on the size of the holes.
\newblock {\em Annales de l'Institut Henri Poincaré C. Analyse Non Linéaire},
  36(7):1829--1868, 2019.

\bibitem{heywood1996artificial}
J.~G. Heywood, R.~Rannacher, and S.~Turek.
\newblock Artificial boundaries and flux and pressure conditions for the
  incompressible {N}avier--{S}tokes equations.
\newblock {\em International Journal for Numerical Methods in Fluids},
  22(5):325--352, 1996.

\bibitem{Hillairet2018}
M.~Hillairet.
\newblock On the homogenization of the {S}tokes problem in a perforated domain.
\newblock {\em Archive for Rational Mechanics and Analysis}, 230(3):1179--1228,
  2018.

\bibitem{mecherbet2020estimates}
M.~Hillairet and A.~Mecherbet.
\newblock ${L}^{p}$ estimates for the homogenization of {S}tokes problem in a
  perforated domain.
\newblock {\em Journal of the Institute of Mathematics of Jussieu},
  19(1):231--258, 2020.

\bibitem{hillairet2019effect}
M.~Hillairet, A.~Moussa, and F.~Sueur.
\newblock On the effect of polydispersity and rotation on the {B}rinkman force
  induced by a cloud of particles on a viscous incompressible flow.
\newblock {\em Kinetic and Related Models}, 12(4):681--701, 2019.

\bibitem{hofer2023homogenization}
R.~M. H{\"o}fer.
\newblock Homogenization of the {N}avier--{S}tokes equations in perforated
  domains in the inviscid limit.
\newblock {\em Nonlinearity}, 36(11):6020--6047, 2023.

\bibitem{HoeferJansen20}
R.~M. H{\"o}fer and J.~Jansen.
\newblock Convergence rates and fluctuations for the {S}tokes--{B}rinkman
  equations as homogenization limit in perforated domains.
\newblock {\em Archive for Rational Mechanics and Analysis}, 248(3):50, 2024.

\bibitem{HMS25}
R.~M. H{\"o}fer, A.~Mecherbet, and R.~Schubert.
\newblock Derivation of the monokinetic vlasov-stokes equations.
\newblock {\em ArXiv preprint: 2511.14612}, 2025.

\bibitem{hofer2024quantitative}
R.~M. H{\"o}fer, {\v{S}}.~Necasov{\'a}, and F.~Oschmann.
\newblock Quantitative homogenization of the compressible {N}avier–{S}tokes
  equations towards {D}arcy’s law.
\newblock {\em Annales de l'Institut Henri Poincaré C. Analyse Non Linéaire},
  43(3):669--709, 2026.

\bibitem{HoeferSchubert25}
R.~M. H{\"o}fer and R.~Schubert.
\newblock Sedimentation of particles with very small inertia {II}: Derivation,
  {C}auchy problem and hydrodynamic limit of the {V}lasov-{S}tokes equation.
\newblock {\em ArXiv preprint: 2311.01891}, 2023.

\bibitem{hormander1998analysis}
L.~H{\"o}rmander.
\newblock {\em The {A}nalysis of {L}inear {P}artial {D}ifferential {O}perators
  {I}: {D}istribution {T}heory and {F}ourier {A}nalysis}, volume 256.
\newblock Springer-Verlag, 2003.

\bibitem{hornung1996homogenization}
U.~Hornung.
\newblock {\em Homogenization and {P}orous {M}edia}, volume~6 of {\em
  Interdisciplinary Applied Mathematics}.
\newblock Springer Science \& Business Media, 1996.

\bibitem{jing2025unified}
W.~Jing, Y.~Lu, and C.~Prange.
\newblock Unified quantitative analysis of the {S}tokes equations in dilute
  perforated domains via layer potentials.
\newblock {\em Multiscale Modeling \& Simulation}, 23(3):1145--1182, 2025.

\bibitem{kapitanskii1983spaces}
L.~V. Kapitanskii and K.~Pileckas.
\newblock Spaces of solenoidal vector fields in boundary value problems for the
  {N}avier-{S}tokes equations in regions with noncompact boundaries.
\newblock {\em Matematicheskii Institut imeni Steklova Trudy}, 159:5--36, 1983.

\bibitem{korobkov2011bernoulli}
M.~V. Korobkov.
\newblock Bernoulli law under minimal smoothness assumptions.
\newblock {\em Doklady Mathematics}, 83(1):107--110, 2011.

\bibitem{korobkov2015solution}
M.~V. Korobkov, K.~Pileckas, and R.~Russo.
\newblock Solution of {L}eray's problem for stationary {N}avier-{S}tokes
  equations in plane and axially symmetric spatial domains.
\newblock {\em Annals of Mathematics}, 181(2):769--807, 2015.

\bibitem{korobkov2020solvability}
M.~V. Korobkov, K.~Pileckas, and R.~Russo.
\newblock Solvability in a finite pipe of steady-state {N}avier--{S}tokes
  equations with boundary conditions involving {B}ernoulli pressure.
\newblock {\em Calculus of Variations and Partial Differential Equations},
  59(1):1--22, 2020.

\bibitem{korobkov2024steady}
M.~V. Korobkov, K.~Pileckas, and R.~Russo.
\newblock {\em The {S}teady {N}avier-{S}tokes {S}ystem: {B}asics of the
  {T}heory and the {L}eray {P}roblem}.
\newblock Springer, 2024.

\bibitem{ladyzhenskaya1969mathematical}
O.~A. Ladyzhenskaya.
\newblock {\em The {M}athematical {T}heory of {V}iscous {I}ncompressible
  {F}low}, volume~76.
\newblock Gordon and Breach New York, 1969.

\bibitem{landau}
L.~Landau and E.~Lifshitz.
\newblock {\em Theoretical {P}hysics: {F}luid {M}echanics}, volume~6.
\newblock Pergamon Press, 1987.

\bibitem{pouso2011mean}
R.~L{\'o}pez~Pouso.
\newblock Mean value integral inequalities.
\newblock {\em Real Analysis Exchange}, 37(2):439--449, 2011.

\bibitem{marchenko1974boundary}
V.~A. Marchenko and E.~Y. Khruslov.
\newblock {\em Boundary {V}alue {P}roblems in {D}omains with {F}ine-{G}rained
  {B}oundaries}.
\newblock Naukova Dumka, Academy of Sciences of the Ukrainian SSR, Kiev, 1974.

\bibitem{maremonti1999stokes}
P.~Maremonti, R.~Russo, and G.~Starita.
\newblock On the {S}tokes equations: the boundary value problem.
\newblock In {\em Advances in Fluid Dynamics}, volume~4, pages 69--140. Aracne
  Editrice, 1999.

\bibitem{masmoudi2002homogenization}
N.~Masmoudi.
\newblock Homogenization of the compressible {N}avier--{S}tokes equations in a
  porous medium.
\newblock {\em ESAIM: Control, Optimisation and Calculus of Variations},
  8:885--906, 2002.

\bibitem{maz2013sobolev}
V.~Maz'ya.
\newblock {\em Sobolev Spaces: with Applications to Elliptic Partial
  Differential Equations}, volume 342.
\newblock Springer Science \& Business Media, 2011.

\bibitem{necas2011direct}
J.~Ne{\v{c}}as.
\newblock {\em Direct {M}ethods in the {T}heory of {E}lliptic {E}quations}.
\newblock Springer Science \& Business Media, 2011.

\bibitem{panasenko2024multiscale}
G.~Panasenko and K.~Pileckas.
\newblock {\em Multiscale {A}nalysis of {V}iscous {F}lows in {T}hin {T}ube
  {S}tructures}.
\newblock Springer, 2024.

\bibitem{patriarca2025homogenization}
C.~Patriarca and G.~Sperone.
\newblock Homogenization of {L}eray’s flux problem for the steady-state
  {N}avier--{S}tokes equations in a multiply-connected planar domain.
\newblock {\em Journal of Nonlinear Science}, 35(6):111, 2025.

\bibitem{pileckas1983spaces}
K.~Pileckas.
\newblock Spaces of solenoidal vectors.
\newblock {\em Trudy Matematicheskogo Instituta imeni V. A. Steklova},
  159:137--149, 1983.

\bibitem{pileckas2007navier}
K.~Pileckas.
\newblock {N}avier--{S}tokes system in domains with cylindrical outlets to
  infinity. {L}eray’s problem.
\newblock In {\em Handbook of Mathematical Fluid Dynamics}, volume~4, pages
  445--647. Elsevier Amsterdam, 2007.

\bibitem{sanchez1980non}
E.~S{\'a}nchez-Palencia.
\newblock {\em Non-{H}omogeneous {M}edia and {V}ibration {T}heory}, volume 127
  of {\em Lecture Note in Physics}.
\newblock Springer-Verlag, 1980.

\bibitem{shen2022sharp}
Z.~Shen.
\newblock Sharp convergence rates for {D}arcy’s law.
\newblock {\em Communications in Partial Differential Equations},
  47(6):1098--1123, 2022.

\bibitem{sperone2023homogenization}
G.~Sperone.
\newblock Homogenization of the steady-state {N}avier-{S}tokes equations with
  prescribed flux rate or pressure drop in a perforated pipe.
\newblock {\em Journal of Differential Equations}, 375:653--681, 2023.

\bibitem{sperone2022}
G.~Sperone.
\newblock Steady-state {N}avier--{S}tokes flow in an obstructed pipe under
  mixed boundary conditions and with a prescribed transversal flux rate.
\newblock {\em Calculus of Variations and Partial Differential Equations},
  62(9):236, 2023.

\bibitem{stokes1851effect}
G.~G. Stokes.
\newblock On the effect of the internal friction of fluids on the motion of
  pendulums.
\newblock {\em Transactions of the Cambridge Philosophical Society}, 9:8, 1851.

\bibitem{tartar2006homogeneisation}
L.~Tartar.
\newblock Homog{\'e}n{\'e}isation en hydrodynamique.
\newblock In {\em Singular Perturbations and Boundary Layer Theory: Proceedings
  of the Conference held at the \'{E}cole Centrale de Lyon, December 8--10,
  1976}, pages 474--481. Springer, 1976.

\bibitem{tartar2009general}
L.~Tartar.
\newblock {\em The {G}eneral {T}heory of {H}omogenization: {A} {P}ersonalized
  {I}ntroduction}, volume~7 of {\em Lecture Notes of the Unione Matematica
  Italiana}.
\newblock Springer Science \& Business Media, 2009.

\bibitem{von2004aerodynamics}
T.~von K{\'a}rm{\'a}n.
\newblock {\em Aerodynamics: {S}elected {T}opics in the {L}ight of their
  {H}istorical {D}evelopment}.
\newblock Courier Corporation, 2004.

\end{thebibliography}
		\vspace{5mm}
	
\noindent
\hspace{0.1mm}
\begin{minipage}{140mm}
	\textbf{Richard H\"ofer}\\
	Fakultät für Mathematik\\
	Universität Regensburg\\
	Universitätsstraße 31\\
	93053 Regensburg - Germany\\
	E-mail: richard.hoefer@ur.de
	\vspace{0.5cm}	
\end{minipage}
\newline
\vspace{0.5cm}
\noindent
\begin{minipage}{100mm}
	\textbf{Amina Mecherbet}\\
	Institut de Mathématiques de Jussieu - Paris Rive Gauche\\
	Université Paris Cité\\
	8 Place Aurélie Nemours\\
	75013 Paris - France\\
	E-mail: mecherbet@imj-prg.fr
\end{minipage}
\newline
\vspace{0.5cm}
\begin{minipage}{100mm}
	\textbf{Gianmarco Sperone}\\
	Facultad de Matemáticas\\
	Pontificia Universidad Católica de Chile\\
	Avenida Vicuña Mackenna 4860\\
	7820436 Santiago - Chile\\
	E-mail: gianmarco.sperone@uc.cl
\end{minipage}
	\end{document}